\documentclass[12pt,leqno,a4paper]{amsart}
\usepackage{amssymb,enumerate}
\usepackage{amsaddr}
\usepackage[utf8]{inputenc}
\usepackage[T1]{fontenc} 
\usepackage{float}
\usepackage{slashbox}
\usepackage{tikz}\usetikzlibrary{matrix}\usetikzlibrary{arrows}
\usepackage[textwidth=2.5cm]{todonotes}
\allowdisplaybreaks

\DeclareMathOperator{\Irr}{Irr}
\DeclareMathOperator{\IBr}{IBr}
\DeclareMathOperator{\GL}{GL}
\DeclareMathOperator{\PGL}{PGL}
\DeclareMathOperator{\SL}{SL}
\DeclareMathOperator{\Sp}{Sp}
\DeclareMathOperator{\SO}{SO}
\DeclareMathOperator{\GO}{GO}
\DeclareMathOperator{\CO}{CO}
\DeclareMathOperator{\SU}{SU}
\DeclareMathOperator{\Orth}{O}
\DeclareMathOperator{\Syl}{Syl}
\DeclareMathOperator{\norma}{N}
\DeclareMathOperator{\cent}{C}
\DeclareMathOperator{\Block}{Bl}
\DeclareMathOperator{\block}{bl}
\DeclareMathOperator{\dz}{dz}
\DeclareMathOperator{\Aut}{Aut}
\DeclareMathOperator{\Inn}{Inn}
\DeclareMathOperator{\Out}{Out}
\DeclareMathOperator{\Z}{Z}
\DeclareMathOperator{\Rad}{Rad}
\DeclareMathOperator{\G}{{\bf G}}
\DeclareMathOperator{\T}{{\bf T}}
\DeclareMathOperator{\X}{{\bf X}}
\DeclareMathOperator{\Ind}{Ind}
\DeclareMathOperator{\St}{St}
\DeclareMathOperator{\W}{{\bf W}}
\DeclareMathOperator{\Res}{Res}
\DeclareMathOperator{\Lbf}{{\bf L}}
\DeclareMathOperator{\PSL}{PSL}

\newcommand{\trial}{{^3\hspace{-0.05cm}D_4(q)}}
\newcommand{\GAP}{{\sf GAP}}

\newtheorem{thm}{Theorem}[section]
\newtheorem{lem}[thm]{Lemma}
\newtheorem{conj}{Conjecture}
\newtheorem{coro}[thm]{Corollary}
\newtheorem{prop}[thm]{Proposition}

\newtheorem*{athm}{Theorem A}
\newtheorem*{bthm}{Theorem B}

\theoremstyle{definition}

\newtheorem{defi}[thm]{Definition}

\newtheorem{nota}[thm]{Notation}
\newtheorem{constr}[thm]{Construction}

\theoremstyle{remark}
\newtheorem{rmk}[thm]{Remark}

\begin{document}

\title{The inductive blockwise Alperin weight condition for $G_2(q)$ and $\trial$}

\date{\today}

\author{Elisabeth Schulte}
\email{eschulte@mathematik.uni-kl.de}
\address{FB Mathematik, TU Kaiserslautern, Postfach 3049,
         67653 Kaiserslautern, Germany.}


\begin{abstract}
The inductive blockwise Alperin weight condition is a system of conditions whose verification for all non-abelian finite simple groups would imply the blockwise Alperin weight conjecture. We establish this condition for the groups $G_2(q)$, $q \geqslant 5$, and $\trial$ for all primes dividing their order.
\end{abstract}

\maketitle

\pagestyle{myheadings}
\markboth{}{The iBAW condition for $G_2(q)$ and $\trial$}


\section{Introduction} \label{sec:intro}

The Alperin weight conjecture relates for a prime $\ell$ information about a finite group $G$ (\textit{global}) to properties of $\ell$-local subgroups of $G$, that is, normalizers of $\ell$-subgroups of $G$ (\textit{local}). For this reason it is called a \textit{global-local} conjecture. Here, an $\ell$-{weight} of $G$ is a pair $(R, \psi)$, where $R$ is an $\ell$-subgroup of $G$ and $\psi$ is an irreducible complex character of the normalizer $\norma_G(R)$ containing $R$ in its kernel such that $\psi$ is of $\ell$-defect zero when regarded as a character of $\norma_G(R)/R$. The group $G$ acts on the set of its $\ell$-weights by conjugation. 
Alperin's conjecture asserts the following:

\begin{conj}[Alperin, 1986 {\cite[p.~369]{AlperinWeights}}]
\label{conj:AWC}
Let $G$ be a finite group and $\ell$ a prime. The number of $G$-conjugacy classes of $\ell$-weights of $G$ equals the number of its irreducible Brauer characters defined over characteristic $\ell$.
\end{conj}

Via Brauer block induction each $\ell$-weight of $G$ may be assigned to a unique $\ell$-block of $G$ by calling an $\ell$-weight $(R, \psi)$ a $B$-weight if the $\ell$-block of $\norma_G(R)$ containing $\psi$ induces the $\ell$-block $B$ of $G$. The {blockwise} Alperin weight conjecture then reads as follows:

\begin{conj}[Alperin, 1986 {\cite[p.~371]{AlperinWeights}}]
\label{conj:BAWC}
Let $G$ be a finite group, $\ell$ a prime and $B$ an $\ell$-block of $G$. The
 number of irreducible Brauer characters in $B$ coincides with the number of $G$-conjugacy classes of $B$-weights of $G$.
\end{conj}

Even though Alperin's (blockwise) weight conjecture has been verified in many particular instances, it has not been possible so far to find a general proof for arbitrary finite groups. However, in recent years there has been considerable progress towards a solution for this question, the central focus of this development being a reduction of the original problem to a question on finite (quasi-) simple groups. A reduction theorem for the blockfree version of Alperin's conjecture was obtained by Navarro--Tiep in 2011 \cite{NavarroTiep}. They give an inductive proof showing that if all finite simple groups satisfy the so-called \textit{inductive Alperin weight condition}, then Alperin's weight conjecture holds for any finite group. In 2013, Späth \cite{SpaethBlockwise} refined this result to achieve a reduction theorem for the blockwise version along with a corresponding system of inductive conditions, the \textit{inductive blockwise Alperin weight condition} (iBAW) (cf.~Definition~\ref{defi:iBAWCgroup}).

The iBAW condition for a prime $\ell$ has already been proven to hold for many of the 26 sporadic groups \cite{Breuer}.
 Moreover, Malle gave its verification in \cite{MalleInductiveAbelian} for the simple alternating groups as well as for the Suzuki and Ree groups of types $^2\hspace{-0.05cm}B_2$, $^2G_2$ and $^2\hspace{-0.05cm}F_4$.
  Späth proved it in \cite{SpaethBlockwise} for finite simple groups of Lie type defined over characteristic $\ell$, and together with Koshitani in \cite{SpaethKoshiCyclic, SpaethKoshiCyclic2} for $\ell$-blocks of cyclic defect.
   Apart from some special cases, the problem of establishing the iBAW condition for finite simple groups of Lie type for all primes dividing their order is still open for types other than $^2\hspace{-0.05cm}B_2$, $^2G_2$ and $^2\hspace{-0.05cm}F_4$.
Our results are the following:

\begin{athm}
\label{thm:G2_MainResult}
Let $q \geqslant 5$ be a prime power. The inductive blockwise Alperin weight condition (cf.~Definition~\ref{defi:iBAWCgroup}) holds for the group $G_2(q)$ and every prime $\ell$ dividing its order.
\end{athm}

(For the groups $G_2(3)$ and $G_2(4)$ with exceptional Schur multipliers, Breuer \cite{Breuer} verified the iBAW condition by means of computational methods.)

\begin{bthm}
\label{thm:3D4_MainResult}
Let $q$ be a prime power. The inductive blockwise Alperin weight condition (cf.~Definition~\ref{defi:iBAWCgroup}) holds for the group $\trial$ and every prime $\ell$ dividing its order.
\end{bthm}

This work grew out of the author's dissertation \cite{Diss} prepared at the University of Kaiserslautern under the supervision of Prof.~Dr.~Gunter Malle.


\section{Preliminaries} \label{sec:prelimi}

\subsection{General Notation} \label{ssec:nota}

Let $G$ be a finite group. Concerning the block and character theory of $G$ we follow the notation of \cite{navarro}. We denote by $\Irr(G)$ and $\IBr_\ell(G)$ the sets of irreducible complex characters and of irreducible Brauer characters of $G$ with respect to a prime $\ell$, respectively. For a subgroup $H \leqslant G$ and $\vartheta \in \Irr(H) \cup \IBr_\ell(H)$ we let $\Ind_H^G(\vartheta)$ or $\vartheta^G$ denote the character of $G$ induced by $\vartheta$, and we write $\Res_H^G(\psi)$ or $\psi_{|H}$ for the restriction of $\psi \in \Irr(G) \cup \IBr_\ell(G)$ to $H$. For $\vartheta \in \Irr(H)$ (or $\vartheta \in \IBr_\ell(H)$) we denote by $\Irr(G \mid \vartheta)$ (or $\IBr_\ell(G \mid \vartheta)$) the set of irreducible (Brauer) characters of $G$ lying above $\vartheta$.

The set of $\ell$-blocks of $G$ (with defect group $D$) will be denoted by $\Block_\ell(G)$ (respectively $\Block_\ell(G \mid D)$), and for $B \in \Block_\ell(G)$ we set $\Irr(B):= \Irr(G) \cap B$ and $\IBr(B) := \IBr_\ell(G) \cap B$. For a (Brauer) character $\psi \in \Irr(G) \cup \IBr_\ell(G)$ we denote by $\block(\psi)$ the $\ell$-block of $G$ containing $\psi$, and for an $\ell$-block $b$ of a subgroup of $G$ we denote by $b^G$ the induced $\ell$-block of $G$ if defined.

Group actions play a major role here. If $X$ is a set with a finite group $A$ acting on it, then for $x \in X$, $a \in A$ and $X' \subseteq X$ we denote by 
 $^ax = a.\,x$, $x^a = a^{-1}.\,x$ the images of $a$ applied to $x$ from left and right, respectively,
 $A_x$ the stabilizer of $x$ in $A$,
 $^a\!X'$ the image of $X'$ under $a$,
 $X'^a$ the image of $X'$ under $a^{-1}$,
 $A_{X'}$ the setwise stabilizer of $X'$ in $A$.
If $A_{X'}$ acts on a set $Y$, then we denote by
 $A_{X',y}$ the stabilizer of $y \in Y$ in $A_{X'}$.

If $A$ acts on the finite group $G$ by automorphisms, then there is an induced action of $A$ on the set $\Irr(G)$ of irreducible characters of $G$ via
 $\chi^a(g) = \chi(a.g)$ for $\chi \in \Irr(G)$ and $a \in A$.
Analogously, one obtains an action of $A$ on $\IBr_\ell(G)$ for every prime $\ell$.

From now, $p$ will be a prime number. Moreover, we denote by $\mathbb{F}$ the algebraic closure of a field consisting of $p$ elements, and for an integral power $q$ of $p$ we define $\mathbb{F}_q$ as the unique subfield of $\mathbb{F}$ consisting of $q$ elements.


\subsection{Alperin's Conjecture and its Reduction to Finite Simple Groups}

Here we consider certain aspects of the weight conjecture and its reduction. Throughout, $G$ denotes a finite group and $\ell$ is a prime.

\begin{defi}
\label{def:RadicalSubgroup}
An $\ell$-subgroup $R\leqslant G$ is called a \textit{radical $\ell$-subgroup} of $G$ if it satisfies $R = \mathcal{O}_\ell(\norma_G(R))$. We also say that \textit{$R$ is ($\ell$-) radical in $G$}. The set of radical $\ell$-subgroups of $G$ will be denoted by $\Rad_\ell(G)$.
\end{defi}

\begin{lem}
\label{lem:WeightRadicalSubgroup}
If $(R, \varphi)$ is an $\ell$-weight of $G$, then $R$ is a radical $\ell$-subgroup of $G$.
\end{lem}

\begin{proof}
The $\ell$-block of $\norma_G(R)/R$ containing $\varphi$ is of $\ell$-defect zero and by \cite[Thm.~4.8]{navarro} the $\ell$-core $\mathcal{O}_\ell(\norma_G(R)/R)$ is contained in any defect group of any $\ell$-block of $\norma_G(R)/R$.
\end{proof}

 By \cite[p.~3]{AlperinFong} the $B$-weights of $G$ may be constructed in the following manner:

\begin{constr}
\label{constr:B-weights}
Let $B$ be an $\ell$-block of $G$. For a radical $\ell$-subgroup $R$ of $G$ and an $\ell$-block $b \in \Block_\ell(R \cent_G(R) \mid R)$ with $b^G = B$ we denote by $\theta$ the canonical character of $b$. Then for every $\psi \in \Irr(\norma_G(R)_\theta \mid \theta)$ with 
$$\frac{\psi(1)_\ell}{\theta(1)_\ell} = |\norma_G(R)_\theta \colon R \cent_G(R)|_\ell$$
the pair $$(R,\, \Ind_{\norma_G(R)_\theta}^{\norma_G(R)}(\psi))$$ constitutes a $B$-weight of $G$. As a result of Clifford correspondence (e.g.~\cite[Thm.~6.11]{Isaacs}) distinct characters $\psi$ yield distinct $B$-weights.
Letting $R$ run over a complete set of representatives for the $G$-conjugacy classes of radical $\ell$-subgroups of $G$, and for each such $R$ letting $b$ run over a complete set of representatives for the $\norma_G(R)$-conjugacy classes of $\ell$-blocks $b \in \Block_\ell(R \cent_G(R) \mid R)$ with $b^G = B$ provides all $B$-weights of $G$.
\end{constr}

In general there can be various types of radical $\ell$-subgroups giving rise to $\ell$-weights belonging to an $\ell$-block $B$. However, if $B$ is an $\ell$-block of abelian defect, then the situation is more restrictive. Compare \cite[pp.~24/25]{AnG2Weights} for the following:

\begin{lem}
\label{lem:AbelianDefectWeight}
Let $B$ be an $\ell$-block of $G$ of abelian defect. If $(R, \varphi)$ is a $B$-weight of $G$, then $R$ is a defect group of $B$.
\end{lem}

There are several versions of the iBAW condition. Apart from the original one by Späth \cite[Def.~4.1]{SpaethBlockwise} there is also a version treating only blocks with defect groups involved in certain sets of $\ell$-groups \cite[Def.~5.17]{SpaethBlockwise}, or a version handling single blocks \cite[Def.~3.2]{SpaethKoshiCyclic}.
We shall consider the inductive condition for a single block. For a non-abelian finite simple group this condition may then be verified block by block, this way proving the original inductive condition for the whole group. Note that due to the nature of the groups considered here we restrict the definition of the iBAW condition to the case of simple groups with cyclic outer automorphism group. The definition of the iBAW condition for the general case may be found in \cite[Def.~3.2]{SpaethKoshiCyclic}.

\begin{nota}
Let $G$ be a finite group and $\ell$ be a prime.
\begin{enumerate}[(i)]
\item  If $Q$ is a radical $\ell$-subgroup of $G$ and $B$ an $\ell$-block of $G$, then we define the set
$$\dz(\norma_G(Q)/Q,\, B) := \{ \chi \in \Irr(\norma_G(Q)/Q) \text{ of $\ell$-defect zero} \mid \block(\chi)^G = B \},$$
where $\chi$ is regarded as a character of $\norma_G(Q)$ in the expression $\block(\chi)^G$.
\item By $\Rad_\ell(G)/\!\sim_G$ we denote a complete system of representatives for the $G$-conjugacy classes of radical $\ell$-subgroups of $G$.
\end{enumerate}
\end{nota}

\begin{defi}[{\cite[Def.~3.2]{SpaethKoshiCyclic}, \cite[Lemma~6.1]{SpaethBlockwise}}] 
\label{defi:iBAWCblock}
Let $S$ be a finite non-abelian simple group such that $\Aut(S)/S$ is cyclic and let $X$ be the universal $\ell'$-covering group of $S$. Let $B$ be an $\ell$-block of $X$. We say that the \textit{inductive blockwise Alperin weight (iBAW) condition holds for $B$} if the following conditions are satisfied:
\begin{enumerate}[(i)]
\item There exist subsets $\IBr(B \mid Q) \subseteq \IBr(B)$ for $Q \in \Rad_\ell(X)$ with the following properties:
\begin{enumerate}[(1)]
\item $\IBr(B \mid Q)^a = \IBr(B \mid Q^a)$ for every $Q \in \Rad_\ell(X)$, $a \in \Aut(X)_B$,
\item $\IBr(B) = \dot{\bigcup}_{Q \in \Rad_\ell(X)/\sim_{X}} \IBr(B \mid Q)$.
\end{enumerate}
\item For every $Q \in \Rad_\ell(X)$ there exists a bijection
\begin{align*}
\Omega_Q^X \colon \IBr(B \mid Q) \longrightarrow \dz(\norma_X(Q)/Q,\, B)
\end{align*}
such that $\Omega_Q^X(\phi)^a = \Omega_{Q^a}^X(\phi^a)$ for every $\phi \in \IBr(B \mid Q)$ and $a \in \Aut(X)_B$.
\end{enumerate}
\end{defi}

\begin{defi}[\textit{iBAW condition for $S$ and $\ell$ with $\Aut(S)/S$ cyclic}]
\label{defi:iBAWCgroup}
Let $S$ be a finite non-abelian simple group such that $\Aut(S)/S$ is cyclic and let $X$ be the universal $\ell'$-covering group of $S$. We say that the \textit{inductive blockwise Alperin weight (iBAW) condition holds for $S$ and $\ell$} if the iBAW condition holds for every $\ell$-block of $X$.
\end{defi}

\begin{rmk}
Note that for $S$, $\ell$ and $X$ as above Späth gives a slightly different definition of the iBAW condition for $S$ and $\ell$ in \cite[Def.~4.1]{SpaethBlockwise}. However, the conditions demanded there immediately imply those stated in Definition~\ref{defi:iBAWCgroup}, and vice versa, by \cite[Lemma~3.3]{SpaethKoshiCyclic} the iBAW condition holds for $S$ and $\ell$ in the sense of Späth if it holds for some $\Aut(X)$-transversal in $\Block_\ell(X)$ in the sense of Definition~\ref{defi:iBAWCblock}. Thus, both definitions are equivalent.
\end{rmk}

We call a group $K$ \textit{involved} in a finite group $H$ if there exist subgroups $H_1 \trianglelefteq H_2 \leq H$ such that $K \cong H_2/H_1$.
Due to the following theorem by Späth \cite[Thm.~A]{SpaethBlockwise} there is great interest in verifying the iBAW condition for all finite non-abelian simple groups:

\begin{thm}[Späth]
Let $G$ be a finite group. If all non-abelian simple groups involved in $G$ satisfy the iBAW condition for $\ell$, then Conjecture~\ref{conj:BAWC} holds for every $\ell$-block of $G$.
\end{thm}

\begin{nota}
Let $B$ be an $\ell$-block of a finite group $X$.
\begin{enumerate}[(i)]
\item For a radical $\ell$-subgroup $Q \in \Rad_\ell(X)$ we set
\begin{align*}
\Irr^0(\norma_X(Q)) &:= \{ \psi \in \Irr(\norma_X(Q)) \mid Q \subseteq \ker(\psi),\; \psi(1)_\ell = |\norma_X(Q)/Q\,|_\ell  \},\\
\Irr^0(\norma_X(Q),\, B) &:= \{ \psi \in \Irr^0(\norma_X(Q)) \mid \block(\psi)^X =B \}.
\end{align*}
\item For an $\ell$-weight $(Q, \psi)$ we denote by $[(Q, \psi)]_X$ its $X$-conjugacy class.
\item The set of $X$-conjugacy classes of $B$-weights of $X$ will be denoted by $\mathcal{W}(B)$.
\end{enumerate}
\end{nota}

\begin{lem}
\label{lem:SL3_BijectionConstruction}
Let $X$ be a finite group and suppose that for an $\ell$-block $B$ of $X$ we have a bijection
\begin{align*}
\Omega_B \colon \IBr(B) \longrightarrow \mathcal{W}(B)
\end{align*}
satisfying $\Omega_B(\chi)^a = \Omega_{B}(\chi^a)$ for all $\chi \in \IBr(B)$ and $a \in \Aut(X)_B$. We set 
\begin{align*}
\IBr(B \mid Q) := \bigcup\limits_{\psi \in \Irr^0(\norma_X(Q),\, B)} \{\Omega_B^{-1}([(Q,\psi)]_X)\}
\end{align*}
 for each $Q \in \Rad_\ell(X)$ and define a map
\begin{align*}
\Omega_Q^X \colon \IBr(B \mid Q) &\longrightarrow \dz(\norma_X(Q) / Q,\, B),\quad 
\chi \longmapsto \widetilde{\Omega_B(\chi)},
\end{align*}
where $\widetilde{\Omega_B(\chi)}$ denotes the unique element in $\dz(\norma_X(Q) / Q,\, B)$ whose inflation $\psi$ to $\norma_X(Q)$ satisfies $\Omega_B(\chi) = [(Q, \psi)]_X$. 
Then for all $Q \in \Rad_\ell(X)$ and $a \in \Aut(X)_B$ it holds that
\begin{align*}
\IBr(B \mid Q)^a = \IBr(B \mid Q^a)
\intertext{and we have a disjoint union}
\IBr(B) = \bigcup\limits_{Q \in \Rad_\ell(X)/\sim_{X}} \IBr(B \mid Q),
\end{align*}
so part (i) of Definition~\ref{defi:iBAWCblock} is fulfilled for the $\ell$-block $B$.
Moreover, the map $\Omega_Q^X$ is well-defined, bijective, and satisfies part (ii) of Definition~\ref{defi:iBAWCblock}.
\end{lem}

\begin{proof}
Straightforward.
\end{proof}

For certain cases some important results have already been established: 
Let $S$ be a finite non-abelian simple group and $X$ its universal $\ell'$-covering group. Then:
\begin{enumerate}[(i)]
\item The iBAW condition holds for every $\ell$-block of $X$ that has cyclic defect groups, see {\cite[Thm.~1.1, Thm.~1.3]{SpaethKoshiCyclic}, \cite[Lemma~2.3]{SpaethKoshiCyclic2}}.
\item If $S$ is a group of Lie type defined over a field of characteristic $\ell$, then the iBAW condition holds for $S$ and $\ell$, see {\cite[Thm.~C]{SpaethBlockwise}}.
\end{enumerate}


\section{Steinberg Relations for Universal Chevalley Groups} \label{sec:UCG}

The groups considered here are finite groups of Lie type consisting of fixed points of universal Chevalley groups over $\mathbb{F}$ under certain Steinberg endomorphisms. Each universal Chevalley group has a semisimple complex Lie algebra associated to it, along with a root system $\Sigma$.
As an abstract group, such a group with $\Sigma \neq A_1$ is generated by elements $x_r(t)$, $r \in \Sigma$, $t \in \mathbb{F}^\times$, related to certain automorphisms of the underlying complex Lie algebra with respect to some Chevalley basis and subject to the relations
\begin{align*}
{x}_r(t_1){x}_r(t_2) &= {x}_r(t_1 + t_2) \ \text{for } t_1, t_2 \in K,\; r \in \Sigma,\\
[{x}_{r}(t), {x}_{s}(u)] &= \prod_{i, j} {x}_{i {r} + j {s}}(c_{i j {r} {s}} (-t)^i u^j) \ \text{for linearly independent $r, s \in \Sigma$, $t, u \in K$},\\
{h}_r(t_1){h}_r(t_2) &= {h}_r(t_1 t_2) \ \text{for } t_1, t_2 \in K^\times,\; r \in \Sigma,
\end{align*}
where ${h}_{r}(t) := {n}_{r}(t) {n}_{r}(-1)$ for ${n}_{r}(t) := {x}_{r}(t) {x}_{-{r}}(-t^{-1}) {x}_{r}(t)$, the product ranges over all pairs $i$, $j$ of positive integers such that $i {r} + j {s} \in \Sigma$, and the terms occur in a fixed order independent of $t$ and $u$ with $i+j$ non-decreasing. The scalars $c_{ijrs}$ depend on the chosen Chevalley basis of the underlying Lie algebra.

Additionally, the generators of a universal Chevalley group obey the relations given in Theorem~\ref{thm:ChevalleyRelUni} below, where $(\_,\_)$ denotes an inner product of the Euclidean vector space $\mathbb{R}\Sigma$ and $\langle v, w \rangle := 2 (v, w) / (v,v)$ for $v, w \in \mathbb{R}\Sigma$. Moreover, for $r \in \Sigma$ we denote by $\omega_r$ the reflection along the hyperplane orthogonal to $r$ in $\mathbb{R}\Sigma$. Since $\Sigma$ is a root system, it holds that $\omega_r(s) \in \Sigma$ for all $s \in \Sigma$ (e.g.~\cite[Def.~2.1.1]{CarterSimple}).

\begin{thm}[Steinberg relations for universal Chevalley groups]
\label{thm:ChevalleyRelUni}
Consider a universal Chevalley group over $\mathbb{F}$ with root system $\Sigma \neq A_1$ and generators ${x}_r(t)$, $r \in \Sigma$, $t \in \mathbb{F}$, and ${n}_{r}(t) = {x}_{r}(t) {x}_{-{r}}(-t^{-1}) {x}_{r}(t)$, ${h}_{r}(t) = {n}_{r}(t) {n}_{r}(-1)$ for ${r} \in \Sigma$ and $t \in \mathbb{F}^\times$.
 The following relations hold for all ${r}, {s} \in \Sigma$ and $t, u \in \mathbb{F}$ (with $t\neq 0$ or $u \neq 0$ whenever appropriate):
 \vspace*{0.2cm}
\begin{enumerate}[\quad(i)]\setlength{\itemsep}{10pt}
\item $[{h}_{r}(t), {h}_{s}(u) ] = 1$;
\item if $\{ {r}_1, \ldots, {r}_m \}$ is a base in $\Sigma$ and we set $\check{{r}} := {2{r}}/{({r}, {r})}$ for all ${r} \in \Sigma$, then $${h}_{r}(t) = \prod_{i=1}^m {h}_{{r}_i}(t^{c_i}),$$ where $c_i \in \mathbb{Z}$ are such that $\check{{r}} = \sum_{i=1}^m c_i \check{{r}}_i$;
\item $\prod_{i=1}^m {h}_{{r}_i}(t_i) = 1$ if and only if $t_i=1$ for all $1 \leqslant i \leqslant m$, where $\{ {r}_1, \ldots, {r}_m \}$ is a base in $\Sigma$ as above;
\item  ${h}_{r}(t) {x}_{s}(u) {h}_{r}(t)^{-1} = {x}_{s}(t^{\langle {r}, {s} \rangle}u)$;
\item ${n}_{r}(t) {x}_{s}(u) {n}_{r}(t)^{-1} = {x}_{\omega_{r}({s})}(\eta_{{r},{s}} t^{-\langle {r}, {s} \rangle}u)$ for some sign  $\eta_{{r},{s}} \in \{ \pm 1 \}$;
\item ${n}_{r}(t) {n}_{s}(u) {n}_{r}(t)^{-1} = {n}_{\omega_{r}({s})}(\eta_{{r},{s}} t^{-\langle {r}, {s} \rangle}u)$ with $\eta_{{r},{s}}$ as in (v);
\item  ${n}_{r}(t) {h}_{s}(u) {n}_{r}(t)^{-1} = {h}_{\omega_{r}({s})}(u)$;
\item ${n}_{r}(1)^2 = {h}_{r}(-1)$.
\end{enumerate}
\end{thm}

\begin{proof}
Relations (i) and (iv) to (viii) can be found in the proof of \cite[Thm.~12.1.1]{CarterSimple}. The remaining statements are part of \cite[Thm.~1.12.1]{Classification}.
\end{proof}

\begin{rmk}
\label{rmk:ChevalleyRelations}
The signs $\eta_{r,s}$ in Theorem~\ref{thm:ChevalleyRelUni} are not uniquely determined by the root system $\Sigma$ but rather they also depend on the chosen Chevalley basis of the underlying Lie algebra. However, they are independent of the prime $p$ (cf.~the proof of \cite[Prop.~6.4.3]{CarterSimple}).
\end{rmk}


\section{The Groups $G_2(q)$} \label{sec:G2}

Let us first study the finite Chevalley groups $G_2(q)$ with regard to the iBAW condition. We begin by providing a brief collection of important properties of these groups.


\subsection{Properties of $G_2(q)$}
\label{ssec:PropertiesG2}

We let $q = p^f$ for some natural number $f \in \mathbb{N}_{>0}$ and consider a root system $\Sigma$ of type $G_2$, that is, 
$$\Sigma := \{ \pm a,\; \pm b,\; \pm (a+b),\; \pm (2a+b),\; \pm (3a+b),\; \pm (3a+2b) \} \subset \mathbb{R}^2,$$ where $\Pi:= \{ a, b \}$ is a base for $\Sigma$ (see, e.g., \cite[Rmk.~1.8.8]{Classification}). There are two root lengths in $\Sigma$, with $\pm b$, $\pm (3a+2b)$ and $\pm (3a+b)$ being the long roots of $\Sigma$. 

 Let $\G$ be a universal Chevalley group over $\mathbb{F}$ with root system $\Sigma$ and Steinberg generators $x_r(t)$, $h_r(s)$ and $n_r(s)$, $r \in \Sigma$, $t \in \mathbb{F}$, $s \in \mathbb{F}^\times$. 
 The linear map given by
\begin{align*}
F \colon \G \longrightarrow \G, \quad x_r(t) \longmapsto x_r(t^q),\quad r \in \Sigma,\, t \in \mathbb{F},
\end{align*}
is the Frobenius endomorphism of $\G$ with respect to $\mathbb{F}_q$, hence a Steinberg endomorphism of $\G$, and $G:=G_2(q) := \G^F$. This group is generated by all $x_r(t)$ for $r \in \Sigma$ and $t \in \mathbb{F}_q$, and its order is given by
\begin{align*}
|G_2(q)| &= q^6 \Phi_1(q)^2 \Phi_2(q)^2 \Phi_3(q) \Phi_6(q),
\end{align*}
where $\Phi_i(X) \in \mathbb{Z}[X]$ denotes the $i$-th cyclotomic polynomial (see \cite[Table~24.1]{malletest}).

Following \cite[Table~24.2, Thm.~24.17 and Rmk.~24.19]{malletest} the group $G_2(q)$ is simple for $q \geqslant 3$, and if $q \geqslant 5$, then $G_2(q)$ has trivial Schur multiplier, that is, it is its own universal covering group. Here,
unless stated differently, we assume that $q \geqslant 5$.


\subsubsection{Weyl Group and Maximal Tori of $G_2(q)$}
\label{sssec:WeylTori_G2}

Let $\T$ be the maximal torus of $\G$ generated by all $h_r(t)$, $r \in \Sigma$, $t \in \mathbb{F}^\times$, and denote by $\W := \norma_{\G}(\T)/ \T$ the corresponding Weyl group, a dihedral group of order $12$.

\begin{nota}
Following \cite{ChangConjugate} we introduce a different way of writing the elements of $\T$. 
Set $\xi_1 := a+b$, $\xi_2:= a$ and $\xi_3 := -\xi_1 -\xi_2$, which yields 
\begin{align*}
\Sigma = \{ \pm \xi_1, \pm \xi_2, \pm \xi_3, \pm (\xi_1 - \xi_2), \pm (\xi_2 - \xi_3), \pm (\xi_3 - \xi_1) \}.
\end{align*}
We consider the additive group $\mathbb{Z}\Sigma$. By \cite[p.~98]{CarterSimple} each group homomorphism $\chi \colon \mathbb{Z}\Sigma \rightarrow \mathbb{F}^\times$ (called an $\mathbb{F}$-\textit{character} of $\mathbb{Z}\Sigma$) gives rise to an element $h(\chi)$ of the maximal torus $\T$, and according to \cite[Thm.~7.1.1]{CarterSimple} we have 
\begin{align*}
\T = \{ h(\chi) \mid \chi \colon \mathbb{Z}\Sigma \longrightarrow \mathbb{F}^\times\; \text{is a group homomorphism} \}.
\end{align*}
Since $\mathbb{Z}\Sigma$ is a free abelian group, it is evident that for any fixed basis $B \subseteq \mathbb{Z}\Sigma$ a group homomorphism $\chi \colon \mathbb{Z}\Sigma \longrightarrow \mathbb{F}^\times$ (and hence $h(\chi)$) is uniquely determined by the images of the basis elements under $\chi$. Now let us fix the basis $B = \{ \xi_1, \xi_2 \}$. Then we denote by $h(z_1, z_2, z_3)$ with $z_1, z_2, z_3 \in \mathbb{F}^\times$ the element $h(\chi)$ of $\T$ defined by $\chi(\xi_1) = z_1$, $\chi(\xi_2) = z_2$, and the condition $z_1 z_2 z_3 = 1$.
We have $$h_r(z) = h(z^{\langle r,\; \xi_1 \rangle}, z^{\langle r,\; \xi_2 \rangle}, z^{\langle r,\; \xi_3 \rangle})$$ for $r \in \Sigma$ and $z \in \mathbb{F}^\times$ (see \cite[p.~98]{CarterSimple}).
\end{nota}

The notation introduced above allows a uniform description of the action of $\norma_{\G}(\T)$, and hence of $\W$, on $\T$ as follows, see \cite[p.~193]{ChangConjugate}:

\begin{lem}
\label{lem:WeylGroupAction}
Let $z_1, z_2, z_3 \in \mathbb{F}^\times$ with $z_1 z_2 z_3=1$ and suppose that $i, j, k \in \{ 1, 2, 3 \}$ are pairwise distinct. Then we have
\begin{align*}
n_{\xi_i - \xi_j}(1)^{-1} h(z_1, z_2, z_3) n_{\xi_i - \xi_j}(1) &= h(z_{\pi(1)}, z_{\pi(2)}, z_{\pi(3)}),\; \\
n_{\xi_k}(1)^{-1} h(z_1, z_2, z_3) n_{\xi_k}(1) &= h(z_{\pi(1)}^{-1}, z_{\pi(2)}^{-1}, z_{\pi(3)}^{-1}),
\end{align*}
where $\pi$ denotes the transposition $(i j) \in \mathfrak{S}_3$.
\end{lem}

Up to $G$-conjugation there exist six maximal tori in $G$. Representatives of these in the algebraic group $\G$ are given in Table~\ref{tb:MaximalTorusG2} below (cf.~\cite[p.~194]{ChangConjugate}, \cite[p.~507]{EnomotoConjugacy}, \cite[Table~I]{KleidmanMaxG2}), where we write
$v_2 := n_b(1) n_{-(2a+b)}(1)$,
$v_3 := n_{3a+b}(1) n_{-b}(1)$ and
$v_6 := v_2 v_3$.

\renewcommand{\arraystretch}{1.5}
\begin{table}[H]
\centering
$\begin{array}{|l||l|c|}
\hline 
 w \in W & \ \ \T^{wF} & \W^{wF} \\ 
\hline\hline
 1 &  \begin{aligned} \! T_+ &= \{ h(z_1, z_2, z_3) \mid z_i^{q-1} = 1,\, z_1 z_2 z_3 = 1 \} \cong C_{q-1} \times C_{q-1} \end{aligned}  & D_{12} \\[0.2em] 
\hline
 v_2 T  &  \begin{aligned} \! T_- &= \{ h(z_1, z_2, z_3) \mid z_i^{q+1} = 1,\, z_1 z_2 z_3 = 1 \} \cong C_{q+1} \times C_{q+1} \end{aligned}  & D_{12} \\[0.2em] 
\hline 
n_a(1)T & \begin{aligned} T_{a} &= \{ h(z, z^{q-1}, z^{-q}) \mid z^{q^2-1} = 1 \} \cong C_{q^2-1} \end{aligned}  & C_2 \times C_2 \\[0.2em] 
\hline 
 n_b(1)T & \begin{aligned} T_{b} &= \{ h(z, z^{q}, z^{-(q+1)}) \mid z^{q^2-1} = 1 \} \cong C_{q^2-1} \end{aligned}  & C_2 \times C_2 \\[0.2em]  
\hline 
 v_{3}T & \begin{aligned} T_{3} &= \{ h(z, z^q, z^{q^2}) \mid z^{q^2+q+1} = 1 \} \cong C_{q^2+q+1} \end{aligned}  & C_6 \\[0.2em]  
\hline 
 v_{6}T & \begin{aligned} T_{6} &= \{ h(z, z^{-q}, z^{q^2}) \mid z^{q^2-q+1} = 1 \} \cong C_{q^2-q+1} \end{aligned}  & C_6 \\[0.2em]  
\hline 
\end{array}$
\caption{Maximal tori of $G_2(q)$}
\label{tb:MaximalTorusG2}
\end{table}
\renewcommand{\arraystretch}{1}
\noindent
For future use we set $F_+:=F$ and $F_-:=v_2 F$, and, accordingly, $G_\varepsilon := \G^{F_\varepsilon}$ for $\varepsilon \in \{ \pm \}$ (or $\varepsilon \in \{ \pm 1 \}$ slightly abusing notation). Then $G_+ = G$, and in consequence of Lang--Steinberg's theorem \cite[Thm.~21.7]{malletest} the finite group $G_-$ is $\G$-conjugate to $G$.


\subsubsection{Relations in $G_2(q)$}
\label{sec:RelationsG2}

Following Remark~\ref{rmk:ChevalleyRelations} the signs $\eta_{r,s}$ for roots $r, s \in \Sigma$ occurring in the Steinberg relations given in Theorem~\ref{thm:ChevalleyRelUni} depend on the chosen Chevalley basis underlying $\G$. By \cite[p.~439 and pp.~441/442]{ReeAFamilyof} there exists a Chevalley basis for the simple Lie algebra of type $G_2$ underlying $\G$ with respect to which the signs $\eta_{r, s}$, $r, s \in \Sigma$, are given by Table~\ref{tab:G2_Signs} and the relations $\eta_{r, r} = -1$, $\eta_{r, -s} = \eta_{r, s}$ and $\eta_{-r, s} = \eta_{r, \omega_r(s)}$:
\begin{table}[H]
\centering
\renewcommand{\arraystretch}{1.1}
\begin{tabular}{|c||*{6}{c}|}
\hline
\backslashbox{$r$}{$s$}
& $\phantom{-}a+b$ & $a$ & $-(2a+b)$ & $b$ & $3a+b$ & $-(3a+2b)$  \\
\hline\hline
$\phantom{-}a+b$ &$-1$&$-1$&$\phantom{-}1$&$-1$&$\phantom{-}1$&$\phantom{-}1$\\
$\phantom{-}a$ &$\phantom{-}1$&$-1$&$-1$&$\phantom{-}1$&$-1$&$\phantom{-}1$\\
$-(2a+b)$ &$-1$&$\phantom{-}1$&$-1$&$\phantom{-}1$&$\phantom{-}1$&$-1$\\
$\phantom{-}b$ &$-1$&$\phantom{-}1$&$\phantom{-}1$&$-1$&$\phantom{-}1$&$-1$\\
$\phantom{-}3a+b$ &$\phantom{-}1$&$-1$&$\phantom{-}1$&$-1$&$-1$&$\phantom{-}1$\\
$-(3a+2b)$ &$\phantom{-}1$&$\phantom{-}1$&$-1$&$\phantom{-}1$&$-1$&$-1$\\[0.1em]
\hline
\end{tabular}
\caption{Signs $\eta_{r,s}$ for $G_2(q)$}
\label{tab:G2_Signs}
\end{table}

\renewcommand{\arraystretch}{1}

For consistency in our future calculations we shall henceforth assume that the generators $x_r(t)$ of $\G$ are derived from a Chevalley basis as above. As a first observation we obtain the following relations for the elements $v_2, v_3, v_6 \in \norma_{\G}(\T)$:

\begin{lem}
\label{lem:Omega23Action}
We have $v_2^2 = v_3^3 = v_6^6 = 1$, $[v_2, v_3]=1$, and moreover
\begin{enumerate}[(i)]
\item $v_2^{-1} h(z_1, z_2, z_3) v_2 = h(z_{1}^{-1}, z_{2}^{-1}, z_{3}^{-1})$,
\item $v_3^{-1} h(z_1, z_2, z_3) v_3 = h(z_{3}, z_{1}, z_{2})$
\end{enumerate}
for all $z_1, z_2, z_3 \in \mathbb{F}^\times$ with $z_1 z_2 z_3 = 1$.
\end{lem}

\begin{proof}
The first part of the claim follows from Theorem~\ref{thm:ChevalleyRelUni}(vi) in combination with Table~\ref{tab:G2_Signs}. Statements (i) and (ii) are immediate consequences of Lemma~\ref{lem:WeylGroupAction}.
\end{proof}

Moreover, the following relations hold:

\begin{lem}
\label{lem:Omega2Action}
Let $t \in \mathbb{F}$ (with $t\neq 0$ in (ii)) and $r \in \Sigma$. Then
\begin{enumerate}[(i)]
\item $v_2 x_r(t) v_2^{-1} = x_{-r}(-t)$,
\item $v_2 n_r(t) v_2^{-1} = n_{-r}(-t)$.
\end{enumerate}
In particular, $v_2$ commutes with all $n_r(1)$, $r \in \Sigma$.
\end{lem}

\begin{proof}
Use Theorem~\ref{thm:ChevalleyRelUni} and Table~\ref{tab:G2_Signs}.
\end{proof}

\begin{lem}
\label{lem:TableCartanIntegers}
The Cartan integers $\langle r, s \rangle$, $r, s \in \Sigma$, are given by Table~\ref{tab:G2_CartanIntegers} below, where $\langle -r, s \rangle = \langle r, -s \rangle = - \langle r, s \rangle$ for all $r, s \in \Sigma$.
\begin{table}[H]
\renewcommand{\arraystretch}{1.1}
\centering
\begin{tabular}{|r||*{6}{c}|}
\hline
\backslashbox{$r$}{$s$}
& $a$ & $b$ & $a+b$ & $2a+b$ & $3a+b$ & $3a+2b$  \\
\hline\hline
$a$ &\vphantom{$\Big($}$\phantom{-}2$&$-3$&$-1$&$\phantom{-}1$&$\phantom{-}3$& $\phantom{-}0$\\
$b$ &$-1$&$\phantom{-}2$&$\phantom{-}1$&$\phantom{-}0$&$-1$&$\phantom{-}1$\\
$a+b$ &$-1$&$\phantom{-}3$&$\phantom{-}2$&$\phantom{-}1$&$\phantom{-}0$& $\phantom{-}3$\\
$2a+b$ &$\phantom{-}1$&$\phantom{-}0$&$\phantom{-}1$&$\phantom{-}2$&$\phantom{-}3$& $\phantom{-}3$\\
$3a+b$ &$\phantom{-}1$&$-1$&$\phantom{-}0$&$\phantom{-}1$&$\phantom{-}2$& $\phantom{-}1$\\
$3a+2b$ &$\phantom{-}0$&$\phantom{-}1$&$\phantom{-}1$&$\phantom{-}1$&$\phantom{-}1$& $\phantom{-}2$\\[0.2em]
\hline
\end{tabular}
\renewcommand{\arraystretch}{1}
\caption{Cartan integers for a root system of type $G_2$}\label{tab:G2_CartanIntegers}
\end{table}
\end{lem}

\begin{proof}
Following \cite[p.~40]{CarterSimple} the Cartan integers are given by $\langle r, s \rangle = p(r, s)-q(r, s)$ for $r, s \in \Sigma$ with $r \neq \pm s$, where
\begin{align*}
p(r, s) &= \max\{ i\geqslant 0 \mid -jr + s \in \Sigma \,\text{ for all $0 \leqslant j \leqslant i$}\,\},\\
q(r, s) &= \max\{ i\geqslant 0 \mid \phantom{-}jr + s \in \Sigma \,\text{ for all $0 \leqslant j \leqslant i$}\, \}.
\end{align*}
For $r = s$ we clearly have $\langle r, r \rangle = 2$. Application of these formulae yields the claim.
\end{proof}


\subsubsection{Automorphisms of $G_2(q)$}
\label{sec:AutomorphismsG2}

The aim of this section is a description of the automorphism group of $G$.
By \cite[p.~158]{SteinbergLecture} the field automorphism $\mathbb{F}_q \longrightarrow \mathbb{F}_q$, $a \longmapsto a^p$, induces an automorphism $F_p$ of the group $G = G_2(q)$ via
\begin{align*}
F_p \colon G \longrightarrow G,\quad x_r(t) \longmapsto x_r(t^p), \quad  r \in \Sigma,\; t \in \mathbb{F}_q.
\end{align*}
$F_p$ is called a \textit{field automorphism} of $G$. Its order in $\Aut(G)$ is given by $f$.

In the case that $p=3$ we may define another automorphism of $G$: 
By \cite[p.~156]{SteinbergLecture} there exists a unique angle-preserving and length-changing bijection $\rho \colon \Sigma \longrightarrow \Sigma$ satisfying $\rho(\Delta) = \Delta$. Now $\rho$ induces an automorphism $\Gamma$ of $G_2(q)$ via
\begin{align*}
\Gamma \colon G \longrightarrow G,\quad x_r(t) \longmapsto \begin{cases}
x_{\rho(r)}(\epsilon_r t) & \text{if $r$ is long},\\
x_{\rho(r)}(\epsilon_r t^3) & \text{if $r$ is short},
\end{cases}
\end{align*}
for suitable signs $\epsilon_r \in \{ \pm 1 \}$ with $\epsilon_r=1$ if one of $\pm r$ is contained in $\Delta$ (compare \cite[pp.~156/157]{SteinbergLecture}). This automorphism is called a \textit{graph automorphism} of $G$.

\begin{rmk}
\label{rmk:G2_SignsInGraphAutomorphism}
Assume again that $p=3$. The signs $\epsilon_r$ in the definition of $\Gamma$ depend, like the signs $\eta_{r,s}$, on the chosen Chevalley basis of the Lie algebra underlying $\G$. Since we are working with the fixed Chevalley basis introduced in Section~\ref{sec:RelationsG2}, the question arises of determining the values of the signs $\epsilon_r$ with respect to this particular basis. Denote the elements of this Chevalley basis by $h_r$, $r\in \Pi$, and $e_r$, $r \in \Sigma$, with $[e_r e_{-r}] = h_r$, and for $r, s \in \Sigma$ let $N_{r,s} \in \mathbb{Z}$ be the integer satisfying
$[e_r e_s] = N_{r,s} e_{r+s}$.
Following the proof of \cite[Prop.~12.4.1]{CarterSimple} we may assume that the signs $\epsilon_r$ equal $1$ for all $r \in \Sigma$ if the following condition on the multiplication constants is satisfied:
$$ -\frac{1}{2} N_{\xi_2, \xi_1} = \frac{1}{3} N_{\xi_2, -\xi_3} = N_{\xi_1-\xi_2, \xi_2 - \xi_3}. $$
By \cite[p.~439 and pp.~441/442]{ReeAFamilyof} we have
$
N_{\xi_2, \xi_1} = -2$,
$N_{\xi_2, -\xi_3} = 3$
and $N_{\xi_1-\xi_2, \xi_2 - \xi_3} = 1$,
whence the above condition is clearly fulfilled. Thus, in the following we work with the graph automorphism
\begin{align*}
\Gamma \colon G \longrightarrow G,\quad x_r(t) \longmapsto \begin{cases}
x_{\rho(r)}(t) & \text{if $r$ is long},\\
x_{\rho(r)}(t^3) & \text{if $r$ is short}.
\end{cases}
\end{align*}

Observe that since $\rho^2$ is the identity map on $\Sigma$ and $\rho$ interchanges root lengths, we have $\Gamma^2(x_r(t)) = x_r(t^3)$ for all $r \in \Sigma$, so $\Gamma^2 = F_3 = F_p$.
\end{rmk}

Easy calculations yield the following:

\begin{lem}
\label{lem:GammaAction}
Suppose that $p = 3$. Then for every $r \in \Sigma$ and every $t \in \mathbb{F}_q^\times$ we have
\begin{enumerate}[(i)]
\item $\Gamma(n_r(t)) = n_{\rho(r)}(t^{\lambda_r})$,
\item $\Gamma(h_r(t)) = h_{\rho(r)}(t^{\lambda_r})$,
\end{enumerate}
where $\lambda_r=1$ if $r$ is long, $\lambda_r=3$ if $r$ is short.
\end{lem}

The automorphism group of $G$ may now be described as follows:

\begin{prop}
\label{prop:G2_AutomorphismGroup}
For $G = G_2(q)$ with $q=p^f$ one has
\begin{align*}
\Aut(G) = \begin{cases}
G \rtimes \langle F_p \rangle & \text{if $p \neq 3$},\\
G \rtimes \langle \Gamma \rangle & \text{if $p = 3$}.
\end{cases}
\end{align*}
In particular, the outer automorphism group of $G$ is cyclic for all primes $p$.
\end{prop}

\begin{proof}
This is well-known and follows, e.g., from \cite[Thm.~2.5.12(a)--(e)]{Classification} by taking into account that the algebraic group $\G$ is both adjoint and universal (cf.~\cite[Table~9.2]{malletest}).
\end{proof}

\begin{rmk}
\label{rmk:G2_FieldAuto}
(i) Inside the group $G_- = \G^{v_2F}$ the maximal tori of type $T_-$ have a representative lying in the maximal torus $\T$ of $\G$ (cf.~Table~\ref{tb:MaximalTorusG2}), which allows for a nice description of the action of the field automorphism on this torus. For this we note that as for the group $G_+$ the endomorphism of $\G$ defined by $x_r(t) \mapsto x_r(t^p)$ induces an automorphism of $G_-$. Similarly as for $G_+$ we denote this automorphism by $F_p$, and if $p \neq 3$, then $\Aut(G_-) = \langle G_-, F_p \rangle$. The action of $F_p$ on the maximal torus $T_-$ of $G_-$ 
is then given by raising each element to its $p$-th power, which is analogous to the action of $F_p$ on the maximal torus $T_+$ of $G_+$. 
The order of $F_p$ is given by $2f$ in this case since $(v_2 F_q)^2 = F_{q^2} = F_p^{2f}$ acts trivially on $G_-$, while $F_p^f = F_q$ acts on $G_-$ by conjugation with $v_2$.

(ii) Following its definition, for $\delta \in \{ \pm \}$ the field automorphism $F_p$ acts trivially on 
$$\norma_{G_\delta}(\T)/T_\delta = \norma_{\G^{F_\delta}}(\T) / \T^{F_\delta} = \langle T_\delta, n_r(1) \mid r \in \Sigma \rangle / T_\delta \cong D_{12},$$
where we have $\norma_{\G^{F_\delta}}(\T) = \langle T_\delta, n_r(1) \mid r \in \Sigma \rangle $ by Lemma~\ref{lem:Omega2Action}.
Furthermore, any element of $\norma_{G_\delta}(\T)$ may be written in the form $n\cdot t$ for some $t \in T_\delta$ and $n \in \norma_{G_\delta}(\T)$ with $F_p(n) = n$.
\end{rmk}


\subsubsection{Blocks of $G_2(q)$}

The $\ell$-blocks and $\ell$-decomposition numbers of $G$ have been described by Hiß and Shamash in a series of papers for various primes $\ell$. The case $\ell \geqslant 5$ is not considered here since in this situation the iBAW condition has already been proven to hold for $G$, see the proof of Theorem~A.

If $q$ is odd, then by \cite{HissShamash2} the 2-blocks of $G$ may be divided into the following classes:
the principal $2$-block $B_0$,
the $2$-block $B_3$ (only for $3 \nmid q$),
the $2$-blocks of types $B_{1a}$, $B_{1b}$, $B_{2a}$, $B_{2b}$,
of types $B_{X_1}$, $B_{X_2}$, $B_{X_a}$, $B_{X_b}$,
and those of $2$-defect zero,
where the $2$-blocks of non-cyclic defect are exactly $B_0$, $B_3$ (if $3 \nmid q$), and those of types $B_{1a}$, $B_{1b}$, $B_{2a}$, $B_{2b}$, $B_{X_1}$ and $B_{X_2}$ (cf.~also \cite[p.~36]{AnG2Weights}).

For $3 \nmid q$, by \cite{HissShamash3} the 3-blocks of $G_2(q)$ are given by
the principal $3$-block $B_0$,
the $3$-block $B_2$ (only for $2 \nmid q$),
the $3$-blocks of types $B_{1a}$, $B_{1b}$, $B_{2a}$, $B_{2b}$,
of types $B_{X_1}$, $B_{X_3}$, $B_{X_a}$, $B_{X_b}$ if $q \equiv 1 \bmod 3$ or
of types $B_{X_2}$, $B_{X_6}$, $B_{X_a}$, $B_{X_b}$ if $q \equiv -1 \bmod 3$,
and those of $3$-defect zero,
where the $3$-blocks of non-cyclic defect are exactly $B_1$, $B_2$ (if $q$ is odd) and those of types $B_{\delta a}$, $B_{\delta b}$ and $B_{X_\delta}$, where $\delta = 1$ if $q \equiv 1 \bmod 3$ and $\delta = 2$ if $q \equiv -1 \bmod 3$.


\subsection{Action of Automorphisms}
\label{ch:AutoG2}
We are interested in the action of the automorphisms of $G$ on its irreducible Brauer characters as well as on its $\ell$-weights in the cases $\ell=2$ and $\ell=3$, where $\ell \nmid q$. 
By Proposition~\ref{prop:G2_AutomorphismGroup} it suffices to understand the behaviour of the Brauer characters and weights of $G$ under the action of $F_p$ if $3 \nmid q$ or $\Gamma$ if $3 \mid q$.


\subsubsection{Action on the Brauer Characters of $G_2(q)$}
Using simple arguments on the degrees of the irreducible (Brauer) characters of $G$, the unitriangularity of the corresponding decomposition matrices and certain character values on unipotent elements, we proved the following in \cite[Prop.~12.1, 12.2]{Diss}:

\begin{prop}
\label{prop:G2_Auto2Char}
Let $B$ be a $2$-block of $G$ and assume that $q$ is odd. Then one of the following holds:
\begin{enumerate}[(i)]
\item $B \neq B_0$ and $Aut(G)_B$ acts trivially on $\IBr(B)$,
\item $B = B_0$, $3 \nmid q$, and $Aut(G)_B = \Aut(G)$ acts trivially on $\IBr(B)$,
\item $B = B_0$, $3 \mid q$, and $\Gamma$ interchanges 
exactly two irreducible Brauer characters in $B$ and leaves the remaining elements of $\IBr(B)$ invariant.
\end{enumerate}
\end{prop}

\begin{prop}
\label{prop:G2_Auto3Char}
Let $B$ be a $3$-block of $G$ and assume that $3 \nmid q$. Then $\Aut(G)_B$ acts trivially on $\IBr(B)$.
\end{prop}


\subsubsection{Action on the Weights of $G_2(q)$}

The $\ell$-weights of $G$ for blocks of non-cyclic defect have been determined by An \cite{AnG2Weights} for various primes $\ell$. We summarize his results for $\ell \in \{ 2, 3 \}$ and examine the action of $\Aut(G)$ on those weights.

For the course of this section we also allow $q < 5$ since in certain cases we need to consider weights of the subgroup \smash{$G_2(p)$ of $G_2(q)$}, which makes it necessary to include the groups $G_2(2)$ and $G_2(3)$ in our investigations.
This shall not cause any problems here since neither our proofs nor any of the results we refer to rely on the condition that $q \geqslant 5$.\vspace{\baselineskip}


\paragraph{\it The Case $\ell = 2$}
\label{ssec:G2_Auto2Weights}

Let $\ell=2 \nmid q$ and $\varepsilon \in \{ \pm 1 \}$ be such that $q \equiv \varepsilon \bmod 4$. 
Since any Sylow 2-subgroup of $G$ is contained in the centralizer of an involution, so is any 2-subgroup of $G$. Let us hence fix the involution $y := h_a(-1) h_b(-1) \in G$ for the following investigations. By application of Lemma~\ref{lem:TableCartanIntegers} and the fact that $$h_\gamma(z) = h(z^{\langle \gamma,\; \xi_1 \rangle}, z^{\langle \gamma,\; \xi_2 \rangle}, z^{\langle \gamma,\; \xi_3 \rangle})$$ for $\gamma \in \Sigma$ and $z \in \mathbb{F}^\times$ we obtain that $y = h(1, -1, -1) \in T_+$. Moreover, by \cite[Thm.~2.5]{ClassInvolutions} all involutions in $G$ are $G$-conjugate to $y$.

\begin{lem}
\label{lem:G2_Centralizer_y}
The centralizer in $G$ of the involution $y= h_a(-1) h_b(-1)$ is given by
\begin{align*}
\cent_G(y) = \langle T_+,\; x_{a+b}(t),\; x_{3a+b}(t),\; n_{a+b}(1),\; n_{3a+b}(1) \mid t \in \mathbb{F}_q \rangle.
\end{align*}
\end{lem}

\begin{proof}
Since $y \in T_+ \subseteq \T$, there exists an $\mathbb{F}$-character $\chi$ such that $y = h(\chi)$. By \cite[Thm.~4.1]{ChangConjugate} (for $p \geqslant 5$) and \cite[Prop.~3.1 and Sec.~7]{EnomotoConjugacy} (for $p=3$) we have $$\cent_G(h(\chi)) = \langle T_+,\; x_{\pm r}(t) \mid r \in \Sigma,\; \chi(r)=1, \; t \in \mathbb{F}_q \rangle.$$%
As $y = h(1, -1, -1)$, it follows that $y$ corresponds to the $\mathbb{F}$-character $\chi \colon \mathbb{Z}\Sigma \longrightarrow \mathbb{F}^\times$ with $\chi(a+b)=1$ and $\chi(a)=-1$, and since $$\chi(c_1 a + c_2 b) = \chi(a)^{c_1} \chi(b)^{c_2} = (-1)^{c_1 + c_2}$$ for $c_1, c_2 \in \mathbb{Z}$, we have $\chi(r) = 1$ for $r \in \Sigma\,$ if and only if $\,r \in \{ \pm (a+b), \pm (3a+b) \}$.
\end{proof}

\begin{nota}
For an even natural number $n \geqslant 2$, a prime power $q$ and a sign $\delta \in \{ \pm \}$ we denote by
$\CO^{\delta}_n(q),$ 
$\GO^{\delta}_n(q)$ and
$\SO^{\delta}_n(q)$
with
$$ \CO^{\delta}_n(q) \trianglerighteq \GO^{\delta}_n(q) \trianglerighteq \SO^{\delta}_n(q) \quad \text{and} \quad \CO^{\delta}_n(q) \trianglerighteq \SO^{\delta}_n(q)$$
 the \textit{conformal orthogonal group}, the \textit{general orthogonal group} and the \textit{special orthogonal group} \textit{over $\mathbb{F}_q$ of degree $n$ and $\delta$-type}, respectively. We refer to \cite[Sec.~2.5]{KleidmanLiebeck} for a detailed description of these groups. One should note that the notation in \cite{KleidmanLiebeck} differs from the notation used here in the way that in \cite{KleidmanLiebeck} the conformal orthogonal groups are denoted by $\GO^{\pm}_n(q)$, while the general orthogonal groups are written as $\Orth^{\pm}_n(q)$.
\end{nota}

\begin{rmk}
It follows from \cite[(1A) and (3D)]{AnG2Weights} that $\cent_G(y)$ is isomorphic to $\SO_4^+(q)$.
\end{rmk}

Now we study extraspecial $2$-subgroups $2_+^{1+4}$ of $G$ of order $2^{1+4}$ and plus type.

\begin{lem}
\label{lem:Uniqueness 2+4+1}
Suppose that $q \equiv \pm 3 \bmod 8$. Then any two subgroups of \,$\SO_4^+(q)$ isomorphic to $2_+^{1+4}$ are conjugate in $\CO_4^+(q)$.

Moreover, any subgroup $R\leqslant\SO_4^+(q)$ isomorphic to $2_+^{1+4}$ is $2$-radical in $\SO_4^+(q)$ with $$\norma_{\SO_4^+(q)}(R)/R \cong (C_3 \times C_3) \rtimes C_2,$$ where the non-trivial element of $C_2$ acts on $C_3 \times C_3$ by inversion.
\end{lem}

\begin{proof}
The first statement is part of \cite[(1G)(b)]{2WeightsClassical}. Hence, for $R_1, R_2 \leqslant \SO_4^+(q)$ isomorphic to $2_+^{1+4}$ it holds that $R_1$ is a radical $2$-subgroup of $\SO_4^+(q)$ if and only if $R_2$ is. By \cite[(2B)]{AnG2Weights} the group $\SO_4^+(q)$ contains a radical $2$-subgroup $R \cong 2_+^{1+4}$, which satisfies $\norma_{\SO_4^+(q)}(R)/R \cong (C_3 \times C_3) \rtimes C_2$ with the non-trivial element of $C_2$ acting on $C_3 \times C_3$ by inversion.
\end{proof}

Observe that for $3 \mid q$, we have $q \equiv 1 \bmod 8$ if $q$ is an even power of $3$, and $q \equiv 3 \bmod 8$ else. Hence, $q \equiv \pm 3 \bmod 8$ immediately implies that $q \equiv 3 \bmod 8$ in this situation.

\begin{prop}
\label{prop:R_2+4+1}
Suppose that $3 \mid q$ and let $q \equiv 3 \bmod 8$. Moreover, define 
\begin{align*}
R := \langle x_{3a+b}(-1) x_{-(3a+b)}(-1),\; x_{a+b}(-1) x_{-(a+b)}(-1),\; n_{3a+b}(1),\; n_{a+b}(1) \rangle \leqslant G_2(3) \leqslant G.
\end{align*}
Then $R \cong 2_+^{1+4}$ is a radical 2-subgroup of both $G_2(3)$ and $G_2(q)$.
\end{prop}

\begin{proof}
We first observe that $R \cong 2_+^{1+4}$ by exploiting the fact that $2_+^{1+4} \cong D_8 \circ D_8$ is the central product of two dihedral groups of order 8: from the relations in Theorem~\ref{thm:ChevalleyRelUni} we conclude that
\begin{align*}
R &= \left( \langle x_{a+b}(-1) x_{-(a+b)}(-1) \rangle \rtimes \langle n_{a+b}(1) \rangle \right) \circ \left( \langle x_{3a+b}(-1) x_{-(3a+b)}(-1) \rangle \rtimes \langle n_{3a+b}(1) \rangle \right)
\end{align*}
is isomorphic to $D_8 \circ D_8$ as claimed.

We have $y= h_{3a+b}(-1) = n_{3a+b}(1)^2 \in R$, and one easily verifies that $\Z(R) = \langle y \rangle$. Hence, $R \subseteq \cent_G(y) \cong \SO_4^+(q)$ (and $R \subseteq \cent_{G_2(3)}(y) \cong \SO_4^+(3)$), and from Lemma~\ref{lem:Uniqueness 2+4+1} it thus follows that $R$ is 2-radical in both $\SO_4^+(q)$ and $\SO_4^+(3)$. The center of $R$ is a characteristic subgroup of $R$, so we obtain that $\norma_G(R) \subseteq \norma_G(\Z(R)) = \cent_G(y)$. Consequently, $R$ is 2-radical in $G_2(q)$, and analogously it also follows that $R$ is 2-radical in $G_2(3)$.
\end{proof}

\begin{prop}
\label{prop:G2_NormalizerR_+1+4}
Assume that $3 \mid q$ and let $q \equiv 3 \bmod 8$. Moreover, suppose that the group $R \cong 2_+^{1+4}$ is as in Proposition~\ref{prop:R_2+4+1}. Then $\norma_G(R) = \cent_{G_2(3)}(y)$.
\end{prop}

\begin{proof}
In the proof of the Proposition~\ref{prop:R_2+4+1} we observed that $\norma_G(R) \subseteq \cent_G(y) \cong \SO_4^+(q)$, whence we may apply Lemma~\ref{lem:Uniqueness 2+4+1} to deduce that
$$\norma_G(R)/R \cong (C_3 \times C_3) \rtimes C_2,$$
with the action of the non-trivial element of $C_2$ on $C_3 \times C_3$ given by inversion. E.g.~by application of Theorem~\ref{thm:ChevalleyRelUni} one verifies that the group $$\cent_{G_2(3)}(y) = \langle h_a(-1),\; h_b(-1),\; x_{a+b}(1),\; x_{3a+b}(1),\; n_{a+b}(1),\; n_{3a+b}(1)\rangle \geqslant R$$ 
is contained in $\norma_G(R)$. Now $h_a(-1)$ and $h_b(-1)$ agree modulo $R$ since $y = h_a(-1) h_b(-1)$ lies in $R$, so by the Chevalley commutator formula, Theorem~\ref{thm:ChevalleyRelUni}(iv) and Lemma~\ref{lem:TableCartanIntegers}
\begin{align*}
\cent_{G_2(3)}(y) / R &= \langle \overline{h_a(-1)},\; \overline{x_{a+b}(1)},\; \overline{x_{3a+b}(1)} \rangle
= (\langle \overline{x_{a+b}(1)} \rangle \times \langle \overline{x_{3a+b}(1)} \rangle) \rtimes \langle \overline{h_a(-1)} \rangle
\end{align*}
with the action of $\overline{h_a(-1)}$ given by inversion and $\overline{\phantom{m}} \colon \cent_{G_2(3)}(y) \twoheadrightarrow \cent_{G_2(3)}(y) / R$ denoting the natural epimorphism. Since $3 \mid q$, it follows that $x_{a+b}(1)^3 = x_{3a+b}(1)^3 = 1$, and hence that $\norma_G(R) = \cent_{G_2(3)}(y)$.
\end{proof}

\begin{lem}
\label{lem:GammaActionp3_2+4+1}
Assume that $3 \mid q$ and let $q \equiv 3 \bmod 8$. Moreover, let $R \cong 2_+^{1+4}$ be as in Proposition~\ref{prop:R_2+4+1} and suppose the notation of the proof of Proposition~\ref{prop:G2_NormalizerR_+1+4}, that is,
$$\norma_G(R)/R = (\langle \overline{x_{a+b}(1)} \rangle \times \langle \overline{x_{3a+b}(1)} \rangle) \rtimes \langle \overline{h_a(-1)} \rangle.$$
Let us parametrize the elements of $\norma_G(R) / R$ by setting
$$((s,t),z) :=  \overline{x_{a+b}(s)}\; \overline{x_{3a+b}(t)}\, z \in \norma_G(R)/R$$
for $s, t \in \mathbb{F}_3$ and $z \in \langle \overline{h_{a}(-1)} \rangle$. Then $\Gamma(((s,t),z)) = ((t,s),z)$.
\end{lem}

\begin{proof}
This is obvious by Lemma~\ref{lem:GammaAction}.
\end{proof}

Let us now determine the conjugacy classes of the quotient $\norma_G(R)/R$, where $R$ is as in Proposition~\ref{prop:R_2+4+1}. Information on these will allow us to understand the action of $\Gamma$ on the irreducible characters of  $\norma_G(R)/R$ in the case that $3 \mid q$.

\begin{lem}
\label{lem:ConjClass_2+4+1}
Assume that $3 \mid q$ with $q \equiv 3 \bmod 8$, let $R \cong 2_+^{1+4}$ be as in Proposition~\ref{prop:R_2+4+1} and assume the notation of Lemma~\ref{lem:GammaActionp3_2+4+1}. Then the conjugacy classes of $\norma_G(R) / R$ are given by
\begin{align*}
C_{s,t} &:= \{ ((s,t),1),\; ((-s,-t),1) \},\; s, t \in \mathbb{F}_3,\; \text{and}\\
C_\star &:=  \{ ((u,v),z) \mid u, v \in \mathbb{F}_3,\; z \neq 1 \},
\end{align*}
where $C_{s,t} = C_{s',t'}$ if and only if $(s,t) = (s',t')$ or $(s,t) = (-s',-t')$.
In particular, the quotient $\norma_G(R) / R$ has the 6 conjugacy classes $C_{0,0}$, $C_{0,1}$, $C_{1,0}$, $C_{1,1}$, $C_{1,-1}$ and $C_\star$.
\end{lem}

\begin{proof}
This follows by easy computations.
\end{proof}

\begin{lem}
\label{lem:CharTab_2+4+1}
Assume that $3 \mid q$ with $q \equiv 3 \bmod 8$. Moreover, suppose that $R \cong 2_+^{1+4}$ is as in Proposition~\ref{prop:R_2+4+1}. Then the character table of $\norma_G(R) / R$ is given by:
\renewcommand{\arraystretch}{1.1}
$$\begin{array}{|l||cccccc|}
\hline
& C_{0,0} & C_\star & C_{0,1} & C_{1,0} & C_{1,1} & C_{1,-1}\\ 
\hline\hline
\chi_{0,0} & 1 & \phantom{-}1 & \phantom{-}1 & \phantom{-}1 & \phantom{-}1 & \phantom{-}1 \\
\chi_{\star} & 1 & -1 & \phantom{-}1 & \phantom{-}1 & \phantom{-}1 & \phantom{-}1\\
\chi_{0,1} & 2 & \phantom{-}\cdot & \phantom{-}2 & -1 & -1 & -1 \\
\chi_{1,0} & 2 & \phantom{-}\cdot & -1 & \phantom{-}2 & -1 & -1 \\
\chi_{1,1} & 2 & \phantom{-}\cdot & -1 & -1 & \phantom{-}2 & -1 \\
\chi_{1,-1} & 2 & \phantom{-}\cdot & -1 & -1 & -1 & \phantom{-}2 \\
\hline
\end{array}$$
\end{lem}
\renewcommand{\arraystretch}{1}

\begin{proof}
According to Lemma~\ref{lem:Uniqueness 2+4+1} it holds that $\norma_G(R) / R \cong (C_3 \times C_3) \rtimes C_2$, where the non-trivial element of $C_2$ acts on $C_3 \times C_3$ by inversion. This group has identifier $[18,4]$ in the \textsf{SmallGroups Library} \cite{CTblLib1.2.1} provided by \GAP\;\cite{GAP4}.
\end{proof}

\begin{prop}
\label{prop:GammaAction_2+4+1}
Let $3 \mid q$ with $q \equiv 3 \bmod 8$. Moreover, let $R \cong 2_+^{1+4}$ be as in Proposition~\ref{prop:R_2+4+1} with $\Irr(\norma_G(R)/R) = \{ \chi_{0,0}, \chi_*, \chi_{0,1}, \chi_{1,0}, \chi_{1,1}, \chi_{1,-1} \}$ as in Lemma~\ref{lem:CharTab_2+4+1}. Then $\Gamma$ interchanges $\chi_{0,1}$ and $\chi_{1,0}$ but stabilizes  the remaining irreducible characters of $\norma_G(R)/R$.
\end{prop}

\begin{proof}
By Lemma~\ref{lem:GammaActionp3_2+4+1} and Lemma~\ref{lem:ConjClass_2+4+1} we have $\Gamma(C_{0,1}) = C_{1,0}$, while the remaining conjugacy classes are left invariant by $\Gamma$. Hence, the claim follows immediately upon comparison with the character table of $\norma_G(R)/R$ given in Lemma~\ref{lem:CharTab_2+4+1}.
\end{proof}

We now study subgroups of $G$ isomorphic to the central product $2_+^{1+2} \circ D_{(q^2-1)_2}$ of an extraspecial group of order 8 and plus type with a dihedral group of order $(q^2-1)_2$.

\begin{lem}
\label{lem:Unique_2+2+1DSO}
Suppose that $q \equiv 1 \bmod 8$. Then any two subgroups of $\SO_4^+(q)$ isomorphic to $2_+^{1+2} \circ D_{(q^2-1)_2}$ are conjugate in $\CO_4^+(q)$.
Moreover, any subgroup $R\leqslant\SO_4^+(q)$ isomorphic to $2_+^{1+2} \circ D_{(q^2-1)_2}$ is $2$-radical in $\SO_4^+(q)$ with $\norma_{\SO_4^+(q)}(R)/R \cong \mathfrak{S}_3$.
\end{lem}

\begin{proof}
For the first statement we note that both groups $2_+^{1+2}$ and $D_{(q^2-1)_2}$ may be embedded into $\GO_2^+(q)$ (e.g.~by \cite[(1G)(b)]{2WeightsClassical} and the fact that $D_{(q^2-1)_2}$ is isomorphic to a Sylow $2$-subgroup of $\GO_2^+(q)$ by \cite[Prop.~2.9.1(iii)]{KleidmanLiebeck}).
One easily verifies that any faithful absolutely irreducible $\mathbb{F}_q$-representation of $2_+^{1+2} \circ D_{(q^2-1)_2}$ is $4$-dimensional, and that, moreover, up to $\GL_4(q)$-conjugation the images in $\GL_4(q)$ of all faithful irreducible $4$-dimensional $\mathbb{F}_q$-representations of $2_+^{1+2} \circ D_{(q^2-1)_2}$ agree. Thus, any two subgroups of $\SO_4^+(q)$ that are isomorphic to $2_+^{1+2} \circ D_{(q^2-1)_2}$ must be conjugate in $\GL_4(q)$, whence from \cite[Cor.~2.10.4(iii)]{KleidmanLiebeck} it follows that they must even be conjugate in $\CO_4^+(q)$ as claimed.

Now, if $R_1, R_2 \leqslant \SO_4^+(q)$ with $R_1, R_2 \cong 2_+^{1+2} \circ D_{(q^2-1)_2}$, then $$\norma_{\SO_4^+(q)}(R_1)/R_1 \cong \norma_{\SO_4^+(q)}(R_2)/R_2$$ by the first part, so it suffices to know that there exists at least one radical 2-subgroup of $\SO_4^+(q)$ isomorphic to $2_+^{1+2} \circ D_{(q^2-1)_2}$. This, as well as the statement on the normalizer of $R$ in $\SO_4^+(q)$, holds by \cite[(2B)]{AnG2Weights}. 
\end{proof}

\begin{lem}
\label{lem:Rad_2+2+1D}
Suppose that $q \equiv 1 \bmod 8$. The groups
\begin{align*}
R_{a+b} &:= \langle n_{a+b}(1),\; n_{3a+b}(1),\; h_b(-1),\;\, h_{a+b}(t)\; \mid t\in \mathbb{F}_q^\times,\; t^{(q-1)_2}=1 \rangle \leqslant G,\; \\
R_{3a+b} &:= \langle n_{a+b}(1),\; n_{3a+b}(1),\; h_a(-1),\; h_{3a+b}(t) \mid t\in \mathbb{F}_q^\times,\; t^{(q-1)_2}=1 \rangle \leqslant G
\end{align*}
are isomorphic to $2_+^{1+2} \circ D_{(q^2-1)_2}$ and 2-radical in $G$. Moreover, $R_{a+b}$ and $R_{3a+b}$ are not conjugate in $G$.
\end{lem}

\begin{proof}
We have $y =h_a(-1) h_b(-1) = h_{a+b}(-1) \in R_{a+b}$, so in particular also $h_a(-1) \in R_{a+b}$. Using the relations given in Theorem~\ref{thm:ChevalleyRelUni} one can show that $h_a(-1)n_{a+b}(1)$ has order 2 and acts on $\langle h_{a+b}(t) \mid t\in \mathbb{F}_q^\times,\; t^{(q-1)_2}=1 \rangle$ by inversion. Moreover, $h_a(-1)n_{3a+b}(1)$ has order 2 and acts on $\langle h_{a}(-1) h_{a+b}(i) \rangle$ by inversion, where $i \in \mathbb{F}_q^\times$ denotes an element of order 4 (which exists by assumption). Now the subgroups 
\begin{align*}
&\langle h_{a+b}(t) \mid t\in \mathbb{F}_q^\times,\; t^{(q-1)_2}=1 \rangle \rtimes \langle h_a(-1)n_{a+b}(1) \rangle \cong D_{(q^2-1)_2}
\intertext{and}
&\langle h_{a}(-1) h_{a+b}(i) \rangle \rtimes \langle h_a(-1)n_{3a+b}(1) \rangle \cong D_{8} \cong 2_+^{1+2}
\end{align*}
of $R_{a+b}$ commute with each other (following Lemma~\ref{lem:WeylGroupAction}), both have center $\langle h_{a+b}(-1) \rangle$, and together they generate $R_{a+b}$. We conclude that $R_{a+b} \cong 2_+^{1+2} \circ D_{(q^2-1)_2}$. Analogously, one can also show that $R_{3a+b} \cong 2_+^{1+2} \circ D_{(q^2-1)_2}$.

Since $y = h_{a}(-1)h_{b}(-1) = h_{a+b}(-1) = h_{3a+b}(-1)$ as observed before, it is contained in both groups, and in fact both groups have center generated by $y$. Hence,
$$R_{a+b},\, R_{3a+b} \subseteq \cent_G(y) \cong \SO_4^+(q),$$
and from Lemma~\ref{lem:Unique_2+2+1DSO} we deduce that $R_{a+b}$ and $R_{3a+b}$ are 2-radical in $\cent_G(y)$. Moreover, since $\Z(R_{a+b}) = \Z(R_{3a+b}) = \langle y \rangle$, also
$$\norma_G(R_{a+b}),\, \norma_G(R_{3a+b}) \subseteq \cent_G(y) \cong \SO_4^+(q).$$
Hence, $R_{a+b}$ and $R_{3a+b}$ are 2-radical in $G$ with $\norma_G(R_{a+b})/R_{a+b}$ and $\norma_G(R_{3a+b})/R_{3a+b}$ isomorphic to $\mathfrak{S}_3$ following Lemma~\ref{lem:Unique_2+2+1DSO}.

Now suppose there is $g \in G$ such that $R_{3a+b} = R_{a+b}^g$. We then have $[R_{3a+b}, R_{3a+b}] = [R_{a+b}, R_{a+b}]^g$ as well. Easy calculations under consideration of Theorem~\ref{thm:ChevalleyRelUni} yield
\begin{align*}
[R_{a+b}, R_{a+b}] &= \langle \, h_{a+b}(t^2)\, \mid t\in \mathbb{F}_q^\times,\; t^{(q-1)_2}=1 \rangle,\\
[R_{3a+b}, R_{3a+b}] &= \langle h_{3a+b}(t^2) \mid t\in \mathbb{F}_q^\times,\; t^{(q-1)_2}=1 \rangle.
\end{align*}
Now let $t \in \mathbb{F}_q^\times$ be of order $(q-1)_2$. Then there must exist another element $z \in \mathbb{F}_q^\times$ of order $(q-1)_2$ such that $h_{3a+b}(t^2) = h_{a+b}(z^2)^g$. It is well-known that all elements of $T_+ = \T^F$ which are conjugate in $G$ must be conjugate by an element of the Weyl group $\W$ (see, for instance, \cite[Prop.~3.7.1]{CarterFiniteLie}), which equals 
\begin{align*}
\W = \langle \T,\; n_r(1) \mid r \in \Sigma \rangle / \T.
\end{align*}
Moreover, by Theorem~\ref{thm:ChevalleyRelUni}(vii) we have $$n_r(1) h_{s}(z^2) n_r(1)^{-1} = h_{\omega_r(s)}(z^2)$$ for all $r, s \in \Sigma$, and $\omega_r(s)$ is a short root if and only if $s$ is short, so there must exist a short root $s$ such that $h_{3a+b}(t^2) = h_s(z^2)$. From \cite[Thm.~1.9.5(d)]{Classification} it follows that $$t^{2\langle 3a+b,\, r \rangle} = z^{2 \langle s,\, r \rangle}$$ for all $r \in \Sigma$ (note that the definition of $\langle \_\, , \_ \rangle$ in \cite{Classification} differs from the one used here in the way that the roles of the two arguments are interchanged). For $r = a+b$ we have $$\langle 3a+b  ,\, r\rangle = 2\frac{(3a+b,\, a+b)}{(3a+b,\, 3a+b)} = 0$$ since $a+b$ and $3a+b$ are orthogonal roots. Hence, $z^{2 \langle s,\, a+b \rangle} = 1$. Being a short root, $s$ is contained in $\{ \pm a,\; \pm (a+b),\; \pm (2a+b) \}$, so $\langle s,\, a+b \rangle \in \{ \pm 1,\; \pm 2 \}$ by Lemma~\ref{lem:TableCartanIntegers}. Thus, $z^4 = 1$, which contradicts the assumption that $z$ has order $(q-1)_2$ and $q \equiv 1 \bmod 8$.
\end{proof}

\begin{prop}
\label{prop:Auto_Rad_2+2+1D}
Let $q \equiv 1 \bmod 8$. The groups $R_{a+b}$ and $R_{3a+b}$ from Lemma~\ref{lem:Rad_2+2+1D} are stabilized by $F_p$.
If $3 \mid q$, then $\Gamma$ interchanges $R_{a+b}$ and $R_{3a+b}$.
\end{prop}

\begin{proof}
The first statement is obvious. The second claim follows from Lemma~\ref{lem:GammaAction} and the fact that $\rho \colon \Sigma \longrightarrow \Sigma$ interchanges $a+b$ and $3a+b$, as well as $a$ and $b$.
\end{proof}

We now study the action of $\Aut(G)$ on the $2$-weights of $G$.\vspace{\baselineskip}


\subparagraph{\it The Principal Block $B_0$}
\label{sssec:2WeightsB1}

The $2$-weights of $G$ belonging to the principal $2$-block $B_0$ have been described by An in \cite{AnG2Weights} as follows:

\begin{prop}
\label{prop:2WeightsB1}
Suppose that $B = B_0$ is the principal $2$-block of $G$. Then $|\mathcal{W}(B)| = 7$. Moreover, if $(R, \varphi)$ is a $B$-weight of $G$, then up to $G$-conjugation one of the following holds:
\begin{enumerate}[(i)]
\item $R \cong (C_2)^3$, $\norma_G(R) / R \cong \GL_3(2)$, and $\varphi$ is the inflation of the Steinberg character of $\GL_3(2)$. There exists exactly one $G$-conjugacy class of such $B$-weights in $G$.
\item $R \sim_{\G} \langle \mathcal{O}_2(T_\varepsilon), v_2 \rangle$, $\norma_G(R) / R \cong \mathfrak{S}_3$, and $\varphi$ is the inflation of the unique irreducible character of $\mathfrak{S}_3$ of degree $2$. There exists exactly one $G$-conjugacy class of such $B$-weights in $G$.
\item $R \in \Syl_2(G)$ is a Sylow 2-subgroup of $G$, $\norma_G(R) = R$, and $\varphi$ is the trivial character of $\norma_G(R)$. There exists exactly one $G$-conjugacy class of such $B$-weights in $G$.
\item $q \equiv \pm 1 \bmod 8$, $R \cong 2_+^{1+2} \circ D_{(q^2-1)_2}$, $\norma_G(R) / R \cong \mathfrak{S}_3$, and $\varphi$ is the inflation of the unique irreducible character of $\mathfrak{S}_3$ of degree $2$. There exist exactly two $G$-conjugacy classes of such $B$-weights in $G$.
\item $q \equiv \pm 1 \bmod 8$, $R \cong 2_+^{1+4}$, $\norma_G(R) / R \cong \mathfrak{S}_3 \times \mathfrak{S}_3$, and $\varphi$ is the inflation of the unique irreducible character of $\mathfrak{S}_3 \times \mathfrak{S}_3$ of degree 4. There exist exactly two $G$-conjugacy classes of such $B$-weights in $G$. 
\item $q \equiv \pm 3 \bmod 8$, $R \cong 2_+^{1+4}$, $\norma_G(R) / R \cong (C_3 \times C_3) \rtimes C_2$ with the non-trivial element of $C_2$ acting on $C_3 \times C_3$ by inversion, and $\varphi$ is the inflation of one of the four irreducible characters of $\norma_G(R) / R$ of degree 2. There exists exactly one $G$-conjugacy class of such $R$ in $G$.
\end{enumerate}
\end{prop}

\begin{proof}
This follows from \cite[(3C)]{AnG2Weights} and the proof of \cite[(3H)]{AnG2Weights}.
\end{proof}

The action of $\Aut(G)$ on the $B_0$-weights of $G$ is given as follows:

\begin{prop}
\label{prop:Auto2WeightsB1}
Let $B = B_0$ be the principal $2$-block of $G$ and suppose that $(R, \varphi)$ is a $B$-weight of $G$. The following statements hold:
\begin{enumerate}[(i)]
\item If $R$ is as in (i), (ii), (iii) or (v) of Proposition~\ref{prop:2WeightsB1}, then up to $G$-conjugation $(R, \varphi)$ is invariant under $\Aut(G)$.
\item If $q \equiv \pm 1 \bmod 8$ and $R \cong 2_+^{1+2} \circ D_{(q^2-1)_2}$, then $F_p$ stabilizes the $G$-conjugacy class of $(R, \varphi)$, while $\Gamma$ (if existent, i.e.,~if $3 \mid q$) interchanges the two $G$-conjugacy classes of type $(R, \varphi)$.
\item If $q \equiv \pm 3 \bmod 8$ and $R \cong 2_+^{1+4}$, then $F_p$ stabilizes the $G$-conjugacy class of $(R, \varphi)$. If $3 \mid q$, then two of the weight characters corresponding to $R$ are stabilized by $\Gamma$, the remaining two are interchanged.
\end{enumerate}
\end{prop}

\begin{proof}
(i) This is clear in cases (i), (ii) or (iii) of Proposition~\ref{prop:2WeightsB1} since the character $\varphi$ is uniquely determined by $R$ and there exists a unique $G$-conjugacy class of $B$-weights with first component isomorphic to $R$. If $R$ is as in (v) of Proposition~\ref{prop:2WeightsB1}, then there exist exactly two $G$-conjugacy classes of $B$-weights of type $(R, \varphi)$ with $\varphi$ uniquely determined by $R$, and at least one of these has a representative with first component lying in $G_2(p)$, so both classes are stabilized by $F_p$. If $p=3$, then the automorphism $\Gamma$ also acts on $G_2(3)$, which contains exactly one conjugacy class of such $R$ by Proposition~\ref{prop:2WeightsB1}(vi). Hence, $\Gamma$ stabilizes both $G$-conjugacy classes of $R$ that exist in $G=G_2(q)$, so in particular, the $G$-conjugacy class of $(R, \varphi)$ is stabilized.

For (ii) we observe that in consequence of Lemma~\ref{lem:Rad_2+2+1D} we may assume that up to $G$-conjugation $R$ is one of $R_{a+b}$ or $R_{3a+b}$. Since $R_{a+b}$ and $R_{3a+b}$ are not $G$-conjugate by Lemma~\ref{lem:Rad_2+2+1D}, it follows from Proposition~\ref{prop:Auto_Rad_2+2+1D} that the claim holds if $q \equiv 1 \bmod 8$. Now suppose that $q \equiv -1 \bmod 8$. Then we have $q = p^f$ for odd $f$ and $p \neq 3$. Hence, the automorphism $\Gamma$ does not exist, and $F_p$ has odd order $f$ on $G$. Consequently, $F_p$ must stabilize both $G$-conjugacy classes of type $(R, \varphi)$.

For (iii) we observe that by Proposition~\ref{prop:2WeightsB1}(vi) there exists a unique $G$-conjugacy class of $B$-weights of type $(R, \varphi)$ in $G$, and this has a representative with first component contained in $G_2(p)$, so assume that $R \leqslant G_2(p)$. Moreover, $q \equiv \pm 3 \bmod 8$ implies that also $p \equiv \pm 3 \bmod 8$, so by the same proposition $\norma_G(R) = \norma_{G_2(p)}(R)$. We deduce that $F_p$ stabilizes all characters of $\norma_G(R) / R$, so in particular $(R, \varphi)$ is stabilized. For $p=3$ we may assume that $R$ is as in Proposition~\ref{prop:R_2+4+1}, whence the action of $\Gamma$ on the irreducible characters of $\norma_G(R) / R$ of degree two is described in Proposition~\ref{prop:GammaAction_2+4+1}.
\end{proof}


\subparagraph{\it The Blocks of Types $B_3$, $B_{1a}$, $B_{1b}$, $B_{2a}$, $B_{2b}$, $B_{X_1}$ and $B_{X_2}$}
\label{sssec:2WeightsB3}

For a $2$-block $B$ of $G$ we write $B \in \{ B_{1a}, B_{1b}, B_{2a}, B_{2b} \}$ if $B$ is of one of the types $B_{1a}$, $B_{1b}$, $B_{2a}$, $B_{2b}$, and $B \in \{ B_{X_1}, B_{X_2}\}$ if $B$ is of type $B_{X_1}$ or $B_{X_2}$.

\begin{prop}
\label{prop:2Weights_B3BabBx}
The following statements hold for a $2$-block $B$ of $G$:
\begin{enumerate}[(i)]
\item If $B=B_3$, then $|\mathcal{W}(B)| = 3$. Moreover, if $(R_1, \varphi_1)$ and $(R_2, \varphi_2)$ are $B$-weights of $G$ that are not $G$-conjugate, then $R_1$ and $R_2$ are not isomorphic.
\item If $B \in \{ B_{1a}, B_{1b}, B_{2a}, B_{2b} \}$, then $|\mathcal{W}(B)| = 2$. Moreover, if $(R_1, \varphi_1)$ and $(R_2, \varphi_2)$ are $B$-weights of $G$ that are not $G$-conjugate, then $R_1$ and $R_2$ are not isomorphic.
\item If $B \in \{ B_{X_1}, B_{X_2} \}$, then $|\mathcal{W}(B)| = 1$.
\end{enumerate}
In particular, in any of these cases $\Aut(G)_B$ acts trivially on $\mathcal{W}(B)$.
\end{prop}

\begin{proof}
These statements follow from the proof of \cite[(3I)]{AnG2Weights}.
\end{proof}


\paragraph{\it The Case $\ell = 3$}
\label{ssec:G2_Auto3Weights}

Let $\ell=3$ so that $3 \nmid q$. Moreover, define $\varepsilon \in \{ \pm 1 \}$ to be such that $q \equiv \varepsilon \bmod 3$ and recall that we also allow $q<5$ here.
For convenience we will mainly work with the group $G_\varepsilon = \G^{F_\varepsilon}$ throughout this section since it provides a particularly nice description for the maximal torus $T_\varepsilon = \T^{F_\varepsilon}$, see Table~\ref{tb:MaximalTorusG2}. By \cite[(1E)]{AnG2Weights} we have
 $$\norma_{G_\varepsilon}(T_\varepsilon) / T_\varepsilon \cong D_{12}.$$
Since we also have $\norma_{G_\varepsilon}(\T)/\T^{F_\varepsilon} \cong \W^{F_\varepsilon} \cong D_{12}$ with $\norma_{G_\varepsilon}(\T) \subseteq \norma_{G_\varepsilon}(T_\varepsilon)$, it follows that $$\norma_{G_\varepsilon}(T_\varepsilon) = \norma_{G_\varepsilon}(\T) = \langle T_\varepsilon, n_r(1) \mid r \in \Sigma \rangle$$ in this case (cf.~Lemma~\ref{lem:Omega2Action}).

Let us now consider the group
$$\Lbf := \langle x_b(t),\; x_{-b}(t),\; x_{3a+b}(t),\; x_{-(3a+b)}(t) \mid t \in \mathbb{F} \rangle \leqslant \G = G_2(\mathbb{F}).$$
The Steinberg relations show that $\Lbf \cong \SL_3(\mathbb{F})$. 
Moreover, $\Lbf$ is stable under both $F_+$ and $F_-$, and we set
\begin{align*}
L_+ &:= \Lbf^{F_+} \cong \SL_3(q),\\
L_- &:= \Lbf^{F_-} \cong \SU_3(q) =: \SL_3(-q),
\end{align*}
where $\SU_3(q)$ denotes the special unitary group of degree $3$ over $\mathbb{F}_{q^2}$. In addition, we set $K_\varepsilon := \langle L_\varepsilon, v_2 \rangle \leqslant G_\varepsilon$. This yields the semidirect product $K_\varepsilon = L_\varepsilon \rtimes \langle v_2 \rangle$ with $v_2$ acting on $L_\varepsilon$ as the transpose-inverse automorphism does on $\SL_3(\varepsilon q)$ (cf.~Lemma~\ref{lem:Omega2Action}). Clearly, the field automorphism $F_p$ of $G_\varepsilon$ acts on both $L_\varepsilon$ and $K_\varepsilon$, and $F_p(v_2) = v_2$.

Part of the subsequent result is due to Hiß--Shamash \cite{HissShamash3}, who determined the 3-blocks and corresponding Brauer characters of $G_2(q)$ for $q$ not divisible by 3:

\begin{prop}
\label{prop:Defect3}
Let $B$ be a 3-block of $G_\varepsilon$. Then 
\begin{enumerate}[(i)]
\item $B$ has maximal defect if and only if $B$ is the principal $3$-block of ${G_\varepsilon}$,
\item $B$ has abelian defect groups if and only if $B$ is non-principal, and
\item if $B$ has non-cyclic abelian defect groups, then $\mathcal{O}_3(T_\varepsilon)$ is a defect group of $B$. In this case we have \begin{align*}
\cent_{G_\varepsilon}(\mathcal{O}_3(T_\varepsilon)) = T_\varepsilon 
\text{\quad and \quad}\norma_{G_\varepsilon}(\mathcal{O}_3(T_\varepsilon)) = \norma_{G_\varepsilon}(T_\varepsilon),
\end{align*}
with $\norma_{G_\varepsilon}(T_\varepsilon) /T_\varepsilon \cong D_{12}$.
\end{enumerate}
\end{prop}

\begin{proof}
Parts (i), (ii) and the first statement in (iii) are given by \cite[Sec.~2.2, 2.3]{HissShamash3}. For the normalizer and centralizer in (iii) see, for instance, \cite[(1D)]{AnG2Weights}.
\end{proof}

Hence, if $B$ is a non-principal $3$-block of $G_\varepsilon$ that has non-cyclic defect groups, then up to $G_\varepsilon$-conjugation all $B$-weights of $G_\varepsilon$ derive from the radical 3-subgroup $R=\mathcal{O}_3(T_\varepsilon)$ of ${G_\varepsilon}$. According to Construction~\ref{constr:B-weights} the irreducible characters of $R\cent_{G_\varepsilon}(R) = T_\varepsilon$ play a major role in this case. We fix the following parametrization for $\Irr(T_\varepsilon)$:

\begin{nota}
\label{nota:G2_ParametrizationIrrT}
Let $z \in \mathbb{F}^\times$ be of order $q-\varepsilon$ and denote by $\theta_0$ the irreducible character of $\langle z \rangle$ given by $\theta_0(z) = \exp(2 \pi \mathfrak{i}/(q-\varepsilon))$. Then the irreducible characters of $T_\varepsilon$ may be parametrized as
\begin{align*}
\Irr(T_\varepsilon) = \{ \theta_0^i \times \theta_0^j \mid 0 \leqslant i, j < q-\varepsilon \},
\end{align*}
where $(\theta_0^i \times \theta_0^j)(h(z_1, z_2, z_3)) = \theta_0^i(z_1) \theta_0^j(z_2)$. 
\end{nota}

Let us now examine how the normalizer of $T_\varepsilon$ in $G_\varepsilon$ acts on $\Irr(T_\varepsilon)$. At the beginning of the present section we observed that $\norma_{G_\varepsilon}(T_\varepsilon) = \langle T_\varepsilon,\, n_r(1) \mid r \in \Sigma \rangle$, whence it suffices to understand the action of $n_r(1)$, $r \in \Sigma$, on $\Irr(T_\varepsilon)$.

\begin{lem}
\label{lem:G2_WeylGroupActionOnIrredTorusChars}
Let $\theta \in \Irr(T_\varepsilon)$ and $0 \leqslant i, j < q-\varepsilon$ be such that $\theta = \theta_0^i \times \theta_0^j$. Then we have
$\theta^{n_a(1)}      =   \theta_0^{i\phantom{-j}}      \times    \theta_0^{i-j}$ and $\theta^{n_b(1)}            =   \theta_0^{j\phantom{-i}}      \times    \theta_0^{i\phantom{-j}}$.
\end{lem}

\begin{proof}
This is a direct consequence of Lemma~\ref{lem:WeylGroupAction} and Lemma~\ref{lem:Omega23Action}.
\end{proof}

\begin{coro}
\label{coro:G2_IrredTorusCharsInert2}
Let $\theta \in \Irr(T_\varepsilon)$ be such that $\norma_{G_\varepsilon}(T_\varepsilon)_\theta / T_\varepsilon \cong C_2$. Then up to $\norma_{G_\varepsilon}(T_\varepsilon)$-conjugation it holds that
$$\theta = \theta_0^i \times \theta_0^i \quad \text{or} \quad \theta = \theta_0^i \times \theta_0^{-i}$$
for a suitable $0 < i < q-\varepsilon$.
\end{coro}

\begin{proof}
This follows easily from Lemma~\ref{lem:G2_WeylGroupActionOnIrredTorusChars} since $\norma_{G_\varepsilon}(T_\varepsilon)/ T_\varepsilon = \langle T_\varepsilon, n_r(1) \mid r \in \Sigma \rangle / T_\varepsilon$ is dihedral of order 12, and as such it contains exactly seven involutions, which must be given by the elements $n_r(1) T_\varepsilon$, $r \in \Sigma$ a positive root, and $v_2 T_\varepsilon$.
\end{proof}


\subparagraph{\it The Principal Block $B_0$}
\label{sssec:3WeightsB1}

The statement below provides information on the local properties of extraspecial 3-subgroups $3^{1+2}_+$ of $G_\varepsilon$ of order $3^{1+2}$ and exponent $3$. Along with the Sylow 3-subgroups of $G_\varepsilon$ these groups give rise to the 3-weights for the principal $3$-block $B_0$.

\begin{prop}
\label{prop:extraspecial_3_2+1}
For $R \leqslant G_\varepsilon$ with $R \cong 3^{1+2}_+$ the following statements hold: 
\begin{enumerate}[(i)]
\item Up to $G_\varepsilon$-conjugation we have $R \leqslant L_\varepsilon$ and $\norma_{G_\varepsilon}(R) \leqslant K_\varepsilon$.
\item If $R$ is contained in $L_\varepsilon$, then we have 
\begin{align*}
\norma_{L_\varepsilon}(R)/R \cong \begin{cases}
Q_8 & \text{if $(q^2-1)_3=3$},\\
\Sp_2(3) & \text{if $(q^2-1)_3>3$}.
\end{cases}
\end{align*}
Moreover, $L_\varepsilon$ contains exactly one $L_\varepsilon$-conjugacy class of subgroups isomorphic to $R$ if $(q^2-1)_3=3$, and three such $L_\varepsilon$-conjugacy classes if $(q^2-1)_3>3$.
\item \begin{enumerate}[(1)]
\item If $(q^2-1)_3=3$, then $R \in \Syl_3(G_\varepsilon)$, so $G_\varepsilon$ contains exactly one $G_\varepsilon$-conjugacy class of subgroups isomorphic to $R$, and we have $\norma_{G_\varepsilon}(R) = \langle \norma_{L_\varepsilon}(R), \rho\rangle$ for some $\rho \in K_\varepsilon \setminus L_\varepsilon$.
\item If $(q^2-1)_3>3$, then $G_\varepsilon$ contains two $G_\varepsilon$-conjugacy classes of subgroups isomorphic to $R$. One of these has $\norma_{G_\varepsilon}(R) = \norma_{L_\varepsilon}(R)$, the other one satisfies $\norma_{G_\varepsilon}(R) = \langle \norma_{L_\varepsilon}(R), \rho\rangle$ for some $\rho \in K_\varepsilon \setminus L_\varepsilon$.
\end{enumerate} 
\end{enumerate} 
\end{prop}

\begin{proof}
This is proven in \cite[(1E)]{AnG2Weights} and \cite[(1G)]{AnG2Weights}.
\end{proof}

\begin{prop}
\label{prop:3WeightsB1}
For $B=B_0$ the principal $3$-block of ${G_\varepsilon}$ we have $|\mathcal{W}(B)| = 7$. Moreover, if $(R,\varphi)$ is a $B$-weight of ${G_\varepsilon}$, then up to $G_\varepsilon$-conjugation one of the following holds:
\begin{enumerate}[(i)]
\item $(q^2-1)_3 = 3$, $R \in \Syl_3({G_\varepsilon})$ with $R \cong 3^{1+2}_+$, and $\varphi$ is the inflation of one of the seven irreducible characters of $\norma_{G_\varepsilon}(R)/R$.
\item $(q^2-1)_3 > 3$, $R \in \Syl_3({G_\varepsilon})$ is a Sylow $3$-subgroup of ${G_\varepsilon}$, $\norma_{G_\varepsilon}(R)/R \cong C_2 \times C_2$, and $\varphi$ is the inflation of one of the four linear characters of $C_2 \times C_2$.
\item $(q^2-1)_3 > 3$, $R \leqslant L_\varepsilon$  with $R \cong 3^{1+2}_+$ such that it holds $\norma_{G_\varepsilon}(R) = \norma_{K_\varepsilon}(R)$, $|\norma_{K_\varepsilon}(R) \colon \norma_{L_\varepsilon}(R)| = 2$, and $\varphi$ is the inflation of one of the two extensions of the Steinberg character of $\norma_{L_\varepsilon}(R)/R \cong \Sp_2(3)$ to $\norma_{G_\varepsilon}(R)/R$. There exists exactly one $G_\varepsilon$-conjugacy class of such $R$ in ${G_\varepsilon}$.
\item $(q^2-1)_3 > 3$, $R \leqslant L_\varepsilon$  with $R \cong 3^{1+2}_+$ where it holds that $\norma_{G_\varepsilon}(R) = \norma_{L_\varepsilon}(R)$, and $\varphi$ is the inflation of the Steinberg character of $\norma_{G_\varepsilon}(R)/R \cong \Sp_2(3)$. There exists exactly one $G_\varepsilon$-conjugacy class of such $R$ in ${G_\varepsilon}$.
\end{enumerate}
\end{prop}

\begin{proof}
This follows from the proof of \cite[(3A)]{AnG2Weights} and Proposition~\ref{prop:extraspecial_3_2+1}.
\end{proof}

\begin{prop}
\label{prop:Auto3WeightsB1}
Suppose that $B = B_0$ is the principal $3$-block of $G_\varepsilon$. Then the action of $\Aut(G_\varepsilon)_B = \Aut(G_\varepsilon)$ on $\mathcal{W}(B)$ is trivial.
\end{prop}

\begin{proof}
Since $3 \nmid q$, according to Proposition~\ref{prop:G2_AutomorphismGroup} it suffices to prove that $F_p$ stabilizes any conjugacy class of $B$-weights in ${G_\varepsilon}$. We go through the cases listed in Proposition~\ref{prop:3WeightsB1}.

Let $(R, \varphi)$ be as in Proposition~\ref{prop:3WeightsB1}(i). Then $R$ is a Sylow $3$-subgroup of $G_2(q)$, and by Proposition~\ref{prop:extraspecial_3_2+1} we have $|\norma_{G_2(q)}(R)/R\,| = 2 |Q_8| = 16$.
Now consider the group $G_2(p)$. Since $p \neq 3$, this has $|G_2(p)|_3 = |R|$, and by changing to a $G_2(q)$-conjugate we may thus assume that $R \leqslant G_2(p)$. Again by Proposition~\ref{prop:extraspecial_3_2+1} it follows that also $|\norma_{G_2(p)}(R)/R\,| = 2 |Q_8| = 16$, whence 
$$\norma_{G_2(q)}(R) = \norma_{G_2(p)}(R).$$
In particular, $F_p$ acts trivially on $R$ and $\norma_{G_2(q)}(R)$, and thus leaves $(R, \varphi)$ invariant.

Now suppose the situation of Proposition~\ref{prop:3WeightsB1}(ii), i.e.,~$(q^2-1)_3>3$ and $R \in \Syl_3({G_\varepsilon})$. Following the proof of \cite[(1E)]{AnG2Weights} we may assume that $$R = \langle \mathcal{O}_3(T_\varepsilon), v_3 \rangle$$ such that $\norma_{G_\varepsilon}(R) T_\varepsilon = \norma_{G_\varepsilon}(T_\varepsilon)$, $\norma_{G_\varepsilon}(R) \cap T_\varepsilon = \mathcal{O}_3(T_\varepsilon)$ and $\norma_{G_\varepsilon}(R)/R \cong C_2 \times C_2$. Moreover, we have $\langle T_\varepsilon, v_2, v_3, n_b(1) \rangle / T_\varepsilon \cong D_{12}$, so
$$\norma_{G_\varepsilon}(T_\varepsilon) = \langle T_\varepsilon, v_2, v_3, n_b(1) \rangle.$$
Hence, we have
$$\norma_{G_\varepsilon}(R) = \langle \mathcal{O}_3(T_\varepsilon), v_3, s v_2, t n_{b}(1) \rangle$$
for suitable $s, t \in T_\varepsilon$ and $\norma_{G_\varepsilon}(R)/R = \langle \overline{\mathstrut s v_2} \rangle \times \langle \overline{t n_{b}(1)} \rangle$, where $\overline{\phantom{-}} \colon \norma_{G_\varepsilon}(R) \twoheadrightarrow \norma_{G_\varepsilon}(R)/R$ denotes the natural epimorphism.
We claim that $F_p$ acts trivially on $\norma_{G_\varepsilon}(R)/R$.

 Suppose that $F_p(\overline{\mathstrut s v_2}) = \overline{t n_b(1)}$. Since $v_2$ and $n_b(1)$ do not coincide modulo $T_\varepsilon$ and $v_3$ has order $3$ by Lemma~\ref{lem:Omega23Action}, we must have $F_p(s v_2) = s^p v_2 = t n_b(1) x v_3^i$ for some $x \in \mathcal{O}_3(T_\varepsilon)$ and $i \in \{ 1, 2 \}$. Hence, modulo $T_\varepsilon$ the elements $v_2$ and $n_b(1) v_3^i$ coincide. Now following Lemmas~\ref{lem:WeylGroupAction} and~\ref{lem:Omega23Action} we have
$$
h(z_1, z_2, z_3)^{v_2} = h(z_1^{-1}, z_2^{-1}, z_3^{-1})
$$
but
\begin{align*}
h(z_1, z_2, z_3)^{n_b(1) v_3} = h(z_3, z_2, z_1),\\
h(z_1, z_2, z_3)^{n_b(1) v_3^2} = h(z_1, z_3, z_2)\phantom{,}
\end{align*}
for any $z_1, z_2, z_3 \in \mathbb{F}^\times$ with $z_1 z_2 z_3 =1$. In particular, since $3 \mid (q-\varepsilon)$, we may always find $z_1, z_2, z_3 \in \mathbb{F}^\times$ such that $h(z_1, z_2, z_3) \in T_\varepsilon$ and $$h(z_1, z_2, z_3)^{v_2} \not\in \{ h(z_1, z_2, z_3)^{n_b(1) v_3},\; h(z_1, z_2, z_3)^{n_b(1) v_3^2} \},$$  so $F_p(\overline{\mathstrut s v_2}) = \overline{t n_b(1)}$ is impossible. 
Similar arguments show that $F_p$, which stabilizes $R$, acts trivially on $\norma_{G_\varepsilon}(R)/R$, and thus on any $B$-weight $(R, \varphi)$.

Let now $(q^2-1)_3>3$ and $R \cong 3^{1+2}_+$ as in (iii) or (iv) of Proposition~\ref{prop:3WeightsB1}. By Proposition~\ref{prop:extraspecial_3_2+1} we may assume that $R \leqslant L_\varepsilon$ and $\norma_{G_\varepsilon}(R) \leqslant K_\varepsilon$, with $\norma_{L_\varepsilon}(R)/R \cong \Sp_2(3)$. Let $\omega \in \mathbb{F}^\times$ be of order $3$ and consider
$$R':=\left\langle \left[\begin{smallmatrix}
\omega & 0\phantom{^{-1}} & 0 \\
0 & \omega^{-1} & 0 \\
0 & 0\phantom{^{-1}} & 1
\end{smallmatrix}\right],\; \left[\begin{smallmatrix}
0 & 0 & 1 \\
1 & 0 & 0 \\
0 & 1 & 0
\end{smallmatrix}\right] \right\rangle \leqslant \SL_3(\varepsilon q),$$
which is extraspecial of order $3^{1+2}$ and exponent 3 (cf.~the proof of \cite[(1G)]{AnG2Weights}). Recall that by Proposition~\ref{prop:extraspecial_3_2+1} there exist exactly three $L_\varepsilon$-conjugacy classes of subgroups isomorphic to $R$ in $L_\varepsilon \cong \SL_3(\varepsilon q)$, two of which are conjugate under $G_\varepsilon$, and the other one having a representative that is stabilized by $v_2$ respectively by the transpose-inverse automorphism (cf.~also the proof of \cite[(1G)]{AnG2Weights}). This automorphism fixes $R'$, so if $R$ is as in (iii), then we may suppose that $R$ corresponds to the group $R'$ in $\SL_3(\varepsilon q) \cong L_\varepsilon$ and $\norma_{G_\varepsilon}(R) = \langle \norma_{L_\varepsilon}(R), v_2 \rangle$. In particular, $R$ is stabilized by $F_p$.
Let $\St^\pm \in \Irr(\norma_{G_\varepsilon}(R)/R)$ denote the two extensions of the Steinberg character $\St$ of $$\norma_{L_\varepsilon}(R)/R \cong \Sp_2(3)$$ to $\langle \norma_{L_\varepsilon}(R), v_2 \rangle / R$. Then we have $\St^-(v_2) = -\St^+(v_2).$  By Lemma~\ref{lem:Omega23Action} it holds that $v_2^2 = 1$. Moreover, $\St^+(1)= 3$. Hence, the character value $\St^+(v_2)$ is congruent to 3 modulo 2. In particular, we have $\St^+(v_2)$, $\St^-(v_2) \neq 0$. Now, since $F_p$ fixes $v_2$, we have
\begin{align*}
(\St^+)^{F_p}(v_2) = \St^+(v_2) \neq 0 \text{\quad and \quad}
(\St^-)^{F_p}(v_2) = \St^-(v_2) \neq 0.
\end{align*}
Moreover, being the unique irreducible character of $\norma_{G_\varepsilon}(R)/R$ of degree $3$, $\St$ is left invariant by $F_p$, so $$((\St^\pm)^{F_p})_{|\norma_{L_\varepsilon}(R)/R} = \St^{F_p} = \St,$$ 
and we conclude by application of \cite[Rmk.~9.3(i)]{SpaethExceptionalGroups} that $F_p$ leaves both $\St^+$ and $\St^-$ invariant. Hence, it follows that $F_p$ fixes the $G_\varepsilon$-conjugacy class of any $B$-weight $(R, \varphi)$ as $\varphi$ is the inflation of one of $\St^+$ or $\St^-$.

Finally, suppose that $(R, \varphi)$ is as in (iv). Then up to ${G_\varepsilon}$-conjugation $R$ is uniquely determined in $G_\varepsilon$ by its normalizer $\norma_{G_\varepsilon}(R)$, and $\varphi$ is uniquely determined by $R$, so the $G_\varepsilon$-conjugacy class of $(R, \varphi)$ is stabilized by $\Aut({G_\varepsilon})$. This completes the proof.
\end{proof}

\renewcommand{\arraystretch}{1}


\subparagraph{\it The Block $B_{2}$}
\label{sssec:3WeightsB2}

Recall that this $3$-block only exists if $q$ is odd.

\begin{prop}
\label{prop:3WeightsB2}\label{prop:Auto3WeightsB2}
Let $B=B_2$. Then $|\mathcal{W}(B)| = 4$. Moreover, if $(R, \varphi)$ is a $B$-weight of ${G_\varepsilon}$, then up to $G_\varepsilon$-conjugation  $R = \mathcal{O}_3(T_\varepsilon)$ with $\norma_{G_\varepsilon}(R) = \norma_{G_\varepsilon}(T_\varepsilon)$ and $\cent_{G_\varepsilon}(R) = T_\varepsilon$, and if a linear character $\theta \in \Irr(T_\varepsilon)$ is a constituent of $\varphi_{|T_\varepsilon}$, then $\theta^2 = 1_{T_\varepsilon} \neq \theta$ and $$\norma_{G_\varepsilon}(R)_\theta/T_\varepsilon \cong C_2 \times C_2.$$
The set $\Irr(\norma_{G_\varepsilon}(R)_\theta \mid \theta)$ consists of four distinct extensions of $\theta$, and $\varphi = \psi^{\norma_{G_\varepsilon}(T_\varepsilon)}$ for some $\psi \in \Irr(\norma_{G_\varepsilon}(R)_\theta \mid \theta)$. Furthermore, $\Aut(G_\varepsilon)_B = \Aut(G_\varepsilon)$ acts trivially on $\mathcal{W}(B)$.
\end{prop}

\begin{proof}
All but the last assertion follow from Proposition~\ref{prop:Defect3} and the proof of \cite[(3B)]{AnG2Weights}.
As before, for the last claim it suffices to check invariance under the action of $F_p$.
We let $(R,\varphi)$ be a $B$-weight of ${G_\varepsilon}$ and assume that $R = \mathcal{O}_3(T_\varepsilon)$ with $R \cent_{G_\varepsilon}(R) = T_\varepsilon$ and $\norma_{G_\varepsilon}(R) = \norma_{G_\varepsilon}(T_\varepsilon)$. Let $\theta$ be an irreducible constituent of $\varphi_{|T_\varepsilon}$. Then this has order 2, so in the parametrization of Notation~\ref{nota:G2_ParametrizationIrrT} it follows that $\theta$ is one of
$$\theta_0^{\frac{q-\varepsilon}{2}} \times \theta_0^{\frac{q-\varepsilon}{2}},\quad \theta_0^{\frac{q-\varepsilon}{2}} \times 1 , \quad 1 \times \theta_0^{\frac{q-\varepsilon}{2}}.$$
The latter two characters are conjugate to the first via $n_a(1)$ and $n_{a+b}(1)$, respectively (cf.~Lemma~\ref{lem:G2_WeylGroupActionOnIrredTorusChars}), so since by Clifford theory all irreducible constituents of $\varphi_{|T_\varepsilon}$ are $\norma_{G_\varepsilon}(T_\varepsilon)$-conjugate, we may assume that $$\theta = \theta_0^{\frac{q-\varepsilon}{2}} \times \theta_0^{\frac{q-\varepsilon}{2}},$$
which is left invariant by $n_b(1)$, $n_{2a+b}(1)$ and $v_2 = n_b(1) n_{2a+b}(1)^{-1}$ following Lemma~\ref{lem:G2_WeylGroupActionOnIrredTorusChars}. Now we have $\norma_{G_\varepsilon}(R)_\theta/T_\varepsilon \cong C_2 \times C_2$, so in fact $\norma_{G_\varepsilon}(R)_\theta = \langle T_\varepsilon, n_b(1), v_2 \rangle$. Since $F_p$ acts trivially on both $n_b(1)$ and $v_2$, it follows that $F_p$ stabilizes $\norma_{G_\varepsilon}(R)_\theta$ and we have $$(\psi^{F_p})^{\norma_{G_\varepsilon}(R)} = (\psi^{\norma_{G_\varepsilon}(R)})^{F_p}$$ for all $\psi \in \norma_{G_\varepsilon}(R)_\theta$. Thus, we only need to check that any extension of $\theta$ to $\norma_{G_\varepsilon}(R)_\theta$ stays invariant under $F_p$. But this follows from \cite[Rmk.~9.3(i)]{SpaethExceptionalGroups} since $\theta$ is linear, $$(\psi^{F_p})_{|T_\varepsilon} = \theta^{F_p} = \theta^p = \theta = \psi_{|T_\varepsilon},$$
$\psi^{F_p}(n_b(1)) = \psi(n_b(1)) \neq 0$ and $\psi^{F_p}(v_2) = \psi(v_2) \neq 0$ for any $\psi \in \norma_{G_\varepsilon}(R)_\theta$.
\end{proof}

\begin{rmk}
\label{rmk:G2_CanCharOrder2Extend}
By Proposition~\ref{prop:3WeightsB2} for the $3$-block $B=B_2$, a $B$-weight $(\mathcal{O}_3(T_\varepsilon), \varphi)$ and an irreducible constituent $\theta \in \Irr(T_\varepsilon)$ of $\varphi_{|T_\varepsilon}$ we have $\theta^{2} = 1_{T_\varepsilon}$, $\theta \neq 1_{T_\varepsilon}$, and $\theta$ extends to its stabilizer $\norma_{G_\varepsilon}(T_\varepsilon)_\theta$ in $\norma_{G_\varepsilon}(T_\varepsilon)$. This is stated in \cite[(3B)]{AnG2Weights} but the proof of the extendibility of $\theta$ given there is not very precise. For later use it will be convenient to reprove it here in more detail. As in the proof of Proposition~\ref{prop:Auto3WeightsB2} we may assume that
$$\theta = \theta_0^{\frac{q-\varepsilon}{2}} \times \theta_0^{\frac{q-\varepsilon}{2}}$$
and $\norma_{G_\varepsilon}(T_\varepsilon)_\theta = \langle T_\varepsilon, n_b(1), v_2 \rangle$. By definition we have $v_2 = n_b(1) n_{-(2a+b)}(1)$. From Theorem~\ref{thm:ChevalleyRelUni}(vi) and Table~\ref{tab:G2_Signs} it follows that $n_b(1)$ and $v_2$ commute. Since $\langle T_\varepsilon, n_b(1) \rangle / T_\varepsilon$ is cyclic, there exists an extension $\eta$ of $\theta$ to $\langle T_\varepsilon, n_b(1) \rangle$ (e.g.~\cite[Problem~6.17]{Isaacs}), and as $[n_b(1), v_2] = 1$, $n_b(1)^2 \in T_\varepsilon$ and $v_2$ normalizes $\langle T_\varepsilon, n_b(1) \rangle$ and stabilizes $\theta$, it follows that $v_2$ leaves $\eta$ invariant. Hence, $\eta$ has an extension $\eta'$ to $\norma_{G_\varepsilon}(T_\varepsilon)_\theta = \langle T_\varepsilon, n_b(1), v_2 \rangle$ since $\langle T_\varepsilon, n_b(1), v_2 \rangle / \langle T_\varepsilon, n_b(1) \rangle \cong C_2$ is cyclic. Then $\eta'$ extends $\theta$ to $\norma_{G_\varepsilon}(T_\varepsilon)_\theta$ as claimed.
\end{rmk}


\subparagraph{\it The Blocks of Types $B_{1a}$, $B_{1b}$, $B_{2a}$ and $B_{2b}$}
\label{sssec:3WeightsB1aB1bB2aB2b}

We set
\begin{align*}
B_a &:= \begin{cases}
B_{1a}\; &\text{if $\varepsilon=+1$},\\
B_{2a}\; &\text{if $\varepsilon=-1$},\\
\end{cases}
&
B_b &:= \begin{cases}
B_{1b}\; &\text{if $\varepsilon=+1$},\\
B_{2b}\; &\text{if $\varepsilon=-1$},\\
\end{cases}
\end{align*}
so if a $3$-block $B$ is of type $B_a$ or $B_b$, then it possesses non-cyclic defect groups. We write $B \in \{ B_a, B_b \}$ in this case.

\begin{prop}
\label{prop:3weightsB1a1b2a2b}\label{prop:Auto3WeightsBaBb}
Suppose that $B \in \{ B_a, B_b \}$ is a $3$-block of $G_\varepsilon$. Then $|\mathcal{W}(B)|=2$. Moreover, if $(R,\varphi)$ is a $B$-weight of ${G_\varepsilon}$, then up to $G_\varepsilon$-conjugation it holds that $R = \mathcal{O}_3(T_\varepsilon)$ with $\norma_{G_\varepsilon}(R) = \norma_{G_\varepsilon}(T_\varepsilon)$ and $\cent_{G_\varepsilon}(R) = T_\varepsilon$, and for any constituent $\theta\in \Irr(T_\varepsilon)$ of $\varphi_{|T_\varepsilon}$ we have $$\norma_{G_\varepsilon}(R)_\theta/T_\varepsilon \cong C_2$$ and $\varphi = \psi^{\norma_{G_\varepsilon}(R)}$ for one of the two extensions $\psi \in \Irr(\norma_{G_\varepsilon}(T_\varepsilon)_\theta \mid \theta)$ of $\theta$.
Furthermore, $\Aut(G_\varepsilon)_B$ acts trivially on $\mathcal{W}(B)$.
\end{prop}

\begin{proof}
All but the last assertion follow from Proposition~\ref{prop:Defect3} and the proof of \cite[(3B)]{AnG2Weights}. To prove the last claim, let $(R, \varphi)$ be a $B$-weight of $G_\varepsilon$ and assume that $R = \mathcal{O}_3(T_\varepsilon)$ with $\norma_{G_\varepsilon}(R) = \norma_{G_\varepsilon}(T_\varepsilon)$, and $\norma_{G_\varepsilon}(R)_\theta / T_\varepsilon \cong C_2$ if $\theta \in \Irr(T_\varepsilon)$ is a constituent of $\varphi_{|T_\varepsilon}$. Since by Clifford theory all irreducible constituents of $\varphi_{|T_\varepsilon}$ are $\norma_{G_\varepsilon}(T_\varepsilon)$-conjugate, in consequence of Corollary~\ref{coro:G2_IrredTorusCharsInert2} we may assume that
$$\theta = \theta_0^i \times \theta_0^i \quad \text{or} \quad \theta = \theta_0^i \times \theta_0^{-i} $$
for a suitable $0 < i < q-\varepsilon$, whence by Lemma~\ref{lem:G2_WeylGroupActionOnIrredTorusChars} we have
\begin{align*}
\norma_{G_\varepsilon}(R)_\theta = \begin{cases}
\langle T_\varepsilon, n_b(1) \rangle &\text{if $\theta = \theta_0^i \times \theta_0^i,$}\\
\langle T_\varepsilon, n_{2a+b}(1) \rangle &\text{if $\theta = \theta_0^i \times \theta_0^{-i}$}.
\end{cases}
\end{align*}
For this one should note that $v_2 \not\in \norma_{G_\varepsilon}(R)_\theta$ as otherwise $\theta$ would be as in the proof of Proposition~\ref{prop:Auto3WeightsB2} with $\norma_{G_\varepsilon}(R)_\theta \cong C_2 \times C_2$. In particular, $i \neq (q-\varepsilon)/2$.

Let us now suppose that $a \in \Aut(G_\varepsilon)_B$. Since any inner automorphism of $G_\varepsilon$ stabilizes $B$, we may assume that $a = F_p^k$ for some $k \in \mathbb{N}$ (cf.~Proposition~\ref{prop:G2_AutomorphismGroup}), so that
$$(R, \varphi)^a = (R, \varphi)^{F_p^k} = (R, \varphi^{F_p^k})$$
and we need to prove that $\varphi$ is left invariant by $F_p^k$. Due to the fact that $F_p$, and hence also $a=F_p^k$, acts on $T_\varepsilon$, we have that $\theta^{a}$ is an irreducible character of $T_\varepsilon$. Since $(R, \varphi)$ is a $B$-weight, it follows from Construction~\ref{constr:B-weights} that $\block(\theta)^{G_\varepsilon} = B$, and from the fact that $a$ stabilizes $B$ we may deduce that $\block(\theta^a)^{G_\varepsilon} =B$.
But $B$ has defect group $\mathcal{O}_3(T_\varepsilon)$ by Proposition~\ref{prop:Defect3}, just like any $3$-block of $T_\varepsilon$ in consequence of \cite[Thm.~4.8]{navarro}, so the extended first main theorem of Brauer \cite[Thm.~9.7]{navarro} implies that $\theta^a$ and $\theta$, the canonical characters of $\block(\theta^a)$ and $\block(\theta)$, respectively, are conjugate under $\norma_{G_\varepsilon}(T_\varepsilon)$. Now Lemma~\ref{lem:G2_WeylGroupActionOnIrredTorusChars} yields the following $\norma_{G_\varepsilon}(T_\varepsilon)$-conjugates of $\theta$ in the case $\theta = \theta_0^i \times \theta_0^i $:
$$\theta_0^i \times \theta_0^i,\quad \theta_0^i \times 1,\quad 1 \times \theta_0^i,\quad \theta_0^{-i} \times \theta_0^{-i},\quad 1 \times \theta_0^{-i},\quad \theta_0^{-i} \times 1.$$
These are pairwise distinct as $i \neq (q-\varepsilon)/2$, so since $|\norma_{G_\varepsilon}(T_\varepsilon) \colon \norma_{G_\varepsilon}(T_\varepsilon)_\theta| = 6$, these must indeed be all $\norma_{G_\varepsilon}(T_\varepsilon)$-conjugates of $\theta$ in this case. Similarly, for $\theta = \theta_0^i \times \theta_0^{-i}$ the $\norma_{G_\varepsilon}(T_\varepsilon)$-conjugates of $\theta$ are given by
$$\theta_0^i \times \theta_0^{-i},\quad  \theta_0^i \times \theta_0^{2i},\quad  \theta_0^{-i} \times \theta_0^i,\quad  \theta_0^{-2i} \times \theta_0^{-i},\quad  \theta_0^{2i} \times \theta_0^{i},\quad  \theta_0^{-i} \times \theta_0^{-2i}.$$
Note that these characters are pairwise distinct as $i \neq (q-\varepsilon)/2$ and $2i \not\equiv -i \bmod (q-\varepsilon)$, where the latter holds since otherwise $\theta$ would be of order $3$, which is not possible since $\mathcal{O}_3(T_\varepsilon) \subseteq \ker(\theta)$.
Now $a=F_p^k$ acts on $\Irr(T_\varepsilon)$ by raising the linear characters of $T_\varepsilon$ to their $p^k$-th power, that is,
$$\theta^{a} = \begin{cases}
\theta_0^{ip^k} \times \theta_0^{ip^k} &\text{if $\theta = \theta_0^i \times \theta_0^i,$}\\
\theta_0^{ip^k} \times \theta_0^{-ip^k} &\text{if $\theta = \theta_0^i \times \theta_0^{-i}$}.
\end{cases}$$
Hence, since $\theta^a$ and $\theta$ are $\norma_{G_\varepsilon}(T_\varepsilon)$-conjugate, the above observations on the shapes of the $\norma_{G_\varepsilon}(T_\varepsilon)$-conjugates of $\theta$ imply that $\theta^a \in \{ \theta, \theta^{-1}\}$. According to Lemma~\ref{lem:G2_WeylGroupActionOnIrredTorusChars} we have $\theta^{-1} = \theta^{v_2}$.
Let us introduce the notation $c_x(g) = x g x^{-1}$ for all $x, g \in G_\varepsilon$ and set
\begin{align*}
a' := \begin{cases}
a = F_p^k & \text{ if $\theta^a = \theta$},\\
a c_{v_2} = F_p^k c_{v_2} & \text{ if $\theta^a = \theta^{-1}$}.
\end{cases}
\end{align*}
Then $\theta^{a'} = \theta$, and $\varphi$ is left invariant by $a = F_p^k$ if and only if it is stabilized by $a'$ since $v_2 \in \norma_{G_\varepsilon}(T_\varepsilon) = \norma_{G_\varepsilon}(R)$. Moreover, in consequence of Theorem~\ref{thm:ChevalleyRelUni}(vi) and Table~\ref{tab:G2_Signs} we have $[n_{2a+b}(1), n_b(1)] = 1$, whence it follows that $a'$ stabilizes both $n_b(1)$ and $n_{2a+b}(1)$. Now by the first part of the claim there exists $\psi \in \Irr(\norma_{G_\varepsilon}(T_\varepsilon)_\theta \mid \theta)$ such that $\varphi = \psi^{\norma_{G_\varepsilon}(T_\varepsilon)}$ and hence
\begin{align*}
\varphi^{a'} &= (\psi^{\norma_{G_\varepsilon}(T_\varepsilon)})^{a'} = (\psi^{a'})^{\norma_{G_\varepsilon}(T_\varepsilon)},
\end{align*}
where $\psi^{a'} \in \Irr(\norma_{G_\varepsilon}(T_\varepsilon)_\theta \mid \theta)$ since $a'$ stabilizes $T_\varepsilon$, $n_b(1)$, $n_{2a+b}(1)$ and $\theta$. We prove that $\psi^{a'} = \psi$. It holds that $(\psi^{a'})_{|T_\varepsilon} = \theta = \psi_{|T_\varepsilon}$. Moreover, we have $\norma_{G_\varepsilon}(T_\varepsilon)_\theta = \langle T_\varepsilon, n \rangle$, where $n$ equals $n_b(1)$ or $n_{2a+b}(1)$ depending on $\theta$, and $n$ is left invariant by $a'$. Hence, since $\psi$ is linear, we conclude that
$$\psi^{a'}(n) = \psi(n) \neq 0,$$
implying in accordance with \cite[Rmk.~9.3(i)]{SpaethExceptionalGroups} that $\psi^{a'} = \psi$. Thus, we have $\varphi^{a} = \varphi^{a'} = \varphi$ as claimed, which concludes the proof.
\end{proof}


\subparagraph{\it The Blocks of Types $B_{X_1}$ and $B_{X_2}$}
\label{sssec:3WeightsBX1BX2}

Recall that a $3$-block of type $B_{X_1}$ has non-cyclic defect groups if and only if $q \equiv 1 \bmod 3$, while a $3$-block of type $B_{X_2}$ is of non-cyclic defect if and only if $q \equiv -1 \bmod 3$. By \cite[(3B)]{AnG2Weights} we have:

\begin{prop}
\label{prop:3weightsBX1BX2}
Let $B \in \{ B_{X_1}, B_{X_2} \}$ be a $3$-block of $G_\varepsilon$ with non-cyclic defect groups. Then $|\mathcal{W}(B)|=1$. In particular, $\Aut(G_\varepsilon)_B$ acts trivially on $\mathcal{W}(B)$.
\end{prop}

In summary, we have proven the following:

\begin{prop}
\label{prop:G2_ExistancePartiBij}
Let $\ell \in \{2, 3\}$ and let $B$ be an $\ell$-block of $G$ of non-cyclic defect. Then the iBAW condition (cf.~Definition~\ref{defi:iBAWCblock}) holds for $B$.
\end{prop}

\begin{proof}
We assume first that $B$ is not the principal $2$-block of $G$. By \cite{AnG2Weights} we have $|\IBr(B)| = |\mathcal{W}(B)|$. We may hence choose a bijection $\Omega_B \colon \IBr(B) \longrightarrow \mathcal{W}(B)$. 
By Propositions~\ref{prop:G2_Auto2Char} and \ref{prop:G2_Auto3Char} the action of $\Aut(G)_B$ on $\IBr(B)$ is trivial, and moreover, by the results of Section~\ref{ssec:G2_Auto2Weights} also $\mathcal{W}(B)$ is stabilized pointwise by $\Aut(G)_B$. In particular, $\Omega_B$ is trivially $\Aut(G)_B$-equivariant, whence by Lemma~\ref{lem:SL3_BijectionConstruction} the iBAW condition holds for $B$.

Assume now that $B$ is the principal $2$-block of $G$. By \cite{AnG2Weights} we have $|\IBr(B)| = |\mathcal{W}(B)| = 7$. If we can find an $\Aut(G)_B$-equivariant bijection $\Omega_B \colon \IBr(B) \longrightarrow \mathcal{W}(B)$, then by the same arguments as above the claim follows. According to Propositions~\ref{prop:G2_Auto2Char} and \ref{prop:Auto2WeightsB1} the action of $\Aut(G)_B=\Aut(G)$ is trivial on both $\IBr(B)$ and $\mathcal{W}(B)$ if $3 \nmid q$ as in this case $\Aut(G) = \langle G, F_p \rangle$ by Proposition~\ref{prop:G2_AutomorphismGroup}, so there is nothing to prove in this situation. Assume hence that $3 \mid q$, in which  case Proposition~\ref{prop:G2_AutomorphismGroup} gives $\Aut(G) = \langle G, \Gamma \rangle$. By Propositions~\ref{prop:G2_Auto2Char} and~\ref{prop:Auto2WeightsB1} it holds that $\Gamma$ induces a transposition on both $\IBr(B)$ and $\mathcal{W}(B)$. Thus, an $\Aut(G)$-equivariant bijection $\Omega_B \colon \IBr(B) \longrightarrow \mathcal{W}(B)$ exists.
\end{proof}

We may now prove our first main result:

\begin{proof}[Proof of Theorem~A]
The simple group $G = G_2(q)$ is its own universal covering group if $q \geqslant 5$, and by Proposition~\ref{prop:G2_AutomorphismGroup} its outer automorphism group is cyclic.

Let $\ell$ be a prime dividing $|G|$. If $\ell = p$, then the claim holds by \cite[Thm.~C]{SpaethBlockwise}, so we assume that $\ell \neq p$. Then $\ell$ divides at least one of the factors $\Phi_1(q)$, $\Phi_2(q)$, $\Phi_3(q)$ and $\Phi_6(q)$. Using \cite[Lemma~5.2]{MalleHeight0} one sees that if $\ell \geqslant 5$, then it divides exactly one of $\Phi_1(q)$, $\Phi_2(q)$, $\Phi_3(q)$ or $\Phi_6(q)$. Suppose that $5 \leqslant \ell \mid \Phi_3(q)\Phi_6(q)$. Then a Sylow $\ell$-subgroup of $G$ is contained in a maximal torus of $G$ of type $T_3$ or $T_6$ depending on whether $\ell$ divides $\Phi_3(q)$ or $\Phi_6(q)$, so in particular the Sylow $\ell$-subgroups of $G$ are cyclic in this case (cf.~Table~\ref{tb:MaximalTorusG2}), whence the iBAW conditions holds for $G$ and $\ell$ by \cite[Thm.~1.1]{SpaethKoshiCyclic}.

Assume now that $5 \leqslant \ell \mid \Phi_1(q)\Phi_2(q)$. Then the $\ell$-blocks of $G$ have either cyclic or maximal defect (see~\cite[Prop.~3.1, 4.1]{Shamashq+1} and \cite[p.~1379]{Shamash2k3k}). For the $\ell$-blocks of cyclic defect the iBAW condition holds again by \cite[Thm.~1.1]{SpaethKoshiCyclic}, while it has been proven to hold for $3$-blocks of maximal defect by Cabanes--Späth in \cite[Cor.~7.6]{CabanesSpaeth}. 

The case of non-cyclic $\ell$-blocks of $G$ for $\ell \in \{2, 3\}$ is treated in Proposition~\ref{prop:G2_ExistancePartiBij}, which concludes the proof.
\end{proof}


\section{The Groups $\trial$} \label{sec:3D4}

We now turn to the second series of exceptional groups of Lie type treated here, the Steinberg triality groups $\trial$.


\subsection{Properties of $\trial$}
\label{ch:3D4_Properties}

Steinberg's triality groups are finite groups of Lie type and may be constructed as fixed point groups of universal Chevalley groups of type $D_4$ under a certain endomorphism deriving from an exceptional symmetry of order $3$ of the associated Dynkin diagram.

We consider a root system $\Sigma$ of type $D_4$ over the field of real numbers, i.e.,
\begin{align*}
\Sigma = \{ \pm e_i \pm e_j \mid 1 \leqslant i < j \leqslant 4 \},
\end{align*}
where $e_1, e_2, e_3, e_4$ form an orthonormal basis of $\mathbb{R}^4$. Inside $\Sigma$ we fix a base for this root system given by $\Pi = \{ r_1, r_2, r_3, r_4 \}$, where $r_1:= e_1- e_2$, $r_2 := e_2 -e_3$, $r_3 := e_3 - e_4$ and $r_4 := e_3 + e_4$.

Let $\G = D_4(\mathbb{F})$ be a universal Chevalley group of type $D_4$ over $\mathbb{F}$ with Steinberg generators $x_r(t)$, $h_r(s)$ and $n_r(s)$, $r \in \mathbb{F}$, $s \in \mathbb{F}^\times$, as in Theorem~\ref{thm:ChevalleyRelUni}. We fix a prime power $q = p^f$, $f \in \mathbb{N}_{>0}$. The field automorphism $\mathbb{F} \longrightarrow \mathbb{F}$, $a \longmapsto a^q$, induces a Frobenius endomorphism $F_q$ of $\G$ (cf.~Section~\ref{ssec:PropertiesG2}).
Another endomorphism of $\G$ is obtained as follows. We consider the triality symmetry $\rho$ of the Dynkin diagram of type $D_4$ given by $$\rho: \Pi \longrightarrow \Pi, \quad r_1 \longmapsto r_3 \longmapsto r_4 \longmapsto r_1, \quad \rho(r_2) = r_2.$$
There exists a unique isometry of $\mathbb{R}^4= \mathbb{R}\Sigma$ which maps each $r \in \Pi$ to its image $\rho(r)$ under $\rho$, cf.~\cite[p.~217]{CarterSimple}. We denote this isometry by $\rho$ as well. According to \cite[Prop.~12.2.2]{CarterSimple} it holds that $\rho$ is a linear transformation of $\mathbb{R}^4$ satisfying $\rho(\Sigma) = \Sigma$. By \cite[Prop.~12.2.3 and Lemma~13.6.2]{CarterSimple} by choosing a suitable Chevalley basis for the simple Lie algebra of type $D_4$ underlying $\G$ we may assume that $\rho$ induces a graph automorphism $\tau$ of $\G$ given by
\begin{align*}
\tau \colon \G \longrightarrow \G, \quad x_r(t) \longmapsto x_{\rho(r)}(t),\quad r\in \Sigma,\; t \in \mathbb{F}.
\end{align*}
This satisfies $\tau(n_r(t))= n_{\rho(r)}(t)$ and $\tau(h_r(t))= h_{\rho(r)}(t)$ for $r \in \Sigma$, $t \in \mathbb{F}^\times$. Note that $\tau$ commutes with $F_q$ and has order $3$ as an automorphism of $\G$.

We define the endomorphism $F$ of $\G$ as the product $\tau F_q = F_q \tau$. Then $F^3 = \tau^3 F_q^3 = F_{q}^3$, so $F$ is a Steinberg endomorphism of $\G$. The group $\trial$ is defined as the group
\begin{align*}
G:=  \trial := \G^F = \{ g \in \G \mid F(g) = g \}.
\end{align*}
Since $F^3 = F_{q}^3$, it follows that $G$ is contained in the finite group $D_4(q^3) := D_4(\mathbb{F}_{q^3}) = \G^{F_q^3}$.
By \cite[Thm.~14.3.2]{CarterSimple} its order is given by
\begin{align*}
\big|\hspace{0.007cm}\trial\big| = q^{12} \Phi_1(q)^2 \Phi_2(q)^2 \Phi_3(q)^2 \Phi_6(q)^2 \Phi_{12}(q).
\end{align*}
It is well-known that for any prime power $q$ the group $\trial$ is simple and constitutes its own universal covering group, see, e.g., \cite[Table~24.2, Thm.~24.17 and Rmk.~24.19]{malletest}.\vspace{\baselineskip}

Let us reconsider the isometry $\rho$ of $\mathbb{R}^4= \mathbb{R}\Sigma$ defined above. For a root $r \in \Sigma$ we let  $\widetilde{r} := \frac{1}{3}(r + \rho(r) + \rho^2(r))$ denote the orthogonal projection of $r$ onto the subspace of $\mathbb{R}\Sigma$ invariant under $\rho$. By \cite[Example~23.6]{malletest} the set
$\widetilde{\Sigma} := \{ \widetilde{r} \mid r \in \Sigma \}$ forms a root system of type $G_2$. 
For $S \in \widetilde{\Sigma}$ we define $\Sigma(S):= \{ r \in \Sigma \mid \widetilde{r} = S \}$.
Then for each $S \in \widetilde{\Sigma}$ the set $\Sigma(S)$ is the positive system of roots of a root system of type $A_1$ or $A_1^3$.

For $S \in \widetilde{\Sigma}$ we define the group
$\X_S := \langle x_r(t) \mid r\in \Sigma(S),\; t \in \mathbb{F} \rangle \leqslant \G.$
 Moreover, for $S \in \widetilde{\Sigma}$, $r \in \Sigma$ with $\widetilde{r} = S$, and $t \in \mathbb{F}_{q^3}$ we set 
$$x_S(t) = \begin{cases}
 x_r(t) & \text{if $\Sigma(S)$ has type $A_1$},\\
 x_r(t) x_{\rho(r)}(t^q) x_{\rho^2(r)}(t^{q^2}) & \text{if $\Sigma(S)$ has type $A_1^3$}.
\end{cases}$$
Then by \cite[Lemma~63]{SteinbergLecture} for $S \in \widetilde{\Sigma}$ we have
$$\X_S^F = \begin{cases} \{ x_S(t) \mid t \in \mathbb{F}_{q} \} & \text{if $\Sigma(S)$ has type $A_1$},\\
\{ x_S(t) \mid t \in \mathbb{F}_{q^3} \} & \text{if $\Sigma(S)$ has type $A_1^3$},\end{cases}$$
and since $\G$ is universal, according to \cite[Lemma~64]{SteinbergLecture} the group $G = \G^F$ is generated by the fixed point groups $\X_S^F$ with $S \in \widetilde{\Sigma}$.


\subsubsection{Weyl Group and Maximal Tori of $\trial$}
\label{sec:3D4_WeylGroupTori}

Let $\T$ be the $F$-stable maximal torus of $\G$ generated by all elements $h_r(t)$, $r \in \Sigma$, $t \in \mathbb{F}^\times$, and 
$\W = \norma_{\G}(\T)/ \T$ be the corresponding Weyl group. 
As is well-known, $\W$ is isomorphic to $\langle \omega_r \mid r \in \Sigma \rangle$ via $n_r(1) \T \mapsto \omega_r$, so we may identify those two groups in the following.
Now we denote by $\omega_0$ the longest element of $\W$ and by $r_*$ the highest root of $\Sigma$. Then $\omega_0 = -1$ by \cite[Rmk.~1.8.9]{Classification} and one easily verifies that the root $e_1 + e_2 \in \Sigma$ has height $5$ with respect to the chosen base $\Pi$, which makes it the highest root of $\Sigma$. Moreover, we use the short notation
$$h(z_1, z_2, z_3, z_4) := h_{r_1}(z_1) h_{r_2}(z_2) h_{r_3}(z_3) h_{r_4}(z_4)$$
for $z_1, z_2, z_3, z_4 \in \mathbb{F}^\times$. There exist seven $G$-conjugacy classes of maximal tori in $G$ with representatives in $\G$ given in Table~\ref{tb:MaximalTorus3D4} below, see \cite[p.~42]{DeriziotisMichler} and \cite[Table I]{KleidmanMax3D4}.

\renewcommand{\arraystretch}{1.5}
\begin{table}[H]
\centering
$\begin{array}{|l||l|c|}
\hline 
 w \in W & \ \ \T^{wF} & \W^{wF} \\ 
\hline\hline
 1 &  \begin{aligned} \hspace{-0.05cm}T_{+\phantom{1,}} &= \{ h(z_1, z_2, z_1^q, z_1^{q^2}) \mid z_1^{q^3-1} = z_2^{q-1} = 1 \} \cong C_{q^3-1} \times C_{q-1} \end{aligned}  & D_{12} \\[0.2em] 
\hline
\omega_{1,+} \!:= \omega_{r_*} & \begin{aligned} \hspace{-0.05cm}T_{1,+} &= \{ h(z, z^{1-q^3}, z^{q^4}, z^{q^2}) \mid z^{(q^3-1) (q+1)} = 1 \} \cong C_{(q^3-1) (q+1)} \end{aligned}  & \!C_2 \times C_2\! \\[0.2em] 
\hline 
 \omega_{1,-} \!:= \omega_0 \omega_{1,+}\! & \begin{aligned} \hspace{-0.05cm}T_{1,-} &= \{ h(z, z^{1+q^3}, z^{q^4}, z^{q^2}) \mid z^{(q^3+1) (q-1)} = 1 \} \cong C_{(q^3+1) (q-1)} \end{aligned}  & \!C_2 \times C_2\! \\[0.2em]  
\hline 
 \omega_{2,+} \!:= \omega_{r_*}\omega_{r_2}\! & \vphantom{\Bigg|}\begin{aligned} \hspace{-0.05cm}T_{2,+} &= \{ h(z_1, z_2, z_1^qz_2, (z_1^{-1}z_2)^{q+1}) \mid z_1^{q^2+q+1} = z_2^{q^2+q+1} = 1 \} \\ &\cong C_{q^2+q+1} \times C_{q^2 +q+1} \end{aligned}  & \SL_2(3) \\[0.6em]  
\hline 
 \omega_{2,-} \!:= \omega_0 \omega_{2,+}\! & \vphantom{\Bigg|}\begin{aligned} \hspace{-0.05cm}T_{2,-} &= \{ h(z_1, z_2, z_1^{-q}z_2, (z_1z_2^{-1})^{q-1}) \mid z_1^{q^2-q+1} = z_2^{q^2-q+1} = 1 \} \\ &\cong C_{q^2-q+1} \times C_{q^2 -q+1} \end{aligned}  & \SL_2(3) \\[0.6em]  
\hline 
\omega_3 := \omega_{r_1}\omega_{r_2} & \begin{aligned} \hspace{-0.05cm}T_{3\phantom{,+}}  &= \{ h(z, z^{1+q^3}, z^{q}, z^{q^2}) \mid z^{q^4 -q^2 +1} = 1 \} \cong C_{q^4 -q^2 +1} \end{aligned}  & C_4 \\[0.2em] 
\hline 
 \omega_0  &  \begin{aligned} \hspace{-0.05cm}T_{-\phantom{1,}} = \{ h(z_1, z_2, z_1^{-q}, z_1^{q^2}) \mid z_1^{q^3+1} = z_2^{q+1} = 1 \} \cong C_{q^3+1} \times C_{q+1} \end{aligned}  & D_{12} \\[0.2em] 
\hline 
\end{array}$
\caption{Maximal tori of $\trial$}
\label{tb:MaximalTorus3D4}
\end{table}
\renewcommand{\arraystretch}{1}

The action of the Weyl group $\W$ on the maximal torus $\T$ is given as follows:

\begin{lem}
\label{lem:3D4_WeylGroupActionOnTorus}
The generators $\omega_r$, $r \in \Pi$, of the Weyl group $\W$ of $\G$ act on the maximal torus $\T$ as follows:
\begin{align*}
\omega_{r_1} h(t_1, t_2, t_3, t_4) \omega_{r_1}^{-1} &= h(t_1^{-1} t_2, t_2, t_3, t_4),\\
\omega_{r_2} h(t_1, t_2, t_3, t_4) \omega_{r_2}^{-1} &= h(t_1, t_1 t_2^{-1} t_3 t_4, t_3, t_4),\\
\omega_{r_3} h(t_1, t_2, t_3, t_4) \omega_{r_3}^{-1} &= h(t_1, t_2, t_2 t_3^{-1}, t_4),\\
\omega_{r_4} h(t_1, t_2, t_3, t_4) \omega_{r_4}^{-1} &= h(t_1, t_2, t_3,  t_2t_4^{-1})
\end{align*}
for all $t_1, t_2, t_3, t_4 \in \mathbb{F}^\times$.
\end{lem}

\begin{proof}
Easy calculations using Theorem~\ref{thm:ChevalleyRelUni}. 
\end{proof}

Let us now take a closer look at the Weyl groups associated to the maximal tori of types $T_+$ and $T_-$, which will occur most frequently in the following.

\begin{lem}
\label{lem:3D4_GeneratorsFixWeylGroup}
Let $w \in \{ 1, \omega_0 \}$. Then $\omega_{r_2}$ and $\omega_{r_1} \omega_{r_3} \omega_{r_4}$ are fixed points of $\W$ under $wF$.
\end{lem}

\begin{proof}
Since $w \in \Z(\W)$, it suffices to show that $\omega_{r_2}$ and $\omega_{r_1} \omega_{r_3} \omega_{r_4}$ are fixed under the action of $F$. The element $\omega_{r_2}$ is identified with $n_{r_2}(1)\T$ and we have $$F(n_{r_2}(1)) = n_{\rho(r_2)}(1^q) = n_{r_2}(1)$$ by definition of $F$, so clearly $\omega_{r_2} \in \W^{wF}$. 
Again by definition of $F$ we moreover have
\begin{align*}
F(n_{r_1}(1) n_{r_3}(1) n_{r_4}(1)) &= n_{\rho(r_1)}(1) n_{\rho(r_3)}(1) n_{\rho(r_4)}(1)\\ &= n_{r_3}(1) n_{r_4}(1) n_{r_1}(1).
\end{align*}
Now the roots $r_1$, $r_3$ and $r_4$ are pairwise perpendicular, whence the corresponding reflections commute, i.e., $n_{r_3}(1) n_{r_4}(1) n_{r_1}(1)$ and $n_{r_1}(1) n_{r_3}(1) n_{r_4}(1)$ agree modulo $\T$. Consequently, $\omega_{r_1} \omega_{r_3} \omega_{r_4}$ is fixed by $F$, and hence by $wF$, as claimed.
\end{proof}

\begin{coro}
\label{coro:3D4_GeneratorsWeyl}
Let $w \in \{ 1, \omega_0 \}$. Then $\W^{wF} = \langle \omega_{r_2},\, \omega_{r_1} \omega_{r_3} \omega_{r_4} \rangle$.
\end{coro}

\begin{proof}
Following Lemma~\ref{lem:3D4_GeneratorsFixWeylGroup} we have $\omega_{r_2},\, \omega_{r_1} \omega_{r_3} \omega_{r_4} \in \W^{wF}$. Moreover, by direct calculations the element $\bar{\omega} := \omega_{r_1} \omega_{r_3} \omega_{r_4} \omega_{r_2} \in \W^{wF}$ has order $6$, and, naturally, the order of the reflection $\omega_{r_2}$ in $\W^{wF}$ is $2$. Direct calculations show that $\omega_{r_2}$ acts on $\langle \bar{w} \rangle$ by inversion. Hence, the claim follows since by Table~\ref{tb:MaximalTorus3D4} the group $\W^{wF}$ is dihedral of order $12$.
\end{proof}


\subsubsection{Automorphisms of $\trial$}
\label{sec:3D4_Automorphisms}

The automorphism group of the triality group $G=\trial$ is particularly easy to describe. 
By \cite[p.~158]{SteinbergLecture} the field automorphism $\mathbb{F} \longrightarrow \mathbb{F}$, $a \longmapsto a^p$, induces an automorphism $F_p$ of the group $\G = D_4(\mathbb{F})$ via
\begin{align*}
F_p \colon \G \longrightarrow \G,\quad x_r(t) \longmapsto x_r(t^p), \quad  r \in \Sigma,\; t \in \mathbb{F}.
\end{align*}
Since $F_p$ clearly commutes with $F$, it also induces an automorphism of $G = \G^F$, which will be called the \textit{field automorphism} $F_p$ of $G$. Note that its order in $\Aut(G)$ is given by $3f$, where $q = p^f$, since $G \leqslant D_4(q^3)$ as observed before.

\begin{prop}
\label{prop:3D4_AutomorphismGroup}
For $G = \trial$ with $q=p^f$, $f \in \mathbb{N}_{>0}$, we have
$$\Aut(G) = G \rtimes \langle F_p \rangle.$$
In particular, the outer automorphism group of $G$ is cyclic.
\end{prop}

\begin{proof}
This is a well-known statement, which may, for instance, be derived from \cite[Table~22.1]{malletest} together with \cite[Thm.~2.5.12(a),(d),(f)]{Classification}.
\end{proof}


\subsubsection{Special Subgroups of $\trial$}

Let us now introduce certain subgroups of $\trial$ which will play a role in the description and examination of the weights of $\trial$.

\begin{prop}
\label{prop:G2Subgroup3D4}
For $G = \trial$ we denote by $\widetilde{G} := G^{\tau}\, (= G^{F_q})$ the subgroup of fixed points of $G$ under $\tau$ (or equivalently under $F_q$). Then $\widetilde{G} \cong G_2(q)$.
\end{prop}

\begin{proof}
This is well known and follows from the fact that $G$ is generated by the groups $\X_S^F$, $S \in \widetilde{\Sigma}$, with
$$({\X_S^F})^{F_q} = \{ x_S(t) \mid t \in \mathbb{F}_q \},$$
and the generators $x_S(t)$, $S \in \widetilde{\Sigma}$, $t \in \mathbb{F}_q$, satisfy the same relations as for type $G_2$ (cf., e.g.,~\cite[Thm.~1.12.1, Table~2.4, Thm.~2.4.5 and Thm.~2.4.7]{Classification}).
\end{proof}

\begin{coro}
\label{coro:FpActionG2Sub}
Let $\widetilde{G}\cong G_2(q)$ be as in Proposition~\ref{prop:G2Subgroup3D4}. Then the field automorphism $F_p$ of $G$ acts on $\widetilde{G}$ as the field automorphism of $G_2(q)$ defined in Section~\ref{sec:AutomorphismsG2}.
\end{coro}

\begin{proof}
This is a consequence of the previous proposition as for $S \in \widetilde{\Sigma}$, $t \in \mathbb{F}_q$, we have $F_p(x_S(t)) = x_S(t^p)$, i.e.,~the operation of $F_p$ on the generators $x_S(t)$ is as for type $G_2$.
\end{proof}

We now consider further subgroups of $G$, some of which we have already encountered as subgroups of $G_2(q)$ in Section~\ref{sec:G2}. We stick to the notation used in \cite{AnWeights3D4}.
The subgroups of $G$ introduced now are denoted $K_\delta$ and $L_\delta$. These groups already appeared in Section~\ref{ssec:G2_Auto3Weights} as subgroups of $G_2(q)$.

\begin{defi}
We fix a maximal subgroup $\widetilde{G} \cong G_2(q)$ of $G$. Then following \cite[p.~275]{AnWeights3D4} for $\delta \in \{ \pm 1 \}$ there exist maximal subgroups $K_\delta$ of $\widetilde{G}$ such that these contain subgroups
\begin{align*}
L_\delta \cong \SL_3(\delta q) = \begin{cases}
\SL_3(q) & \text{if $\delta = +1$},\\
\SU_3(q) & \text{if $\delta = -1$},
\end{cases}
\end{align*}
and $K_\delta$ is the extension of $L_\delta$ by an involutory outer automorphism.
\end{defi}

Two further subgroups will play a role here. For the definition of these a classification of semisimple elements of $G$ is needed, for which we refer to \cite[Table~2.1]{DeriziotisMichler}.

\begin{defi}
Let $\delta \in \{ \pm 1 \}$ be such that $q \equiv \delta \bmod 3$ and let $s \in G$ be a semisimple element of type $s_4$ if $\delta = +1$ or of type $s_9$ if $\delta = -1$. Then as in \cite[pp.~274/275]{AnWeights3D4} we set $M_\delta := \norma_G(\langle  s \rangle)$ and $H_\delta := \cent_G(s)$.
\end{defi}

\begin{rmk}
By \cite[p.~275]{AnWeights3D4} it holds that $M_\delta / H_\delta \cong C_2$. Moreover, as done in \cite{AnWeights3D4} we may assume that the semisimple element $s$ in the previous definition is chosen such that $K_\delta \leqslant M_\delta$ and $L_\delta \leqslant H_\delta$.
\end{rmk}


\subsection{Action of Automorphisms}
\label{ch:3D4_Auto3D4}

Similarly as for the groups $G_2(q)$, easy arguments relying on irreducible character values of $G$ and the unitriangularity of the corresponding decomposition matrices yield the following result, see \cite[Prop.~16.1, 16.2]{Diss}:

\begin{prop}
\label{prop:3D4_Auto2BrauerChars}
Let $B$ be an $\ell$-block of $G$ for $\ell \in \{2, 3\}$ and assume that $\ell \nmid q$. Then any Brauer character in $\IBr(B)$ is left invariant by $\Aut(G)_B$.
\end{prop}

The study of the action of $\Aut(G)$ on the conjugacy classes of the $\ell$-weights in $G$, $\ell \in \{ 2, 3 \}$, is conducted blockwise, the main distinction being made between $\ell$-blocks of abelian defect and $\ell$-blocks of non-abelian defect.\vspace{\baselineskip}


\paragraph{\it The Case $\ell=2$}

Throughout this section we assume that $\ell =2 \nmid q$ and let $\varepsilon \in \{ \pm 1 \}$ be such that $q \equiv \varepsilon \bmod 4$.

The $2$-weights of $G$ for the principal $2$-block $B_0$ have been described by An \cite{AnWeights3D4} as in Proposition~\ref{prop:3D4_2WeightsB1} below. Note that there are significant parallels to the situation of the principal $2$-block of $G_2(q)$ (cf.~Proposition~\ref{prop:2WeightsB1}). 

\begin{prop}
\label{prop:3D4_2WeightsB1}
Suppose that $B = B_0$ is the principal $2$-block of $G$. Then $|\mathcal{W}(B)|=7$.
Moreover, if $(R, \varphi)$ is a $B$-weight of $G$, then up to $G$-conjugation one of the following holds:
\begin{enumerate}[(i)]
\item $R \cong (C_2)^3$, $\norma_G(R)/R \cong \GL_3(2)$, and $\varphi$ is the inflation of
the Steinberg character of $\GL_3(2)$. There exists exactly one $G$-conjugacy class of such $B$-weights in $G$.

\item $R = \langle \mathcal{O}_2(T_\varepsilon), \rho \rangle$ for some involution $\rho \in \norma_G(T_\varepsilon)$, $\norma_G(R) / R \cong \mathfrak{S}_3$, and $\varphi$  is the inflation of the unique irreducible character of $\mathfrak{S}_3$ of degree $2$. There exists exactly one $G$-conjugacy class of such $B$-weights in $G$.

\item $R \in \Syl_2(G)$ is a Sylow $2$-subgroup of $G$, $\norma_G(R) = R$, and $\varphi$ is the trivial character of $\norma_G(R)$. There exists exactly one $G$-conjugacy class of such $B$-weights in $G$.

\item $q \equiv \pm 1 \bmod 8$, $R \cong 2^{1+2}_+ \circ D_{(q^2-1)_2}$, $\norma_G(R)/R \cong \mathfrak{S}_3$,
and $\varphi$  is the inflation of the unique irreducible character of $\mathfrak{S}_3$ of degree $2$. There exist exactly two $G$-conjugacy classes of such $B$-weights in $G$.

\item $q \equiv \pm 1 \bmod 8$, $R \cong 2^{1+4}_+$, $\norma_G(R)/R \cong \mathfrak{S}_3 \times \mathfrak{S}_3$,
and $\varphi$  is the inflation of the unique irreducible character of $\mathfrak{S}_3 \times \mathfrak{S}_3$ of degree $4$. There exist exactly two $G$-conjugacy classes of such $B$-weights in $G$.

\item $q \equiv \pm 3 \bmod 8$, $R \cong 2^{1+4}_+$, $\norma_G(R)/R \cong (C_3 \times C_3) \rtimes C_2$ with $C_2$ acting on $C_3 \times C_3$ by inversion, and $\varphi$  is the inflation of one of the four irreducible characters of $\norma_G(R)/R$ of degree 2. There exists exactly one $G$-conjugacy class of such $R$ in $G$.
\end{enumerate}
\end{prop}

\begin{proof}
This follows from \cite[(2B)--(2D), (3C)]{AnWeights3D4} and the proofs of \cite[(3G)]{AnWeights3D4} and \cite[(3H)]{AnG2Weights}.
\end{proof}

\begin{rmk}
\label{rmk:ContainedInG2}
If $(R, \varphi)$ is a $B_0$-weight for $G$ with $R$ as in (iv), (v) or (vi) of Proposition~\ref{prop:3D4_2WeightsB1}, then as in the proof of \cite[(3C)]{AnWeights3D4} we may assume that $R$ is contained in a maximal subgroup $\widetilde{G}$ of $G$ isomorphic to $G_2(q)$ such that 
$$\norma_{\widetilde{G}}(R) \leqslant \norma_{\widetilde{G}}(\Z(R)) = \cent_{\widetilde{G}}(\Z(R)) \cong \SO_4^+(q)$$
and $R$ is $2$-radical in $\cent_{\widetilde{G}}(\Z(R)) \cong \SO_4^+(q)$ with $\norma_{\widetilde{G}}(R)/R \cong \norma_G(R)/R$ by \cite[(2B)]{AnG2Weights}. Thus, we observe that $\norma_G(R)$ is already contained in \smash{$\widetilde{G}$}.
\end{rmk}

\begin{prop}
\label{prop:3D4_Auto2WeightsB1}
Let $B = B_0$ be the principal $2$-block of $G$. Then $\Aut(G)_B = \Aut(G)$ acts trivially on $\mathcal{W}(B)$.
\end{prop}

\begin{proof}
Let $(R, \varphi)$ be a $B$-weight. If $R$ is as in Proposition~\ref{prop:3D4_2WeightsB1}(i), (ii) or (iii), then the $G$-conjugacy class of $(R, \varphi)$ is uniquely determined by the isomorphism type of $R$. Hence, it stays invariant under $\Aut(G)$.

Suppose now that $R$ is as in (iv), (v) or (vi) of Proposition~\ref{prop:3D4_2WeightsB1}. Following Remark~\ref{rmk:ContainedInG2} we may assume that $\norma_G(R)$ is contained in the subgroup $\widetilde{G} \cong G_2(q)$, where $\widetilde{G}$ is as in Proposition~\ref{prop:G2Subgroup3D4}. According to Proposition~\ref{prop:3D4_AutomorphismGroup} we have $\Aut(G) = \langle G, F_p \rangle$ with $F_p$ stabilizing $\widetilde{G}$ and acting on it as a field automorphism of $G_2(q)$ by Corollary~\ref{coro:FpActionG2Sub}. Now let $a$ be an element of $\Aut(G)$. 
Inner automorphisms of $G$ do not play a role here, i.e.,~we may assume that $a = F_p^i$ for some $i \in \mathbb{N}$. Now  $R \leqslant \norma_G(R) \leqslant \widetilde{G}$, so we may regard $(R, \varphi)$ as a $2$-weight for $\widetilde{G}$ and $a=F_p^i$ as a power of the field automorphism of \smash{$\widetilde{G}$}. Hence, by Proposition~\ref{prop:Auto2WeightsB1} the $2$-weight $(R, \varphi)$ is stabilized by $a$ up to $\widetilde{G}$-conjugation.
\end{proof}

For non-principal $2$-blocks of $G$ of non-abelian defect An \cite[(3G)]{AnWeights3D4} showed the following:

\begin{prop}
\label{prop:3D4_2WeightsNonAb}\label{prop:3D4_Auto2WeightsNonAb}
Let $B$ be a non-principal $2$-block of $G$ with non-abelian defect groups. Then $|\mathcal{W}(B)| \in \{2, 3\}$. Moreover, if $(R_1, \varphi_1)$ and $(R_2, \varphi_2)$ are non-conjugate $B$-weights, then one has $R_1 \not\cong R_2$.
In particular, $\Aut(G)_B$ acts trivially on $\mathcal{W}(B)$.
\end{prop}

Let us now turn towards the abelian defect case.

\begin{prop}
\label{prop:3D4_2noncyclicabelianblocks}
Suppose that $B$ is a $2$-block of $G$ which has a non-cyclic abelian defect group $D$. Then the following statements hold:
\begin{enumerate}[(i)]
\item The centralizer $\cent_G(D) =: T$ is a maximal torus of $G$ of type $T_\varepsilon$, $D = \mathcal{O}_2(T)$ and $\norma_G(D) = \norma_G(T)$ with $\norma_G(T)/T \cong D_{12}$.
\item Consider $\Irr(T)$ as an abelian group and fix an isomorphism $\hat{\phantom{s}} \colon T \longrightarrow \Irr(T)$. Up to $G$-conjugation there exists a unique $2'$-element $s \in T$ and a $2$-block $b$ of $T = \cent_G(D)$ with $b^G = B$ such that $\theta := \hat{s} \in \Irr(T)$ is the canonical character of $b$.
\item For $\theta$ and $s$ as in (ii) we have $\norma_G(T)_\theta / T \cong \cent_{W(T)}(s)$, where $W(T) := \norma_G(T) / T$.
\end{enumerate}
\end{prop}

\begin{proof}
For (i) we observe that according to \cite[Prop.~5.8]{DeriziotisMichler} the centralizer $\cent_G(D) =: T$ is a maximal torus of $G$ such that $D = \mathcal{O}_2(T)$.
Since maximal tori of $G$ of types $T_{1, \pm}$ and $T_3$ are cyclic and maximal tori of types $T_{2, \pm}$ have odd order, it follows that $T$ must be of type $T_\delta$ for some $\delta \in \{ \pm 1 \}$. But by \cite[(2D)(a)]{AnWeights3D4} the centralizer of $\mathcal{O}_2(T_\delta)$ in $G$ is a maximal torus if and only if $\delta = \varepsilon$, so $T$ must be of type $T_\varepsilon$. Moreover, by the same reference it follows that $\norma_G(D) = \norma_G(T)$ with $\norma_G(T)/T \cong D_{12}$. 

Parts (ii) and (iii) can be found as parts (b) and (c) of Proposition~5.8 in \cite{DeriziotisMichler}.
\end{proof}

\begin{prop}
\label{prop:3D4_2BlockAbelianWeight}
Let $B$ be a $2$-block of $G$ with non-cyclic abelian defect groups. Then $\Aut(G)_B$ acts trivially on $\mathcal{W}(B)$.
\end{prop}

\begin{proof}
Suppose that $(R, \varphi)$ is a $B$-weight. Since $B$ has abelian defect, it follows from Lemma~\ref{lem:AbelianDefectWeight} that $R$ is a defect group of $B$. Hence, due to Proposition~\ref{prop:3D4_2noncyclicabelianblocks}(i) we have $R = \mathcal{O}_2(T)$ for some maximal torus $T$ of $G$ of type $T_\varepsilon$ and $\norma_G(R) = \norma_G(T)$. Let $s \in T$, $\theta = \hat{s}$ and $b \in \Block_2(T)$ be as in Proposition~\ref{prop:3D4_2noncyclicabelianblocks}(ii) such that $\norma_G(T)_\theta / T \cong \cent_{W(T)}(s)$ for $W(T) = \norma_G(T) / T$.

Since $B$ is non-principal and $s$ is of odd order, it follows that $s^2 \neq 1$. Hence, according to \cite[Table~3.4]{DeriziotisMichler} up to isomorphism we have $\cent_{W(T)}(s) \in \{ \{1\}, C_2, \mathfrak{S}_3 \}$. 
However, $b$ has defect group $R$, and since $B = b^G$ has defect group $R$ as well, \cite[Thm.~9.22]{navarro} implies that the index $|\norma_G(T)_\theta \colon T|$ is prime to $2$. We conclude that $\norma_G(T)_\theta = T$.

Following Construction~\ref{constr:B-weights} we hence have $\varphi = \theta^{\norma_G(T)}$, so $\varphi$ is uniquely determined by $R$ and $B$.
If $a \in \Aut(G)_B$, then as above the first component $R^a$ of the $B$-weight $(R, \varphi)^a$ must be a defect group of $B$, so in particular it is $G$-conjugate to $R$.
\end{proof}


\paragraph{\it The Case $\ell=3$}
\label{ssec:3D4_Auto3Weights}

Throughout this section we assume that $3 \nmid q$ and denote by $\varepsilon$ the unique element in $\{ \pm 1 \}$ with $q \equiv \varepsilon \bmod 3$.
If $T$ is a maximal torus of $G$ of type $T_\varepsilon$, then in consequence of \cite[(1A)]{AnWeights3D4} it holds that $\norma_G(T)/T \cong D_{12}$. We let $\widetilde{G} \cong G_2(q)$ be as in Proposition~\ref{prop:G2Subgroup3D4} and assume that $T$ is such that $\widetilde{T} := T \cap \widetilde{G}$ is a maximal torus of $\widetilde{G}$ isomorphic to $C_{q-\varepsilon} \times C_{q-\varepsilon}$. Then as observed in the case $\ell=3$ for $G_2(q)$ it also holds that $\norma_{\widetilde{G}}(\widetilde{T}) / \widetilde{T} \cong D_{12}$. This allows us to prove the following two statements:

\begin{lem}
\label{lem:3D4_CentralizerTTilde}
Suppose that $\widetilde{G} \cong G_2(q)$ is as in Proposition~\ref{prop:G2Subgroup3D4} and let $T$ be a maximal torus of $G$ of type $T_\varepsilon$ such that $\widetilde{T} := T \cap \widetilde{G}$ is a maximal torus of $\widetilde{G}$ with $\widetilde{T} \cong C_{q-\varepsilon} \times C_{q-\varepsilon}$. Then it holds that $\cent_G(\widetilde{T}) = T$.
\end{lem}

\begin{proof}
Since $T \cong C_{q^3-\varepsilon} \times C_{q-\varepsilon}$, there exist exactly $8$ elements of order $3$ in $T$, all of which already lie in $\widetilde{T}$ and centralize each other. Following Table~\ref{tb:MaximalTorus3D4} some of these are of the form $x=h(z, z^{-1}, z, z)$ with $z \in \mathbb{F}^\times$ of order $3$, and by \cite[Tables~2.2a, 2.2b]{DeriziotisMichler} these have $\cent_G(x)/Z \cong \PSL_2(q^3)$ or $\cent_G(x)/Z \cong \PGL_2(q^3)$ for some central subgroup $Z \cong C_{q-\varepsilon}$ of $\cent_G(x)$ with $x \in Z$, depending on whether $q$ is even or odd. Then $x, x^{-1} \in Z$, and the residue class of any other element of $T$ of order $3$ is non-trivial in $\cent_G(x)/Z$. 
Hence, for $y \in T$ of order $3$ with $y \neq x, x^{-1}$ we have $|\cent_{\cent_G(x)}(y)| = (q^3-\varepsilon)(q-\varepsilon) = |T|$, so the claim follows since $T \subseteq \cent_G(\widetilde{T}) \subseteq \cent_G(x) \cap \cent_G(y) = \cent_{\cent_G(x)}(y)$.
\end{proof}

\begin{prop}
\label{prop:3D4_MaximalTorusRestrG2}
In the situation of Lemma~\ref{lem:3D4_CentralizerTTilde} we have $\norma_G(T) = \norma_{\widetilde{G}}(\widetilde{T})T$.
\end{prop}

\begin{proof}
We prove that $\norma_{\widetilde{G}}(\widetilde{T})$ normalizes $T$. Then  $\norma_{\widetilde{G}}(\widetilde{T})T \subseteq \norma_G(T)$, and since $\norma_G(T) / T$ and $\norma_{\widetilde{G}}(\widetilde{T}) / \widetilde{T}$ are isomorphic to $D_{12}$ and
\begin{align*}
\norma_{\widetilde{G}}(\widetilde{T})T / T \cong \norma_{\widetilde{G}}(\widetilde{T}) / (\norma_{\widetilde{G}}(\widetilde{T}) \cap T) = \norma_{\widetilde{G}}(\widetilde{T}) / \widetilde{T} \cong D_{12},
\end{align*}
we obtain equality. By Lemma~\ref{lem:3D4_CentralizerTTilde} we have $\cent_G(\widetilde{T}) = T$, so every element of $G$ which normalizes $\widetilde{T}$ stabilizes $\cent_G(\widetilde{T}) = T$. In particular, $\norma_{\widetilde{G}}(\widetilde{T}) \subseteq \norma_G(T)$ as claimed.
\end{proof}

We first consider the $3$-blocks of $G$ with non-abelian defect. For the principal $3$-block $B_0$ the following has been shown to hold by An, see \cite[(1A)]{AnWeights3D4} and the proof of \cite[(3B)]{AnWeights3D4}:

\begin{prop}
\label{prop:3D4_3WeightsB1}
Suppose that $B = B_0$ is the principal $3$-block of $G$. Then $|\mathcal{W}(B)| = 7$. Moreover, if $(R, \varphi)$ is a $B$-weight of $G$, then up to $G$-conjugation one of the following holds:
\begin{enumerate}[(i)]
\item $R$ is an extraspecial group $3^{1+2}_+$ such that $R \leqslant L_\varepsilon \leqslant H_\varepsilon$, $\cent_G(R) = \Z(H_\varepsilon) \cong C_{q^2+\varepsilon q +1}$, $\norma_G(R) = \norma_{M_\varepsilon}(R)$, $\norma_{H_\varepsilon}(R)/R\cent_G(R) \cong \Sp_2(3)$, $|\norma_G(R) \colon \norma_{H_\varepsilon}(R)|=2$, and $\varphi$ is the inflation of one of the two extensions of the Steinberg character of $\Sp_2(3)$ to $\norma_G(R)/R\cent_G(R)$.
There exists exactly one $G$-conjugacy class of such $R$ in $G$.

\item $R \in \Syl_3(G)$, $\cent_G(R) \cong C_{q^2+\varepsilon q +1}$, $\norma_G(R)/R\cent_G(R) \cong C_2 \times C_2$, and $\varphi$ is the inflation of one of the four irreducible characters of $\norma_G(R)/R\cent_G(R)$.

\item $R = \mathcal{O}_3(T)$ for a maximal torus $T$ of $G$ of type $T_{2, \varepsilon}$, $\norma_G(R)/R\cent_G(R) \cong \Sp_2(3)$ with $\cent_G(R) = T$,  and $\varphi$ is the inflation of the Steinberg character of $\norma_G(R)/R\cent_G(R)$.
There exists exactly one $G$-conjugacy class of such $B$-weights in $G$.
\end{enumerate}
\end{prop}

In order to understand the action of $\Aut(G)$ on the $3$-weights associated to the principal $3$-block of $G$ we need the following observation given by \cite[(1C)]{AnWeights3D4}:

\begin{lem}
\label{lem:3D4_ExtraLeps}
Let $R \leqslant G$ be an extraspecial group $3^{1+2}_+$. Then there exists exactly one $G$-conjugacy class of subgroups of $G$ isomorphic to $R$, and if $R \leqslant L_\varepsilon \leqslant K_\varepsilon \leqslant M_\varepsilon$, then 
$$
\norma_G(R) / R \cent_G(R) \cong \Sp_2(3) \rtimes \langle \rho \rangle
$$
for some involution $\rho \in K_\varepsilon \setminus L_\varepsilon$ such that $M_\varepsilon = \langle H_\varepsilon, \rho \rangle$.
\end{lem}

\begin{lem}
\label{lem:3D4_TrivialActionNormaGTT}
Let $T$ be a maximal torus of $G$ of type $T_\varepsilon$. Then there exists an automorphism $a \in \Aut(G)_T$ with $\Out(G) = \langle a \Inn(G) \rangle$ which acts trivially on $\norma_G(T) / T$.
\end{lem}

\begin{proof}
Let $\widetilde{G} \cong G_2(q)$ be as in Proposition~\ref{prop:G2Subgroup3D4}. If necessary, after replacing $T$ by a $G$-conjugate we may assume that $\widetilde{T}:= T \cap \widetilde{G}$ is a maximal torus of $\widetilde{G}$ isomorphic to $C_{q-\varepsilon} \times C_{q- \varepsilon}$. By Remark~\ref{rmk:G2_FieldAuto}(ii) in conjunction with the fact that $\norma_{\widetilde{G}}(\widetilde{T})/\widetilde{T} \cong D_{12}$, and possibly after $\widetilde{G}$-conjugation, there exists an automorphism $\tilde{a}$ of $\widetilde{G}$ (e.g.~the field automorphism of $\widetilde{G}$) such that
$\Out(\widetilde{G}) = \langle \tilde{a} \Inn(\widetilde{G}) \rangle$,
$\tilde{a}(\widetilde{T}) = \widetilde{T}$, and
every element of $\norma_{\widetilde{G}}(\widetilde{T})$ can be written as $n t$ with $t \in \widetilde{T}$ and $n \in \norma_{\widetilde{G}}(\widetilde{T})$ such that $\tilde{a}(n)=n$.

Now the field automorphism $F_p$ acts on $\widetilde{G}$ and hence, since $\Out(\widetilde{G}) = \langle \tilde{a} \Inn(\widetilde{G}) \rangle$, there is some $k \in \mathbb{N}$ and $g \in \widetilde{G}$ such that $F_p$ acts as $\tilde{a}^k c_g$ on $\widetilde{G}$. Then $a :=F_p c_{g^{-1}}$ fixes $\widetilde{G}$ and $\widetilde{T}$, and hence $T$ since $T = \cent_G(\widetilde{T})$ by Lemma~\ref{lem:3D4_CentralizerTTilde}. Moreover, $\Out(G) = \langle a \Inn(G) \rangle$.

Now let $x \in \norma_G(T)$. Following Proposition~\ref{prop:3D4_MaximalTorusRestrG2} it holds that $\norma_G(T) = \norma_{\widetilde{G}}(\widetilde{T}) T$, and hence $x = ms$ for some $m \in \norma_{\widetilde{G}}(\widetilde{T})$ and $s \in T$. Now we have $m = n t$ for suitable elements \smash{$t \in \widetilde{T}$} and \smash{$n \in \norma_{\widetilde{G}}(\widetilde{T})$} such that $\tilde{a}(n)=n$. But then $x = nts \in nT$ and $$a(x) = a(nts) = \tilde{a}^k(n) a(ts) = n a(ts) \in nT,$$ so $a(x)$ and $x$ coincide modulo $T$, i.e.,~$a$ acts trivially on $\norma_G(T)/T$.
\end{proof}

\begin{prop}
\label{prop:3D4_Auto3WeightsB1}
Let $B = B_0$ be the principal $3$-block of $G$.  Then $\Aut(G)_B = \Aut(G)$ acts trivially on $\mathcal{W}(B)$.
\end{prop}

\begin{proof}
We go through the distinct cases in Proposition~\ref{prop:3D4_3WeightsB1}.
Let $(R, \varphi)$ be a $B$-weight of $G$. If $R = \mathcal{O}_3(T)$ for a maximal torus $T$ of $G$ of type $T_{2,\varepsilon}$, then by Proposition~\ref{prop:3D4_3WeightsB1}(iii) the $G$-conjugacy class of $(R, \varphi)$ is uniquely determined by $B$ and the isomorphism type of $R$. Hence, it stays invariant under $\Aut(G)$.

Suppose now that $R$ is as in Proposition~\ref{prop:3D4_3WeightsB1}(i).
Then by Lemma~\ref{lem:3D4_ExtraLeps} the group $R$ is uniquely determined up to $G$-conjugation.
Hence, we may assume that $R$ is contained in the maximal subgroup of $G$ isomorphic to $G_2(q)$ as in Proposition~\ref{prop:G2Subgroup3D4} and has the same form as in the proof of Proposition~\ref{prop:Auto3WeightsB1} for case (iii). As in that proof, the element $\rho$ in Lemma~\ref{lem:3D4_ExtraLeps} can be chosen such that $\rho \in \norma_{K_\varepsilon}(R)$ and $F_p$ acts trivially on $\rho$. Moreover, $R$ is stabilized by $F_p$. According to Proposition~\ref{prop:3D4_3WeightsB1}(i) the character $\varphi$ corresponds to one of the two extensions $\St^\pm$ of the Steinberg character of $\norma_{H_\varepsilon}/R\cent_G(R) \cong \Sp_2(3)$ to $\norma_G(R)/R \cent_G(R)$.
 As in the proof of Proposition~\ref{prop:Auto3WeightsB1} one shows that $F_p$ leaves both $\St^+$ and $\St^-$ invariant. In particular, the $G$-conjugacy class of $(R, \varphi)$ is fixed by $\Aut(G)$.

Finally, we suppose that $R \in \Syl_3(G)$. Then from Proposition~\ref{prop:3D4_3WeightsB1} it is known that $$\norma_G(R)/ R\cent_G(R) \cong C_2 \times C_2.$$ We prove that there exists $a \in \Aut(G)_R$ which acts trivially on $\norma_G(R)/ R\cent_G(R)$ and with $\Out(G) = \langle a \Inn(G) \rangle$. 
Following the proof of \cite[(1A)]{AnWeights3D4} we may assume that $\norma_G(R)$ is contained in $\norma_G(T)$ for some maximal torus $T$ of type $T_\varepsilon$, and where in addition $\norma_G(T) = \langle T, \rho, \tau, \sigma \rangle$ for suitable $\rho, \tau, \sigma \in G$ such that $R = \langle \mathcal{O}_3(T), \sigma \rangle$, $\sigma$ has order $3$ modulo $T$, $\rho T$ generates the center of $\norma_G(T)/T \cong D_{12}$, and $\tau^{-1} \sigma \tau$ coincides with  $\sigma^{-1}$ modulo $T$. Moreover, by the same reference $\cent_G(R) \subseteq T$, $\norma_G(R) = \langle R\cent_G(R), \rho, \tau \rangle$, and we have $$\norma_G(R)/R\cent_G(R) = \langle \rho R\cent_G(R) \rangle \times \langle \tau R\cent_G(R) \rangle.$$
By Lemma~\ref{lem:3D4_TrivialActionNormaGTT} there exists $\tilde{a} \in \Aut(G)_{T}$ such that $\Out(G) = \langle \tilde{a} \Inn(G) \rangle$ and $\tilde{a}$ acts trivially on $\norma_G(T)/T$. Since $R^{\tilde{a}} \leqslant \norma_G(T)$ and $R$ is a Sylow $3$-subgroup of $G$ and hence of $\norma_G(T)$, there exists $y \in \norma_G(T)$ such that $R^{\tilde{a}y} = R$, i.e.,~the automorphism $a := \tilde{a} c_y \in \Aut(G)$ stabilizes $R$, so in particular it acts on $\norma_G(R)/ R\cent_G(R)$. Moreover, it holds that $$\langle a \Inn(G) \rangle = \langle \tilde{a} \Inn(G) \rangle = \Out(G).$$ Since $\norma_G(T) = \langle T, \rho, \tau, \sigma \rangle$ and $\rho, \tau, \sigma \in \norma_G(R)$, we may assume that $y \in T$.
We prove that $a$ acts trivially on 
$\norma_G(R)/ R\cent_G(R)$. Since $\rho, \tau \in \norma_G(T)$, we have $$c_y(\rho) = y \rho y^{-1} = \rho \rho^{-1} y \rho y^{-1} \in \rho T,$$ and analogously $c_y(\tau) \in \tau T$. Hence, it follows that $a(\rho) \in \rho T$ and $a(\tau) \in \tau T$ since $\tilde{a}$ acts trivially on $\norma_G(T)/T$.

Let us suppose that $a(\rho R\cent_G(R)) = \tau R\cent_G(R)$. Then $a(\rho) = \tau x \sigma^i$ for some $x \in T$ and $i \in \{ 0, 1, 2 \}$ since $R = \langle \mathcal{O}_3(T), \sigma \rangle$ and $\cent_G(R) \subseteq T$. It follows that $\rho \equiv a(\rho) \equiv \tau \sigma^i \bmod T,$ i.e.,~$\tau \equiv \rho \sigma^{-i} \bmod T$. Since $\rho T \in \Z(\norma_G(T)/T)$, we conclude that $$\sigma^{-1} \equiv \tau^{-1} \sigma \tau \equiv \sigma^i \rho^{-1} \sigma \rho \sigma^{-i} \equiv \sigma \bmod T$$ in contradiction to the order of $\sigma$ being $3$ modulo $T$.
Similar arguments prove that $a$ stabilizes every element of $\norma_G(R)/R\cent_G(R)$, and hence the weight $(R, \varphi)$.
\end{proof}

For non-principal $3$-blocks of $G$ with non-abelian defect groups An showed the following:

\begin{prop}
\label{prop:3D4_3WeightsNonAb}
Let $B$ be a non-principal $3$-block of $G$ with non-abelian defect groups. Then $|\mathcal{W}(B)|=3$. Moreover, if $(R, \varphi)$ is a $B$-weight of $G$ and $\theta$ is an irreducible constituent of $\varphi_{|R \cent_G(R)}$, then $\theta$ is linear and up to $G$-conjugation one of the following holds:
\begin{enumerate}[(i)]
\item $R \cong 3^{1+2}_+$ such that $R \leqslant L_\varepsilon \leqslant H_\varepsilon$, $\cent_G(R) = \Z(H_\varepsilon) \cong C_{q^2+\varepsilon q +1}$, $\norma_G(R) = \norma_{M_\varepsilon}(R)$, $\norma_{H_\varepsilon}(R)/R\cent_G(R) \cong \Sp_2(3)$, $|\norma_G(R) \colon \norma_{H_\varepsilon}(R)|=2$,  ${\norma_G(R)}_\theta = \norma_{H_\varepsilon}(R)$ and $$\varphi = (\widetilde{\theta} \cdot \St)^{\norma_G(R)},$$ where $\St$ denotes the inflation of the Steinberg character of $\norma_{H_\varepsilon}(R)/R\cent_G(R)$ and $\widetilde{\theta}$ is an extension of $\theta$ to ${\norma_G(R)}_\theta$. The character $\varphi$ is uniquely determined by $R$ and $B$, and there exists exactly one $G$-conjugacy class of such $B$-weights in $G$.
\item $R \in \Syl_3(G)$ with $R \leqslant H_\varepsilon$, $\cent_G(R) = \Z(H_\varepsilon) \cong C_{q^2+\varepsilon q +1}$, $\norma_{H_\varepsilon}(R)/R\cent_G(R) \cong C_2$, $\norma_G(R)/R\cent_G(R) \cong C_2 \times C_2$, ${\norma_G(R)}_\theta = \norma_{H_\varepsilon}(R)$, and $$\varphi = (\widetilde{\theta} \cdot \xi)^{\norma_G(R)},$$ where $\xi \in \Irr({\norma_G(R)}_\theta/R\cent_G(R))$ and $\widetilde{\theta}$ is an extension of $\theta$ to ${\norma_G(R)}_\theta$.
\end{enumerate}
\end{prop}

\begin{proof}
This follows from \cite[(1A)]{AnWeights3D4} and the proof of \cite[(3B)]{AnWeights3D4}.
\end{proof}

In order to understand the action of $\Aut(G)$ on the $3$-weights of $G$ described in Proposition~\ref{prop:3D4_3WeightsNonAb} above we need the following lemma:

\begin{lem}
\label{lem:3D4_H_eps_norma}
Let $R$ be as in Proposition~\ref{prop:3D4_3WeightsNonAb}(ii). Then $H_\varepsilon = U \times \Z(H_\varepsilon)_{3'}$ for some subgroup $U \leqslant H_\varepsilon$ such that $\norma_{H_\varepsilon}(R) = \norma_U(R) \times \Z(H_\varepsilon)_{3'}$. More precisely, $$\norma_{H_\varepsilon}(R) = \langle R, x \rangle \times \Z(H_\varepsilon)_{3'}$$ for some $x \in U$ with $x^2 \in R$.
\end{lem}

\begin{proof}
By the proof of \cite[(3B)]{AnWeights3D4} there exists a subgroup $U \leqslant H_\varepsilon$ such that we have $H_\varepsilon = U \times \Z(H_\varepsilon)_{3'}$ and $\norma_{H_\varepsilon}(R) = \norma_U(R) \times \Z(H_\varepsilon)_{3'}$. Now by Proposition~\ref{prop:3D4_3WeightsNonAb}(ii) it holds that $\cent_{H_\varepsilon}(R) = \Z(H_\varepsilon)$ and $\norma_{H_\varepsilon}(R) / R \cent_{H_\varepsilon}(R) \cong C_2$. Since $R$ is a Sylow $3$-subgroup of $G$ and hence of $H_\varepsilon$, we have $\Z(H_\varepsilon)_3 \leqslant R$. Thus, $R \cent_{H_\varepsilon}(R) = R \Z(H_\varepsilon) = R \times \Z(H_\varepsilon)_{3'}$, so $$\norma_U(R)/R \cong \norma_{H_\varepsilon}(R) / R \cent_{H_\varepsilon}(R) \cong C_2,$$ which proves the claim.
\end{proof}

\begin{prop}
\label{prop:3D4_Auto3WeightsNonAb}
Suppose that $B$ is a non-principal $3$-block of $G$ of non-abelian defect. Then $\Aut(G)_B$ acts trivially on $\mathcal{W}(B)$.
\end{prop}

\begin{proof}
Let $(R, \varphi)$ be a $B$-weight of $G$. We go through the cases in Proposition~\ref{prop:3D4_3WeightsNonAb}. If $R \cong 3^{1+2}_+$, then by Proposition~\ref{prop:3D4_3WeightsNonAb}(i) $\Aut(G)_B$ fixes $[(R, \varphi)]_G$.

Suppose now that $R \in \Syl_3(G)$ such that $R \leqslant H_\varepsilon$ and let $\theta$ be an irreducible constituent of $\varphi_{|R \cent_G(R)}$. Then by Proposition~\ref{prop:3D4_3WeightsNonAb}(ii) the character $\varphi$ is of the form $\eta^{\norma_G(R)}$ for some extension $\eta$ of $\theta$ to $\norma_G(R)_\theta$. Let $a \in \Aut(G)$ be such that $R^a = R$ and $B^a = B$. Moreover, let $b$ be the $3$-block of $R \cent_G(R)$ containing $\theta$. Then $b$ has defect group $R\in \Syl_3(G)$, and hence by \cite[Lemma~4.13]{navarro} also $B \in \Block_3(G \mid R)$. Since $a$ fixes $B$, it follows that $b^a$ induces $B$, so $b$ and $b^a$ of $R \cent_G(R)$ must be conjugate under $\norma_G(R)$ by Brauer's extended first main theorem \cite[Thm.~9.7]{navarro}, and hence also their canonical characters $\theta$ and $\theta^a$ are $\norma_G(R)$-conjugate. Without loss of generality we may thus assume that $\theta^a = \theta$.
Let $x \in H_\varepsilon$ be as in Lemma~\ref{lem:3D4_H_eps_norma}, i.e., such that $\norma_{H_\varepsilon}(R) = \langle R, x \rangle \times \Z(H_\varepsilon)_{3'}$ and $x^2 \in R$.
From Proposition~\ref{prop:3D4_3WeightsNonAb}(ii) it follows that $\norma_G(R)_\theta = \norma_{H_\varepsilon}(R)$, so $$a^{-1}(x) =x^a \in (\norma_G(R)_\theta)^a = \norma_G(R)_{\theta^a} = \norma_G(R)_{\theta} = \langle R, x \rangle \times \Z(H_\varepsilon)_{3'}.$$ Suppose that $x^a = v \cdot z$ for some $v \in \langle R, x \rangle$ and $z \in \Z(H_\varepsilon)_{3'}$. Then $$v^2 \cdot z^2 = (x^a)^2 = (x^2)^a \in R^a = R.$$ Since $v^2 \in R$, we also have $z^2 \in R$. But as $\Z(H_\varepsilon) \cong C_{q^2+\varepsilon q +1}$ (cf.~Proposition~\ref{prop:3D4_3WeightsNonAb}(ii)), it follows that $\Z(H_\varepsilon)_{3'}$ has odd order. Hence, we conclude that $z^2$ can only be contained in the $3$-group $R$ if $z=1$. Thus, $x^a \in \langle R, x \rangle$, so $a^{-1}$ and $a$ act on $\langle R, x \rangle$, and we may write $a(x) = rx$ for some $r \in R$.

Now there exist exactly two extensions of $\theta$ to $\norma_G(R)_\theta$, say $\theta_+$ and $\theta_-$ with $\theta_+ (x) = 1$ and $\theta_- (x) = -1$ (for this recall that $x^2 \in R$ and $\theta$ is a linear character of $R \cent_G(R)$ with $R$ in its kernel). 
We write $\theta_\pm(x) = \pm 1$.  
It holds that
$$(\theta_\pm)^a(c) = \theta_\pm(a(c)) = \theta(a(c)) = \theta^a(c) = \theta(c)$$
for all $c \in R \cent_G(R)$, so $(\theta_\pm)^a$ are extensions of $\theta$ to $\norma_G(R)_\theta$. Moreover,
$$(\theta_\pm)^a(x) = \theta_\pm(a(x)) = \theta_\pm(rx) = \theta_\pm (r) \theta_\pm(x) =  \theta_\pm(x) = \pm 1,$$
so we have $(\theta_\pm)^a = \theta_\pm$. Denote by $\psi_+$ and $\psi_-$ the induced characters $(\theta_+)^{\norma_G(R)}$ and $(\theta_-)^{\norma_G(R)}$, respectively. Then $\varphi \in \{ \psi_+, \psi_- \}$, and we have
$$(\psi_+)^a = ((\theta_+)^{\norma_G(R)})^a = ((\theta_+)^a)^{\norma_G(R)} = (\theta_+)^{\norma_G(R)} = \psi_+,$$
and analogously, $(\psi_-)^a = \psi_-$, so $(R, \varphi)$ is left invariant by $a$ as claimed.
\end{proof}

We next consider $3$-blocks of abelian defect.

\begin{prop}
\label{prop:3D4_3noncyclicabelianblocks}
Let $B$ be a $3$-block of $G$ with a non-cyclic abelian defect group $D$. Then the following statements hold:
\begin{enumerate}[(i)]
\item The centralizer $\cent_G(D) =: T$ is a maximal torus of $G$ of type $T_\varepsilon$ or $T_{2, \varepsilon}$, $D = \mathcal{O}_3(T)$ and $\norma_G(D) = \norma_G(T)$.
\item Consider $\Irr(T)$ as an abelian group and fix an isomorphism $\hat{\phantom{s}} \colon T \longrightarrow \Irr(T)$. Up to $G$-conjugation there exists a unique $3'$-element $s \in T$ and a $3$-block $b$ of $T = \cent_G(D)$ with $b^G = B$ such that $\theta := \hat{s} \in \Irr(T)$ is the canonical character of $b$.
\item For $\theta$ and $s$ as in (ii) we have $\norma_G(T)_\theta / T \cong \cent_{W(T)}(s)$, where $W(T) := \norma_G(T) / T$.
\end{enumerate}
\end{prop}

\begin{proof}
Part (i) follows from \cite[(1A)]{AnWeights3D4} since being a defect group of $B$ the $3$-group $D$ is radical in $G$ (e.g.~\cite[Cor.~4.18]{navarro}). Statements (ii) and (iii) are part of \cite[Prop.~5.8]{DeriziotisMichler}.
\end{proof}

\begin{prop}
\label{prop:3D4_3BlockAbelianWeight}
Let $B$ be a $3$-block of $G$ with defect group $\mathcal{O}_3(T)$ for a maximal torus $T$ of $G$ of type $T_{2, \varepsilon}$. Then $\Aut(G)_B$ acts trivially on $\mathcal{W}(B)$.
\end{prop}

\begin{proof}
Suppose that $(R, \varphi)$ is a $B$-weight. In consequence of Lemma~\ref{lem:AbelianDefectWeight} we may assume that $R = \mathcal{O}_3(T)$ with $\norma_G(R) = \norma_G(T)$ by Proposition~\ref{prop:3D4_3noncyclicabelianblocks}(i).
Let $s \in T$, $\theta = \hat{s}$ and $b \in \Block_3(T)$ be as in Proposition~\ref{prop:3D4_3noncyclicabelianblocks}(ii), so $\norma_G(T)_\theta / T \cong \cent_{W(T)}(s)$ for $W(T) = \norma_G(T) / T$. Since $B$ is non-principal, we have $s \neq 1$. Then by \cite[Lemma~3.4]{DeriziotisMichler} up to isomorphism it holds that $\cent_{W(T)}(s) \in \{ \{1\}, C_3 \}$.
We now conclude as in the proof of Proposition~\ref{prop:3D4_2BlockAbelianWeight}.
\end{proof}

In the following we concentrate on the case of $3$-blocks of $G$ that have a defect group $\mathcal{O}_3(T)$ for a maximal torus $T$ of $G$ of type $T_\varepsilon$ and reduce this to a question on $3$-weights of $G_2(q)$.
We let $\widetilde{G} = G^{F_q} \cong G_2(q)$ be a maximal subgroup of $G$ as in Proposition~\ref{prop:G2Subgroup3D4} and we fix a maximal torus $T$ of $G$ of type $T_\varepsilon$ such that $\widetilde{T} := T \cap \widetilde{G}$ is a maximal torus of $\widetilde{G}$ isomorphic to $C_{q-\varepsilon} \times C_{q-\varepsilon}$. As observed earlier, we have
$$\norma_G(T)/T \cong \norma_{\widetilde{G}}(\widetilde{T})/ \widetilde{T} \cong D_{12},$$ 
such that by Proposition~\ref{prop:3D4_MaximalTorusRestrG2} it holds that $\norma_G(T) = \norma_{\widetilde{G}}(\widetilde{T}) T$.

Furthermore, we define $R:= \mathcal{O}_3(T)$ and $\widetilde{R}:= \mathcal{O}_3(\widetilde{T})$.
Now let $B$ be a $3$-block of $G$ with defect group $R$.  Since $B$ has abelian defect, it follows from Lemma~\ref{lem:AbelianDefectWeight} that a $3$-subgroup $Q\leqslant G$ is $G$-conjugate to $R$ whenever there exists a $B$-weight with first component $Q$.

If $(R, \varphi)$ is a $B$-weight of $G$ and $\theta \in \Irr(T)$ is an irreducible constituent of $\varphi_{|T}$, then $B = \block(\theta)^G$ following Construction~\ref{constr:B-weights}. 

Let us now define \smash{$\widetilde{B} := \block(\theta_{|\widetilde{T}})^{\widetilde{G}}$.} Our objective is to relate the $B$-weight $(R, \varphi)$ of $G$ to a (set of) $\widetilde{B}$-weight(s) of $\widetilde{G}$. To this end we examine the restriction of the weight character $\varphi$ to $\norma_{\widetilde{G}}(\widetilde{T})$. In this connection the irreducible characters of $T$ will play a key role, whence we fix the following parametrization for $\Irr(T)$:

\begin{nota} 
Identify $C_{q^3-\varepsilon}$ and $C_{q - \varepsilon}$ with the unique subgroups of the cyclic group \smash{$\mathbb{F}_{q^6}^\times$} of orders $q^3-\varepsilon$ and $q-\varepsilon$, respectively, and fix an element $z \in \mathbb{F}_{q^6}^\times$ of order $q^3 - \varepsilon$. With the linear map $\theta_0$ defined by
$\theta_0(z) = \exp(2 \pi \mathfrak{i}/(q^3-\varepsilon))$
we obtain
\begin{align*}
\Irr(C_{q^3 - \varepsilon}) = \{ \theta_0^i \mid 0 \leqslant i < q^3 - \varepsilon \}.
\end{align*}
By writing $\theta_0$ instead of ${\theta_0}_{|C_{q-\varepsilon}}$ we also have
$$
\Irr(C_{q - \varepsilon}) = \{ \theta_0^{i} \mid 0 \leqslant i < q - \varepsilon \}
$$
with $\theta_0^{i} \neq \theta_0^{j}$ for $0 \leqslant i, j < q - \varepsilon$ if $i \neq j$.
This yields a parametrization for $\Irr(T)$ given by
$$
\Irr(T) = \{ \theta_0^i \times \theta_0^{j} \mid 0 \leqslant i < q^3 - \varepsilon,\; 0 \leqslant j < q - \varepsilon \},
$$
where we have
$$
(\theta_0^i \times \theta_0^j)(h(t_1, t_2, t_1^{\varepsilon q}, t_1^{q^2})) = \theta_0^i(t_1) \theta_0^{j}(t_2)
$$
for an element $h(t_1, t_2, t_1^{\varepsilon q}, t_1^{q^2})$ with $t_1, t_2 \in \mathbb{F}^\times$ such that $t_1^{q^3-\varepsilon}=1 $ and $t_2^{q - \varepsilon} = 1$.
\end{nota}

The following series of lemmata will allow us to handle all situations which may occur when restricting a character $\varphi$ coming from a $B$-weight $(R, \varphi)$ of $G$ to $\norma_{\widetilde{G}}(\widetilde{T})$.

\begin{lem}
\label{lem:3D4_TrivialRestrOrder3}
Suppose that $\theta \in \Irr(T)$. Then $\widetilde{T}$ is contained in the kernel of $\theta$ if and only if $|\norma_G(T)_{\theta} \colon T\,|$ is divisible by $3$.
\end{lem}

\begin{proof}
We consider the element $\bar{\omega} \in W(T)= \norma_G(T)/T$ corresponding to the reflection $\omega_{r_1}  \omega_{r_3}  \omega_{r_4} \omega_{r_2}$. Such an element exists in $W(T)$ as a result of Lemma~\ref{lem:3D4_GeneratorsFixWeylGroup}, and it has order $6$ in $W(T)$. We set $\omega := \bar{\omega}^2$. By application of Lemma~\ref{lem:3D4_WeylGroupActionOnTorus} it follows that $\widetilde{T} \subseteq \ker(\theta)$ if and only if $\omega$ leaves $\theta$ invariant, which proves the claim since $\omega$ and $\omega^{-1}$ are the only elements of $W(T) \cong D_{12}$ of order 3. 
\end{proof}

\begin{lem}
\label{lem:3D4_ThreeCases}
Let $B$ be a $3$-block of $G$ with defect group $\mathcal{O}_3(T)$ and suppose that $\theta \in \Irr(T)$ is the canonical character of a $3$-block $b$ of $T$ such that $b^G = B$. Moreover, set $\widetilde{\theta} := \theta_{|\widetilde{T}} \in \Irr(\widetilde{T})$. Then (at least) one of the following situations occurs:
\begin{enumerate}[(i)]
\item $\norma_G(T)_\theta = T$,
\item $\norma_G(T)_\theta = \norma_G(T)_{\widetilde{\theta}}$,
\item $\norma_G(T)_\theta / T \cong C_2$ and 
$|\norma_G(T)_{\widetilde{\theta}}\, \colon \norma_G(T)_{\theta}\,| = 2$.
\end{enumerate}
\end{lem}

\begin{proof}
Suppose that neither (i) nor (ii) hold. We show that (iii) holds in this case.
The $3$-block $b$ has defect group $\mathcal{O}_3(T)$, whence it follows from \cite[Thm.~9.22]{navarro} that $3$ does not divide $|\norma_G(T)_\theta \colon T\,|$. In particular, up to isomorphism we have
$$\norma_G(T)_\theta / T \in \{ C_2, C_2 \times C_2 \}$$
since $\norma_G(T)/T \cong D_{12}$.
The group $\norma_G(T) /T$ contains exactly two elements of order $3$, which are given by $\omega$ and $\omega^{-1}$, where $\omega$ is as in the proof of Lemma~\ref{lem:3D4_TrivialRestrOrder3}. We prove that $\omega \not\in \norma_G(T)_{\widetilde{\theta}}$. To this end we assume the contrary, i.e.,~$\omega \in \norma_G(T)_{\widetilde{\theta}}$. Write $\theta = \theta_0^i \times \theta_0^j$ for suitable parameters $i$ and $j$. By Lemma~\ref{lem:3D4_WeylGroupActionOnTorus} it follows that
$$(\theta_0^i \times \theta_0^j)_{|\widetilde{T}} = (\theta_0^{i +j(q^2 + \varepsilon q +1)} \times \theta_0^{-i -2j})_{|\widetilde{T}},$$
or, equivalently, $i \equiv i +j(q^2 + \varepsilon q +1) \bmod q- \varepsilon$ and $j \equiv -i -2 j  \bmod q- \varepsilon$.
These conditions are fulfilled if and only if $3j \equiv 0 \bmod q-\varepsilon$ and $i \equiv 0 \bmod q-\varepsilon$. But $\theta$ is the canonical character of the $3$-block $b$, whence in particular we have $(q-\varepsilon)_3 \mid j$, so in fact $j \equiv 0 \bmod q-\varepsilon$. Hence, $\widetilde{T} \subseteq \ker(\theta)$ and Lemma~\ref{lem:3D4_TrivialRestrOrder3} implies that $3$ divides $|\norma_G(T)_\theta \colon T\,|$, a contradiction.
Thus, $|\norma_G(T)_{\widetilde{\theta}} \colon T\,|$ divides $4$.
\end{proof}

We go through the three cases specified in Lemma~\ref{lem:3D4_ThreeCases}. Case (i) may be treated easily:

\begin{prop}
\label{prop:3D4_3_B_weights_T_Ntheta}
Let $B$ be a $3$-block of $G$ with defect group $\mathcal{O}_3(T)$ and $\theta \in \Irr(T)$ be the canonical character of a $3$-block $b$ of $T$ with $b^G = B$ and $\norma_G(T)_\theta = T$. Then $|\mathcal{W}(B)| = 1$. In particular, $\Aut(G)_B$ acts trivially on $\mathcal{W}(B)$.
\end{prop}

\begin{proof}
This follows immediately from Construction~\ref{constr:B-weights} and Lemma~\ref{lem:AbelianDefectWeight}, which imply that up to $G$-conjugation for any $B$-weight $(Q, \varphi)$ we have $Q = \mathcal{O}_3(T)$ and $\varphi = \theta^{\norma_G(T)}$.
\end{proof}

We now turn towards cases (ii) and (iii) of Lemma~\ref{lem:3D4_ThreeCases}. Application of Proposition~\ref{prop:3D4_MaximalTorusRestrG2} easily yields the following:

\begin{lem}
\label{lem:3D4_SameNormalizerInG2}
Let $B$ be a $3$-block of $G$ with defect group $\mathcal{O}_3(T)$ and $\theta \in \Irr(T)$ be the canonical character of a $3$-block $b$ of $T$ with $b^G = B$.
Set $\widetilde{\theta} := \theta_{|\widetilde{T}} \in \Irr(\widetilde{T})$.
\begin{enumerate}[(i)]
\item If $\norma_G(T)_\theta = \norma_G(T)_{\widetilde{\theta}}$, then we also have $\norma_{\widetilde{G}}(\widetilde{T})_\theta = \norma_{\widetilde{G}}(\widetilde{T})_{\widetilde{\theta}}$.
\item If $\norma_G(T)_\theta / T \cong C_2$ and 
$|\norma_G(T)_{\widetilde{\theta}}\, \colon \norma_G(T)_{\theta}| = 2$, then we also have $\norma_{\widetilde{G}}(\widetilde{T})_\theta / \widetilde{T} \cong C_2$ and 
$|\norma_{\widetilde{G}}(\widetilde{T})_{\widetilde{\theta}}\, \colon \norma_{\widetilde{G}}(\widetilde{T})_{\theta}| = 2$.
\end{enumerate}
\end{lem}

\begin{lem}
\label{lem:3D4_ThetaExtendsToStabilizer}
Let $B$ be a $3$-block of $G$ with defect group $\mathcal{O}_3(T)$ and let $\theta \in \Irr(T)$ be the canonical character of a $3$-block $b \in \Block_3(T)$ with $b^G = B$. Then $\theta$ extends to $\norma_G(T)_\theta$. In particular, $\Irr(\norma_G(T)_\theta \mid \theta)$ only consists of extensions of $\theta$.
\end{lem}

\begin{proof}
Clearly, the stabilizer of $b$ in $\norma_G(T)$ is given by $\norma_G(T)_\theta$. The $3$-block $b$ has defect group $\mathcal{O}_3(T) \in \Syl_3(T)$, and since by assumption $b^G = B$ has defect group $\mathcal{O}_3(T)$ as well, for the stabilizer $\norma_G(T)_\theta$ of $b$ in $\norma_G(T)$ we have $|\norma_G(T)_\theta \colon T\,|_3 = 1$ in consequence of \cite[Thm.~9.22]{navarro}. Hence, $\norma_G(T)_\theta/T$ is isomorphic to one of the $3'$-subgroups $$\text{$\{ 1 \}$, $C_2$ and $C_2 \times C_2$}$$ of $D_{12}$. If $\norma_G(T)_\theta/T$ is cyclic, then $\theta$ extends to $\norma_G(T)_\theta$ by \cite[Problem~6.17]{Isaacs}. Thus, we suppose that $\norma_G(T)_\theta/T \cong C_2 \times C_2$:

 By Proposition~\ref{prop:3D4_3noncyclicabelianblocks} the canonical character $\theta$ of $b$ corresponds via an isomorphism between $\Irr(T)$ and $T$ to a semisimple $3'$-element $s \in T$ and we have $\norma_G(T)_\theta / T \cong \cent_{W(T)}(s)$.
 Since $\theta$ is stabilized by the non-trivial element of the center of $\norma_G(T)/T$, which acts on $T$ and $\Irr(T)$ by inversion, it follows that $\theta = \theta^{-1} \neq 1_T$.
According to Lemma~\ref{lem:3D4_ThreeCases} we have $$\norma_G(T)_\theta = \norma_G(T)_{\widetilde{\theta}},$$ and from Lemma~\ref{lem:3D4_TrivialRestrOrder3} we know that $\widetilde{\theta}$ is not the trivial character of $\widetilde{T}$, so $\widetilde{\theta}$ is of order $2$ as well. Then by Remark~\ref{rmk:G2_CanCharOrder2Extend} there exist $n_1, n_2 \in \norma_{\widetilde{G}}(\widetilde{T})$ with $[n_1, n_2] =1$ and $n_1^2, n_2^2 \in \widetilde{T}$ such that $\norma_{\widetilde{G}}(\widetilde{T})_{\widetilde{\theta}} = \langle \widetilde{T}, n_1, n_2 \rangle$. But then
$$\norma_G(T)_\theta = \norma_G(T)_{\widetilde{\theta}} = T \norma_{\widetilde{G}}(\widetilde{T})_{\widetilde{\theta}} = \langle T, n_1, n_2 \rangle$$
by Proposition~\ref{prop:3D4_MaximalTorusRestrG2}.
As in Remark~\ref{rmk:G2_CanCharOrder2Extend} it follows that $\theta$ extends to $\norma_G(T)_\theta$.

Now since $\norma_G(T)_\theta/T$ is abelian in any case, it follows by Gallagher \cite[Cor.~6.17]{Isaacs} that $\Irr(\norma_G(T)_\theta \mid \theta)$ only consists of extensions of $\theta$, which concludes the proof.
\end{proof}

\begin{lem}
\label{lem:3D4_ExtensionThetaNormaLarger}
Let $B$ be a $3$-block of $G$ with defect group $\mathcal{O}_3(T)$ and suppose that $\theta \in \Irr(T)$ is the canonical character of a $3$-block $b$ of $T$ with $b^G = B$.
Set $\widetilde{\theta} := \theta_{|\widetilde{T}} \in \Irr(\widetilde{T})$ and assume that $\norma_G(T)_\theta / T \cong C_2$ and 
$|\norma_G(T)_{\widetilde{\theta}}\, \colon \norma_G(T)_{\theta}| = 2$. Then
\begin{enumerate}[(i)]
\item $\Irr(\norma_{\widetilde{G}}(\widetilde{T})_{\widetilde{\theta}} \mid \widetilde{\theta})$ only consists of extensions of $\widetilde{\theta}$, and
\item any element of $\Irr(\norma_{\widetilde{G}}(\widetilde{T})_{{\theta}} \mid \widetilde{\theta})$ can be obtained as the restriction of a character in $\Irr(\norma_{\widetilde{G}}(\widetilde{T})_{\widetilde{\theta}} \mid \widetilde{\theta})$ to $\norma_{\widetilde{G}}(\widetilde{T})_{{\theta}}$.
\end{enumerate}
In particular, the characters in $\Irr(\norma_{\widetilde{G}}(\widetilde{T})_{{\theta}} \mid \widetilde{\theta})$ are invariant under $\norma_{\widetilde{G}}(\widetilde{T})_{\widetilde{\theta}}$-conjugation.
\end{lem}

\begin{proof}
This follows from Proposition~\ref{prop:3D4_MaximalTorusRestrG2}, Remark~\ref{rmk:G2_CanCharOrder2Extend} and Gallagher's theorem.
\end{proof}

\begin{lem}
\label{lem:3D4_ConstAboveTheta}
Let $\varphi \in \Irr(\norma_G(T))$ and $\theta \in \Irr(T)$ be an irreducible constituent of $\varphi_{|T}$. If $\psi \in \Irr(\norma_{\widetilde{G}}(\widetilde{T}))$ is an irreducible constituent of $\varphi_{|\norma_{\widetilde{G}}(\widetilde{T})}$, then $\psi$ lies above \smash{$\widetilde{\theta} := \theta_{|\widetilde{T}} \in \Irr(\widetilde{T})$.}

In particular, if $\mathcal{O}_3(T) \subseteq \ker(\theta)$, then $\mathcal{O}_3(\widetilde{T})$ is contained in the kernel of $\psi$.
\end{lem}

\begin{proof}
Let $r, e_i>0$ and $\psi_i \in \Irr(\norma_{\widetilde{G}}(\widetilde{T}))$ for $i \in \{ 1, \ldots, r \}$ with $\psi_i \neq \psi_j$ if $i \neq j$ be such that
$$\varphi_{|\norma_{\widetilde{G}}(\widetilde{T})} = \sum\limits_{i=1}^{r} e_i \psi_i.$$
By Clifford theory \cite[Thm.~6.5]{Isaacs} we may write $$\varphi_{|T} = f \cdot \sum\limits_{j = 1}^t \theta_j$$ for some  $f \in \mathbb{N}_{>0}$, $t = |\norma_G(T)\, \colon \norma_G(T)_\theta\,|$ and $\theta_1 = \theta, \theta_2, \ldots, \theta_t$ the $\norma_G(T)$-conjugates of $\theta$. Now let $i \in \{ 1, \ldots, r \}$  be such that $\psi = \psi_i$. Then $\psi_{|\widetilde{T}}$ is a summand of $\varphi_{|\widetilde{T}} = (\varphi_{|T})_{|\widetilde{T}}$, so ${\theta_j}_{|\widetilde{T}}$ is a constituent of $\psi_{|\widetilde{T}}$ for some $j$. Since $$\norma_G(T) = \norma_{\widetilde{G}}(\widetilde{T})T$$
by Proposition~\ref{prop:3D4_MaximalTorusRestrG2}, the character $\theta_j$ is in fact conjugate to $\theta$ via an element of $\norma_{\widetilde{G}}(\widetilde{T})$, so in particular it follows that $(\theta_j)_{|\widetilde{T}}$ and $\widetilde{\theta} = \theta_{|\widetilde{T}}$ are \smash{$\norma_{\widetilde{G}}(\widetilde{T})$-conjugate.} Hence, by Clifford theory the character $\widetilde{\theta}$ is a constituent of $\psi_{|\widetilde{T}}$ as claimed.

If now $\mathcal{O}_3(T) \subseteq \ker(\theta)$, then also $\mathcal{O}_3(\widetilde{T}) \subseteq \ker(\widetilde{\theta})$, and since $\mathcal{O}_3(\widetilde{T})$ is normal in $\norma_{\widetilde{G}}(\widetilde{T})$, it lies in the kernel of each $\norma_{\widetilde{G}}(\widetilde{T})$-conjugate of $\widetilde{\theta}$, which proves the claim.
\end{proof}

\begin{prop}
\label{prop:3D4_ReductionWeightCharIrred}
Let $B$ be a $3$-block of $G$ with defect group $R=\mathcal{O}_3(T)$ and suppose that $(R, \varphi)$ is a $B$-weight. Moreover, denote by $\theta \in \Irr(T)$ an irreducible constituent of $\varphi_{|T}$ and define $\widetilde{\theta} := \theta_{|\widetilde{T}} \in \Irr(\widetilde{T})$ and $\widetilde{\varphi}:= \varphi_{|\norma_{\widetilde{G}}(\widetilde{T})}$.
\begin{enumerate}[(i)]
\item If $\norma_G(T)_\theta = \norma_G(T)_{\widetilde{\theta}}$, then $\widetilde{\varphi}$ is irreducible.
\item If $\norma_G(T)_\theta/ T \cong C_2$ and $|\norma_G(T)_{\widetilde{\theta}}\, \colon \norma_G(T)_\theta\,| = 2$, then \smash{$\widetilde{\varphi}$} is the sum of two distinct irreducible characters of $\norma_{\widetilde{G}}(\widetilde{T})$ of the same degree.
\end{enumerate}
In particular, if $\chi$ is an irreducible constituent of \smash{$\widetilde{\varphi}$}, then $(\widetilde{R}, \chi)$ is a $\widetilde{B}$-weight for $\widetilde{G}$, where $\widetilde{B} = \block(\widetilde{\theta})^{\widetilde{G}}$ with defect group $\widetilde{R} = \mathcal{O}_3(\widetilde{T})$.
\end{prop}

\begin{proof}
Let $r, e_i >0$ and $\psi_i \in \Irr(\norma_{\widetilde{G}}(\widetilde{T}))$ for $i \in \{ 1, \ldots, r \}$ with $\psi_i \neq \psi_j$ if $i \neq j$ be such that
$$\widetilde{\varphi} = \sum\limits_{i=1}^{r} e_i \psi_i.$$
By Construction~\ref{constr:B-weights} it holds that $\varphi = \psi^{\norma_G(T)}$ for some $\psi \in \Irr(\norma_G(T)_\theta \mid \theta)$ satisfying $\psi(1)_3 = |\norma_G(T)_\theta \colon T\,|_3$, and the $3$-block $\block(\theta)$ with $\block(\theta)^G = B$ has canonical character $\theta$. By Lemma~\ref{lem:3D4_ThetaExtendsToStabilizer} the set $\Irr(\norma_G(T)_\theta \mid \theta)$ only consists of extensions of $\theta$. In particular, $\psi(1)=1$, and thus 
$$\varphi(1) = |\norma_G(T) \colon \norma_G(T)_\theta\,|.$$
Now consider the restriction \smash{${\psi_i}_{|\widetilde{T}}$} for $i \in \{ 1, \ldots, r \}$. By Lemma~\ref{lem:3D4_ConstAboveTheta} all characters $\psi_i$ lie above $\widetilde{\theta}$, so by Clifford theory \cite[Thm.~6.5]{Isaacs} there exist $f_i >0$ such that
$${\psi_i}_{|\widetilde{T}} = f_i \sum\limits_{j=1}^{t} \widetilde{\theta}_j,$$ where $t = |\norma_{\widetilde{G}}(\widetilde{T}) \colon \norma_{\widetilde{G}}(\widetilde{T})_{\widetilde{\theta}}\,|$ and $\widetilde{\theta}_1 = \widetilde{\theta}, \ldots, \widetilde{\theta}_t$ are the $\norma_{\widetilde{G}}(\widetilde{T})$-conjugates of $\widetilde{\theta}$. Using this decomposition the degree of $\varphi$ is given by
$$\varphi(1) = \widetilde{\varphi}(1) = \sum\limits_{i=1}^{r} e_i \psi_i(1) = \sum\limits_{i=1}^{r} e_i f_i \sum\limits_{j=1}^{t}  \widetilde{\theta}_j(1) = |\norma_{\widetilde{G}}(\widetilde{T}) \colon \norma_{\widetilde{G}}(\widetilde{T})_{\widetilde{\theta}}\,| \sum\limits_{i=1}^{r} e_i f_i.$$
If $\norma_G(T)_\theta = \norma_G(T)_{\widetilde{\theta}}$, then it follows easily that $|\norma_G(T) \colon \norma_G(T)_\theta\,| = |\norma_{\widetilde{G}}(\widetilde{T}) \colon \norma_{\widetilde{G}}(\widetilde{T})_{\widetilde{\theta}}\,|$, so comparison of the two formulae for $\varphi(1)$ yields $${\sum\limits_{i=1}^{r} e_i f_i = 1},$$
that is,~$r=1$ and $e_1 = f_1 = 1$. Hence, $\widetilde{\varphi} = \psi_1$ is irreducible as claimed in (i).

Suppose hence that $\norma_G(T)_\theta/ T \cong C_2$ and $|\norma_G(T)_{\widetilde{\theta}}\, \colon \norma_G(T)_\theta\,| = 2$ as in (ii). It follows that $$|\norma_G(T) \colon \norma_G(T)_\theta\,| = 2\,|\norma_{\widetilde{G}}(\widetilde{T}) \colon \norma_{\widetilde{G}}(\widetilde{T})_{\widetilde{\theta}}\,|,$$ so in this case we have $\smash{\sum\limits_{i=1}^{r} e_i f_i = 2}$, whence $\widetilde{\varphi}$ must be the sum of at most two irreducible constituents.

Let us go into more detail. As observed above, it holds that $\varphi = \psi^{\norma_G(T)}$ for some linear extension $\psi \in \Irr(\norma_G(T)_\theta \mid \theta)$ of $\theta$. By Proposition~\ref{prop:3D4_MaximalTorusRestrG2} we have $\norma_G(T) = \norma_G(T)_\theta \norma_{\widetilde{G}}(\widetilde{T})$, so Mackey's theorem (e.g.~\cite[Thm.~3.1.9]{NagaoTsushima}) implies that
$$\widetilde{\varphi} = \Res_{\norma_{\widetilde{G}}(\widetilde{T})}^{\norma_G(T)} (\Ind_{\norma_G(T)_\theta}^{\norma_G(T)}(\psi)) = \Ind_{\norma_{\widetilde{G}}(\widetilde{T})_\theta}^{\norma_{\widetilde{G}}(\widetilde{T})} (\Res_{\norma_{\widetilde{G}}(\widetilde{T})_\theta}^{\norma_G(T)_\theta} (\psi))$$
since $\norma_G(T)_\theta \cap \norma_{\widetilde{G}}(\widetilde{T}) = \norma_{\widetilde{G}}(\widetilde{T})_\theta$.
Moreover, since $\norma_G(T)_\theta \lvertneqq \norma_G(T)_{\widetilde{\theta}} \lvertneqq \norma_G(T)$, we have
$$\widetilde{\varphi} = \Ind_{\norma_{\widetilde{G}}(\widetilde{T})_{\widetilde{\theta}}}^{\norma_{\widetilde{G}}(\widetilde{T})} (\Ind_{\norma_{\widetilde{G}}(\widetilde{T})_\theta}^{\norma_{\widetilde{G}}(\widetilde{T})_{\widetilde{\theta}}} (\Res_{\norma_{\widetilde{G}}(\widetilde{T})_\theta}^{\norma_G(T)_\theta} (\psi))).$$
We examine the character $\Upsilon := \Ind_{\norma_{\widetilde{G}}(\widetilde{T})_\theta}^{\norma_{\widetilde{G}}(\widetilde{T})_{\widetilde{\theta}}} (\Res_{\norma_{\widetilde{G}}(\widetilde{T})_\theta}^{\norma_G(T)_\theta} (\psi))$. By definition it holds that
$$\Upsilon(x) = \frac{1}{|\norma_{\widetilde{G}}(\widetilde{T})_\theta\,|} \sum\limits_{n \in \norma_{\widetilde{G}}(\widetilde{T})_{\widetilde{\theta}}} \dot{\psi}(nxn^{-1})$$
for $x \in \norma_{\widetilde{G}}(\widetilde{T})_{\widetilde{\theta}}$, where we have $\dot{\psi}(h) = \psi(h)$ if $h \in \norma_{\widetilde{G}}(\widetilde{T})_\theta$ and $\dot{\psi}(h) = 0$ else. Now for $x \in \widetilde{T}$ it follows that $$\dot{\psi}(nxn^{-1}) = \widetilde{\theta}(nxn^{-1}) = \widetilde{\theta}^n(x) = \widetilde{\theta}(x)$$
for all $n \in \norma_{\widetilde{G}}(\widetilde{T})_{\widetilde{\theta}}$, so
$$\Upsilon_{|\widetilde{T}} = \frac{|\norma_{\widetilde{G}}(\widetilde{T})_{\widetilde{\theta}}\,|}{|\norma_{\widetilde{G}}(\widetilde{T})_{{\theta}}|}\, \widetilde{\theta} = 2 \widetilde{\theta}$$
by Lemma~\ref{lem:3D4_SameNormalizerInG2}(ii).
Thus, any constituent of $\Upsilon$ lies above $\widetilde{\theta}$, so in particular any such constituent is linear by Lemma~\ref{lem:3D4_ExtensionThetaNormaLarger}(i). Since $\Upsilon$ has degree $2$, it must be the sum of two linear constituents lying above $\widetilde{\theta}$, say
$$\Upsilon = \chi_1 + \chi_2$$
for $\chi_1, \chi_2 \in \Irr(\norma_{\widetilde{G}}(\widetilde{T})_{\widetilde{\theta}} \mid \widetilde{\theta})$.
Since $\norma_{\widetilde{G}}(\widetilde{T})_\theta$ has index $2$ in $\norma_{\widetilde{G}}(\widetilde{T})_{\widetilde{\theta}}$, it is normal in $\norma_{\widetilde{G}}(\widetilde{T})_{\widetilde{\theta}}$. In particular, for $x, n \in \norma_{\widetilde{G}}(\widetilde{T})_{\widetilde{\theta}}$ we have $$\text{$nxn^{-1} \in \norma_{\widetilde{G}}(\widetilde{T})_{{\theta}}$ \quad if and only if \quad$x \in \norma_{\widetilde{G}}(\widetilde{T})_{{\theta}}$}.$$
Thus, $\Upsilon(x) = 0$ whenever $x \in \norma_{\widetilde{G}}(\widetilde{T})_{\widetilde{\theta}} \setminus \norma_{\widetilde{G}}(\widetilde{T})_{{\theta}}$. But $\Upsilon$ is the sum of the two linear characters $\chi_1$ and $\chi_2$, so $\Upsilon(x)=0$ is only possible if $\chi_1 \neq \chi_2$. We conclude that
$$\widetilde{\varphi} = \Upsilon^{\norma_{\widetilde{G}}(\widetilde{T})} = (\chi_1 + \chi_2)^{\norma_{\widetilde{G}}(\widetilde{T})} = (\chi_1)^{\norma_{\widetilde{G}}(\widetilde{T})} + (\chi_2)^{\norma_{\widetilde{G}}(\widetilde{T})}.$$
Now $\chi_1, \chi_2 \in \Irr(\norma_{\widetilde{G}}(\widetilde{T})_{\widetilde{\theta}} \mid \widetilde{\theta})$ are distinct, and following Clifford correspondence \cite[Thm.~6.11]{Isaacs} the map
\begin{align*}
\Irr(\norma_{\widetilde{G}}(\widetilde{T})_{\widetilde{\theta}} \mid \widetilde{\theta}) \longrightarrow \Irr(\norma_{\widetilde{G}}(\widetilde{T}) \mid \widetilde{\theta}),\quad
\vartheta \longmapsto \Ind_{\norma_{\widetilde{G}}(\widetilde{T})_{\widetilde{\theta}}}^{\norma_{\widetilde{G}}(\widetilde{T})}(\vartheta),
\end{align*}
is a bijection. Thus, also $(\chi_1)^{\norma_{\widetilde{G}}(\widetilde{T})}$ and $ (\chi_2)^{\norma_{\widetilde{G}}(\widetilde{T})}$ must be distinct, irreducible, and clearly both of the same degree, as claimed.

For the last statement, suppose that $\chi$ is a constituent of $\widetilde{\varphi}$. We have shown that
$$\chi(1)_3 = \varphi(1)_3 = |\norma_G(R)/R\,|_3,$$
and since $$|\norma_G(R)/R\,|_3 = 3 = |\norma_{\widetilde{G}}(\widetilde{R})/\widetilde{R}\,|_3$$
and $\widetilde{R} \subseteq \ker(\chi)$ by Lemma~\ref{lem:3D4_ConstAboveTheta}, it follows that $(\widetilde{R}, \chi)$ is a $3$-weight for $\widetilde{G}$. Moreover, since by Lemma~\ref{lem:3D4_ConstAboveTheta} the character $\chi$ lies above $\widetilde{\theta}$, we conclude that the $3$-weight $(\widetilde{R}, \widetilde{\varphi})$ belongs to $\widetilde{B} = \block(\widetilde{\theta})^{\widetilde{G}}$ in consequence of Construction~\ref{constr:B-weights}.
We observed in the proof of Lemma~\ref{lem:3D4_ThreeCases} that $|\norma_G(T)_{\widetilde{\theta}} \colon T\,|$, and hence also $|\norma_{\widetilde{G}}(\widetilde{T})_{\widetilde{\theta}} \colon \widetilde{T}\,|$, is a divisor of $4$. Since $\block(\widetilde{\theta})$ has defect group $\widetilde{R} = \mathcal{O}_3(T)$, \cite[Thm.~9.22]{navarro} implies that also \smash{$\widetilde{B} = \block(\widetilde{\theta})^{\widetilde{G}}$} must have defect group $\widetilde{R}$. This finishes the proof.
\end{proof}

\begin{prop}
\label{prop:3D4_RestrictionWeightCharEqualIff}
Let $B$ be a $3$-block of $G$ with defect group $R=\mathcal{O}_3(T)$ and suppose that $(R, \varphi_1)$ and $(R, \varphi_2)$ are $B$-weights. Then the characters \smash{${\varphi_1}_{|\norma_{\widetilde{G}}(\widetilde{T})}$ and ${\varphi_2}_{|\norma_{\widetilde{G}}(\widetilde{T})}$} have a common irreducible constituent if and only if $\varphi_1 = \varphi_2$.
\end{prop}

\begin{proof}
Let $\theta_1, \theta_2 \in \Irr(T)$ be irreducible constituents of \smash{${\varphi_1}_{|T}$ and ${\varphi_2}_{|T}$}, respectively. Following Construction~\ref{constr:B-weights} we have $\block(\theta_1)^G = \block(\theta_2)^G = B$, whence by the extended first main theorem of Brauer \cite[Thm.~9.7]{navarro} the $3$-blocks $\block(\theta_1)$ and $\block(\theta_2)$ are conjugate under $\norma_G(R) = \norma_G(T)$, so in particular their canonical characters $\theta_1$ and $\theta_2$ are conjugate under $\norma_G(T)$. According to Clifford \cite[Thm.~6.5]{Isaacs} all $\norma_G(T)$-conjugates of $\theta_1$ occur with the same multiplicity as irreducible constituents of ${\varphi_1}_{|T}$, so without loss of generality we may assume that $\theta := \theta_1 = \theta_2$.

In consequence of Proposition~\ref{prop:3D4_3_B_weights_T_Ntheta} the claim is trivially true if $\norma_G(T)_\theta = T$, so let us assume from now that $\norma_G(T)_\theta \gvertneqq T$. In particular, from Lemma~\ref{lem:3D4_ThreeCases} we deduce that for $\widetilde{\theta}:= \theta_{|\widetilde{T}} \in \Irr(\widetilde{T})$ one of the following holds: 
\begin{enumerate}[(i)]
\item $\norma_G(T)_\theta = \norma_G(T)_{\widetilde{\theta}}$,
\item $\norma_G(T)_\theta / T \cong C_2$ and 
$|\norma_G(T)_{\widetilde{\theta}}\, \colon \norma_G(T)_{\theta}| = 2$.
\end{enumerate}
Following Construction~\ref{constr:B-weights} there exist $\psi_1, \psi_2 \in \Irr(\norma_G(T)_\theta \mid \theta)$ such that 
$$\text{$\varphi_1 = \Ind_{\norma_G(T)_\theta}^{\norma_G(T)}(\psi_1)$ \quad and \quad $\varphi_2 = \Ind_{\norma_G(T)_\theta}^{\norma_G(T)}(\psi_2)$.}$$
 Hence, if we set $\widetilde{\varphi_1} := {\varphi_1}_{|\norma_{\widetilde{G}}(\widetilde{T})}$ and $\widetilde{\varphi_2} := {\varphi_2}_{|\norma_{\widetilde{G}}(\widetilde{T})}$, then we have
$$\widetilde{\varphi_{i\phantom{i}}} = \Res_{\norma_{\widetilde{G}}(\widetilde{T})}	^{\norma_G(T)} \Ind_{\norma_G(T)_\theta}^{\norma_G(T)}(\psi_i)$$
 for $i = 1, 2$.
 Since \smash{$\norma_G(T) = \norma_{\widetilde{G}}(\widetilde{T})T$} by Proposition~\ref{prop:3D4_MaximalTorusRestrG2} and $T \subseteq \norma_G(T)_\theta$, we have $\norma_G(T) = \norma_G(T)_\theta \norma_{\widetilde{G}}(\widetilde{T})$. Thus, as in the proof of Proposition~\ref{prop:3D4_ReductionWeightCharIrred}, by Mackey's formula \cite[Thm.~3.1.9]{NagaoTsushima} it follows that
$$\widetilde{\varphi_{i\phantom{i}}} = \Ind_{\norma_{\widetilde{G}}(\widetilde{T})_\theta}^{\norma_{\widetilde{G}}(\widetilde{T})} \Res_{\norma_{\widetilde{G}}(\widetilde{T})_\theta}^{\norma_G(T)_\theta} (\psi_i)$$
for $i = 1, 2 $.
Let us suppose first that $\norma_G(T)_\theta = \norma_G(T)_{\widetilde{\theta}}$ as in (i) above. In particular, $\widetilde{\varphi_{1}}$ and $\widetilde{\varphi_{2}}$ are irreducible as shown in Proposition~\ref{prop:3D4_ReductionWeightCharIrred}. As in the proof of Proposition~\ref{prop:3D4_ReductionWeightCharIrred} we consider the bijection
\begin{align*}
\Irr(\norma_{\widetilde{G}}(\widetilde{T})_{\widetilde{\theta}} \mid \widetilde{\theta}) \longrightarrow \Irr(\norma_{\widetilde{G}}(\widetilde{T}) \mid \widetilde{\theta}),\quad
\vartheta \longmapsto \Ind_{\norma_{\widetilde{G}}(\widetilde{T})_{\widetilde{\theta}}}^{\norma_{\widetilde{G}}(\widetilde{T})}(\vartheta).
\end{align*}
Since $\norma_G(T)_\theta = \norma_G(T)_{\widetilde{\theta}}$, we also have $\norma_{\widetilde{G}}(\widetilde{T})_{{\theta}} = \norma_{\widetilde{G}}(\widetilde{T})_{\widetilde{\theta}}$ by Lemma~\ref{lem:3D4_SameNormalizerInG2}(i), whence the above bijection is the same as
\begin{align*}
\Irr(\norma_{\widetilde{G}}(\widetilde{T})_{{\theta}} \mid \widetilde{\theta}) \longrightarrow \Irr(\norma_{\widetilde{G}}(\widetilde{T}) \mid \widetilde{\theta}),\quad
\vartheta \longmapsto \Ind_{\norma_{\widetilde{G}}(\widetilde{T})_{{\theta}}}^{\norma_{\widetilde{G}}(\widetilde{T})}(\vartheta).
\end{align*}
By Lemma~\ref{lem:3D4_ThetaExtendsToStabilizer} the characters $\psi_1$ and $\psi_2$ are extensions of $\theta$. In particular, \smash{$\Res_{\norma_{\widetilde{G}}(\widetilde{T})_\theta}^{\norma_G(T)_\theta} (\psi_i)$} is irreducible for $i = 1, 2$ and lies above $\widetilde{\theta}$. Hence, it follows that
$$\text{$\widetilde{\varphi_1} = \widetilde{\varphi_2}$ \quad if and only if \quad $\Res_{\norma_{\widetilde{G}}(\widetilde{T})_\theta}^{\norma_G(T)_\theta} (\psi_1) = \Res_{\norma_{\widetilde{G}}(\widetilde{T})_\theta}^{\norma_G(T)_\theta} (\psi_2)$}.$$
Both $\psi_1$ and $\psi_2$ are extensions of $\theta$ to $\norma_G(T)_\theta$, so there exists $\eta \in \Irr(\norma_G(T)_\theta / T)$ such that $\psi_2 = \psi_1 \cdot \eta$. Now suppose that
$$\Res_{\norma_{\widetilde{G}}(\widetilde{T})_\theta}^{\norma_G(T)_\theta} (\psi_1) = \Res_{\norma_{\widetilde{G}}(\widetilde{T})_\theta}^{\norma_G(T)_\theta} (\psi_2)$$
and let $n \in \norma_G(T)_\theta$. Since $\norma_G(T)_\theta = T \norma_{\widetilde{G}}(\widetilde{T})_\theta$, there exist $t \in T$ and $m \in \norma_{\widetilde{G}}(\widetilde{T})_\theta$ such that $n = t m$.
It follows that
$$\theta(t) \psi_1(m) \eta(nT) = \psi_1(n) \eta(nT) = \psi_2(n) = \theta(t) \psi_2(m) = \theta(t) \psi_1(m),$$
and since $\psi_1$ and $\theta$ are linear, we have $\psi_1(m), \theta(t) \neq 0$, so $\eta(nT) = 1$, that is, $\eta = 1_{\norma_G(T)_\theta / T}$ in this case. Hence, $\psi_1$ and $\psi_2$ agree, and thus $\varphi_1 = \varphi_2$ as claimed.

Let us finally suppose that $\norma_G(T)_\theta / T \cong C_2$ and 
$|\norma_G(T)_{\widetilde{\theta}}\, \colon \norma_G(T)_{\theta}\,| = 2$ as in case (ii) above. In the proof of Proposition~\ref{prop:3D4_ReductionWeightCharIrred} we observed that there exist linear characters $\chi_{11}, \chi_{12}, \chi_{21}, \chi_{22} \in \Irr(\norma_{\widetilde{G}}(\widetilde{T})_{\widetilde{\theta}} \mid \widetilde{\theta})$ with $\chi_{11} \neq \chi_{12}$ and $\chi_{21} \neq \chi_{22}$ such that
$$\text{$\widetilde{\varphi_{i\phantom{i}}} = (\chi_{i1})^{\norma_{\widetilde{G}}(\widetilde{T})} + (\chi_{i2})^{\norma_{\widetilde{G}}(\widetilde{T})}$} $$
with $(\chi_{ij})^{\norma_{\widetilde{G}}(\widetilde{T})} = (\chi_{i'j'})^{\norma_{\widetilde{G}}(\widetilde{T})}$ if and only if $\chi_{ij} = \chi_{i'j'}$ for $i, i', j, j' \in \{1, 2\}$, and moreover
$$\text{$\Upsilon_i := \Ind_{\norma_{\widetilde{G}}(\widetilde{T})_\theta}^{\norma_{\widetilde{G}}(\widetilde{T})_{\widetilde{\theta}}} (\Res_{\norma_{\widetilde{G}}(\widetilde{T})_\theta}^{\norma_G(T)_\theta} (\psi_i)) = \chi_{i1} + \chi_{i2}$}$$
for $i=1, 2$. Now the restriction of $\psi_i$ to $\norma_{\widetilde{G}}(\widetilde{T})_\theta$ is an extension of $\widetilde{\theta}$ to $\norma_{\widetilde{G}}(\widetilde{T})_\theta$ for $i=1, 2$. Hence, by Lemma~\ref{lem:3D4_ExtensionThetaNormaLarger} it is invariant under the conjugation action of $\norma_{\widetilde{G}}(\widetilde{T})_{\widetilde{\theta}}$. Thus, for $x \in \norma_{\widetilde{G}}(\widetilde{T})_\theta$ we have
$$\Upsilon_i(x) = \frac{1}{|\norma_{\widetilde{G}}(\widetilde{T})_\theta\,|} \sum\limits_{n \in \norma_{\widetilde{G}}(\widetilde{T})_{\widetilde{\theta}}} \dot{\psi_i}(nxn^{-1}) = \frac{1}{|\norma_{\widetilde{G}}(\widetilde{T})_\theta\,|} \sum\limits_{n \in \norma_{\widetilde{G}}(\widetilde{T})_{\widetilde{\theta}}} \psi_i(x) = 2 \psi_i(x),$$
that is,
$${\Upsilon_i}_{|\norma_{\widetilde{G}}(\widetilde{T})_\theta} = 2 {\psi_i}_{|\norma_{\widetilde{G}}(\widetilde{T})_\theta},$$
and hence
$${\chi_{i1}}_{|\norma_{\widetilde{G}}(\widetilde{T})_\theta} = {\chi_{i2}}_{|\norma_{\widetilde{G}}(\widetilde{T})_\theta} = {\psi_i}_{|\norma_{\widetilde{G}}(\widetilde{T})_\theta}$$
for $i=1, 2$. Now suppose that $\widetilde{\varphi_1}$ and $\widetilde{\varphi_2}$ have a common irreducible constituent. Then there exist $j, j' \in \{ 1, 2 \}$ such that $(\chi_{1j})^{\norma_{\widetilde{G}}(\widetilde{T})} = (\chi_{2j'})^{\norma_{\widetilde{G}}(\widetilde{T})}$, and thus $\chi_{1j} = \chi_{2j'}$. By the observation above it follows that
$${\psi_1}_{|\norma_{\widetilde{G}}(\widetilde{T})_\theta} = {\chi_{1j}}_{|\norma_{\widetilde{G}}(\widetilde{T})_\theta} = {\chi_{2j'}}_{|\norma_{\widetilde{G}}(\widetilde{T})_\theta} = {\psi_2}_{|\norma_{\widetilde{G}}(\widetilde{T})_\theta}.$$
The same argumentation as for case (i) yields that $\psi_1 = \psi_2$, whence also $\varphi_1 = \varphi_2$.
\end{proof}

\begin{prop}
\label{prop:3D4_BijectionAbelian}
Let $B$ be a $3$-block of $G$ with defect group $R = \mathcal{O}_3(T)$ such that $\norma_G(T)_\theta \gvertneqq T$ if $\theta \in \Irr(T)$ is the canonical character of a root of $B$ and set
$$
\widetilde{\mathcal{W}}(B) := \left\{ \{ [(\widetilde{R}, \psi)]_{\widetilde{G}} \mid \psi\; \text{irreducible constituent of}\; \varphi_{|\norma_{\widetilde{G}}(\widetilde{T})} \} \ \Big| \ [(R, \varphi)]_G \in \mathcal{W}(B)  \right\}.
$$
Then the map 
\begin{align*}
\Lambda\colon  \mathcal{W}(B) \longrightarrow  \widetilde{\mathcal{W}}(B),
\;\, [(R, \varphi)]_G \longmapsto \{ [(\widetilde{R}, \psi)]_{\widetilde{G}} \mid \psi\; \text{irreducible constituent of}\; \varphi_{|\norma_{\widetilde{G}}(\widetilde{T})} \},
\end{align*}
is a bijection.
\end{prop}

\begin{proof}
$\Lambda$ is surjective by construction of $\widetilde{\mathcal{W}}(B)$ and injective by Proposition~\ref{prop:3D4_RestrictionWeightCharEqualIff}.
\end{proof}

Our next aim is to prove that $\Lambda$ is \smash{$\Aut(G)_{B, \widetilde{G}}$}-equivariant. For this we first note that $\Aut(G)_{B, \widetilde{G}}$ does indeed act on $\widetilde{\mathcal{W}}(B)$:

\begin{prop}
\label{prop:3D4_ActionOnRestrictedWeights}
Let $B$ be a $3$-block of $G$ with defect group $R = \mathcal{O}_3(T)$ such that $\norma_G(T)_\theta \gvertneqq T$ if $\theta \in \Irr(T)$ is the canonical character of a root of $B$.
Then the group $\Aut(G)_{B, \widetilde{G}}$ acts on $\widetilde{\mathcal{W}}(B)$.
\end{prop}

\begin{proof}
Let $[(R, \varphi)]_G \in \mathcal{W}(B)$ and denote by $\widetilde{\varphi}$ the restriction of $\varphi$ to $\norma_{\widetilde{G}}(\widetilde{T})$. Moreover, let $a \in \Aut(G)_{B, \widetilde{G}}$. Since $\widetilde{R}$ is the unique Sylow $3$-subgroup of the maximal torus $\widetilde{T}$ of $\widetilde{G}$, there exists an element $g \in \widetilde{G}$ such that $ag^{-1}$ stabilizes $\widetilde{R}$, so for every irreducible constituent $\psi$ of $\widetilde{\varphi}$ we have 
$$[(\widetilde{R}, \psi)]_{\widetilde{G}}^a = [(\widetilde{R}^a, \psi^a)]_{\widetilde{G}} = [(\widetilde{R}, \psi^{ag^{-1}})]_{\widetilde{G}}.$$ 
By Lemma~\ref{lem:3D4_CentralizerTTilde} and Proposition~\ref{prop:Defect3} we have $T = \cent_G(\widetilde{T}) = \cent_G(\cent_{\widetilde{G}}(\widetilde{R}))$. Hence, $ag^{-1}$ stabilizes $T$, and since $R$ is the unique Sylow 3-subgroup of $T$, also $R^{ag^{-1}} = R$.
In particular $(R, \varphi^{ag^{-1}})$ is a $B$-weight for $G$ with irreducible constituents of ${\varphi^{ag^{-1}}}_{|\norma_{\widetilde{G}}(\widetilde{T})}$ given by the set
$$\{\psi^{ag^{-1}} \mid \psi\; \text{irreducible constituent of}\;\widetilde{\varphi}\}.$$ Hence,
\begin{align*}
\{ [(\widetilde{R}, \psi)]_{\widetilde{G}} \mid \psi\; &\text{irreducible constituent of}\; \widetilde{\varphi} \}^a\\ &= \{ [(\widetilde{R}, \psi^{ag^{-1}})]_{\widetilde{G}} \mid \psi\; \text{irreducible constituent of}\; \widetilde{\varphi} \}\\ &= \{ [(\widetilde{R}, \chi)]_{\widetilde{G}} \mid \chi\; \text{irreducible constituent of}\; {\varphi^{ag^{-1}}}_{|\norma_{\widetilde{G}}(\widetilde{T})} \} \in \widetilde{\mathcal{W}}(B),
\end{align*}
so $\Aut(G)_{B, \widetilde{G}}$ acts on $\widetilde{\mathcal{W}}(B)$ as claimed.
\end{proof}

\begin{thm}
\label{th:3D4_LambdaEquivariant}
The bijection $\Lambda$ in Proposition~\ref{prop:3D4_BijectionAbelian} is $\Aut(G)_{B, \widetilde{G}}$-equivariant.
\end{thm}

\begin{proof}
For $a \in \Aut(G)_{B, \widetilde{G}}$ and $[(R, \varphi)]_G \in \mathcal{W}(B)$ let $g \in \widetilde{G}$ be such that $\widetilde{R}^a = \widetilde{R}^g$. Then also $R^a = R^g$, and as in the proof of Proposition~\ref{prop:3D4_ActionOnRestrictedWeights} we have
\begin{align*}
\Lambda([(R, \varphi)]_G)^a &= \{ [(\widetilde{R}, \chi)]_{\widetilde{G}} \mid \chi\; \text{irreducible constituent of}\; {\varphi^{ag^{-1}}}_{|\norma_{\widetilde{G}}(\widetilde{T})} \}\\
&= \Lambda([(R, \varphi^{ag^{-1}})]_G)\\
&= \Lambda([(R^a, \varphi^a)]_G)\\
&= \Lambda([(R, \varphi)]_G^a).
\end{align*}
Hence, $\Lambda$ is equivariant under the action of $\Aut(G)_{B, \widetilde{G}}$ as claimed.
\end{proof}

We can finally prove the following:

\begin{coro}
\label{coro:3D4_3Auto_AbelianStabLarger}
Let $B$ be a $3$-block of $G$ with defect group $\mathcal{O}_3(T)$ and $\norma_G(T)_\theta \gvertneqq T$ for the canonical character $\theta \in \Irr(T)$ of a root of $B$. Then $\Aut(G)_B$ acts trivially on $\mathcal{W}(B)$.
\end{coro}

\begin{proof}
Let $a \in \Aut(G)_B$. Since all maximal subgroups of $G$ isomorphic to $G_2(q)$ are $G$-conjugate, we may assume that $a \in \Aut(G)_{B, \widetilde{G}}$. Now let $[(R, \varphi)]_G \in \mathcal{W}(B)$. By Theorem~\ref{th:3D4_LambdaEquivariant} we have $$[(R, \varphi)]_G^a = \Lambda^{-1}(\Lambda([(R, \varphi)]_G)^a),$$
where $\Lambda$ denotes the bijection of Proposition~\ref{prop:3D4_BijectionAbelian}.
Since $3 \nmid q$, in consequence of the results obtained in Section~\ref{ssec:G2_Auto3Weights} every $\widetilde{G}$-conjugacy class of $3$-weights in $\widetilde{G}$ belonging to a $3$-block of $\widetilde{G}$ of non-cyclic defect is left invariant by any automorphism of $\widetilde{G}$. In particular,
$$\Lambda([(R, \varphi)]_G)^a = \Lambda([(R, \varphi)]_G)$$
as for any constituent $\chi$ of the restriction of $\varphi$ to $\norma_{\widetilde{G}}(\widetilde{T})$ the weight $(\widetilde{R}, \chi)$ belongs to a $3$-block of $\widetilde{G}$ with non-cyclic defect group given by $\widetilde{R}$ according to Proposition~\ref{prop:3D4_ReductionWeightCharIrred}.  Thus, $$[(R, \varphi)]_G^a = \Lambda^{-1}(\Lambda([(R, \varphi)]_G)) = [(R, \varphi)]_G,$$
which completes the proof.
\end{proof}

In summary, the following holds:

\begin{prop}
\label{prop:3D4_ExistancePartiBij}
Let $\ell \in \{2, 3\}$ and let $B$ be an $\ell$-block of $G$ of non-cyclic defect. Then the iBAW condition (cf.~Definition~\ref{defi:iBAWCblock}) holds for $B$.
\end{prop}

\begin{proof}
By \cite[Rmk.~1, (3B), (3G)]{AnWeights3D4} Alperin's conjecture is true for $B$, so we may choose a bijection $\Omega_B \colon \IBr(B) \longrightarrow \mathcal{W}(B)$. By Propositions~\ref{prop:3D4_Auto2BrauerChars}, \ref{prop:3D4_Auto2WeightsB1}, \ref{prop:3D4_Auto2WeightsNonAb}, \ref{prop:3D4_2BlockAbelianWeight}, \ref{prop:3D4_Auto3WeightsB1}, \ref{prop:3D4_Auto3WeightsNonAb},  \ref{prop:3D4_3BlockAbelianWeight} and~\ref{prop:3D4_3_B_weights_T_Ntheta} and Corollary~\ref{coro:3D4_3Auto_AbelianStabLarger} the group $\Aut(G)_B$ acts trivially on $\IBr(B)$ and $\mathcal{W}(B)$. Accordingly, $\Omega_B$ is $\Aut(G)_B$-equivariant, so Lemma~\ref{lem:SL3_BijectionConstruction} implies the claim.
\end{proof}

\begin{proof}[Proof of Theorem~B]
The simple group $G = \trial$ is its own universal covering group, and moreover the outer automorphism group of $G$ is cyclic by Proposition~\ref{prop:3D4_AutomorphismGroup}.

Let $\ell$ be a prime dividing $|G|$. We may assume that $\ell \neq p$,
so $\ell$ divides at least one of $\Phi_1(q)$, $\Phi_2(q)$, $\Phi_3(q)$, $\Phi_6(q)$ and $\Phi_{12}(q)$. It follows that if $\ell \geqslant 5$, then it divides exactly one of $\Phi_1(q)$, $\Phi_2(q)$, $\Phi_3(q)$, $\Phi_6(q)$ or $\Phi_{12}(q)$. Suppose that $5 \leqslant \ell \mid \Phi_{12}(q)$. Then a Sylow $\ell$-subgroup of $G$ is contained in a maximal torus of $G$ of type $T_3$, so the Sylow $\ell$-subgroups of $G$ are cyclic in this case (cf.~Table~\ref{tb:MaximalTorus3D4}). In particular, the iBAW condition holds for $G$ and $\ell$ by \cite[Thm.~1.1]{SpaethKoshiCyclic}.

Let now $5 \leqslant \ell \mid \Phi_1(q)\Phi_2(q)\Phi_3(q)\Phi_6(q)$. Then according to \cite[Lemma~5.2]{DeriziotisMichler} the $\ell$-blocks of $G$ have either cyclic or maximal defect. In the case of cyclic $\ell$-blocks the iBAW condition holds again by \cite[Thm.~1.1]{SpaethKoshiCyclic}, while it has been proven to hold for $3$-blocks of maximal defect by Cabanes--Späth in \cite[Cor.~7.6]{CabanesSpaeth}.

The case of non-cyclic $\ell$-blocks of $G$ for $\ell \in \{2, 3\}$ is covered by Proposition~\ref{prop:3D4_ExistancePartiBij}.
\end{proof}


\section*{Acknowledgement}
The author sincerely thanks Prof.~Dr.~Gunter Malle for supervising her dissertation and commenting on an earlier version of this paper.
Moreover, the author gratefully acknowledges financial support by ERC
  Advanced Grant 291512.


\end{document}